\documentclass[a4paper,10pt]{article}
\pdfoutput=1
\usepackage{fullpage}
\usepackage{amsmath,amssymb,latexsym,amsthm}
\usepackage[pdftex]{graphicx}
\usepackage{xcolor}
\usepackage{mathrsfs}
\usepackage{enumerate}
\usepackage{stmaryrd}
\usepackage{subcaption}
\usepackage{tikz}
\usepackage{verbatim}
\usepackage{float}
\usepackage[pdfencoding=auto, psdextra]{hyperref}
\PassOptionsToPackage{unicode}{hyperref}
\numberwithin{equation}{section}
\def\N{\mathbb{N}}
\def\Z{\mathbb{Z}}
\def\Q{\mathbb{Q}}
\def\R{\mathbb{R}}

\def\II{\parbox[][1.2cm][c]{0cm}{\ }}

\definecolor{myblue}{RGB}{147, 191, 252}
\definecolor{myblue2}{RGB}{112, 125, 149}
\definecolor{mygreen}{RGB}{0, 255, 162}
\definecolor{mon_rouge}{RGB}{255, 149, 148}

\newtheorem{remark}{Remark}
\newtheorem{definition}{Definition}
\newtheorem{theorem}{Theorem}
\newtheorem{lemma}{Lemma}[section]
\newtheorem{proposition}{Proposition}[section]
\newtheorem{corollary}{Corollary}[section]
\title{$2 \times 2$ hyperbolic systems of conservation laws \\ 
                                                in classes of functions of bounded $p$-variation}
\author{Olivier Glass\footnote{CEREMADE, UMR 7534,
Université Paris-Dauphine \& CNRS, PSL,
Place du Maréchal de Lattre de Tassigny,
75775 Paris Cedex 16, France. E-mail: {\tt glass{@}ceremade.dauphine.fr}}}
\date{\today}
\begin{document}
%
%
%
%
\maketitle
\begin{abstract}
In this paper, we consider $2 \times 2$ hyperbolic systems of conservation laws 
 in one space dimension with characteristic fields satisfying a condition that encompasses
 genuine nonlinearity and linear degeneracy as well as intermediate cases,
 namely, with standard notations, $r_i\cdot \nabla \lambda_i \geq 0$. \par
We prove the existence of entropy solutions in the fractional $BV$ spaces $\mathcal{W}_p(\R)$
 of functions of bounded $p$-variation, $p \in \left[1,\frac{3}{2}\right]$, for small initial data.
\end{abstract}
\tableofcontents
%
%
%
%
%
%
%
\section{Introduction}
\label{Sec:Introduction}
The existence of solutions to the initial value problem for hyperbolic systems of conservation laws
 under standard conditions on the characteristic fields was established in the celebrated paper 
 by James Glimm \cite{Glimm65} in 1965, when the initial data is a function of small total variation. 
The space $BV(\R)$ is of central importance in Glimm's result, and remains up to now 
 a standard space for many results on one-dimensional hyperbolic systems of conservation laws,
 even if for particular genuinely nonlinear systems -- typically $2 \times 2$ systems -- 
 important results were obtained in the space $L^{\infty}(\R)$ 
 \cite{BianchiniColomboMonti10,Cheverry98,GlimmLax70}. \par
In this paper, we are interested in the existence theory for entropy solutions of one-dimensional
 $2 \times 2$ hyperbolic systems of conservation laws in the Wiener class ${\mathcal W}_{p}(\R)$ 
 of functions of bounded $p$-variation (also denoted $BV_{s}(\R)$, $s=1/p$), which can be seen 
 as an intermediate space between $L^{\infty}(\R)$ and $BV(\R)$
 (although it is not an interpolated space between the two in the strict sense). \par
We recall that one-dimensional $2 \times 2$ hyperbolic systems of conservation laws
 are partial differential equations of the following form:
\begin{equation} \label{Eq:SCL}
u_{t} + f(u)_{x}= 0 \ \text{ for } \ (t,x) \in \R^{+} \times \R,
\end{equation}
where the unknown is $u= u(t,x): \R^{+} \times \R \rightarrow \R^2$ 
 and the flux function $f:\Omega \rightarrow \R^{2}$, defined on some open set $\Omega \subset \R^{2}$,
 satisfies the strict hyperbolicity condition that the eigenvalues $\lambda_{1}, \lambda_{2}(u)$ 
 of the Jacobian matrix $Df(u)$ are real and distinct:
\begin{equation} \label{Eq:StrictHyperbolicity}
\lambda_{1}(u) < \lambda_{2}(u).
\end{equation}
Let us denote $r_{1}(u), r_{2}(u)$ eigenvectors associated to these eigenvalues.
%
%
%
%
%
We will use an assumption on characteristic fields $(\lambda_i,r_i)$ that is slightly broader 
 than the standard genuine nonlinearity or linear degeneracy. 
Precisely, we will assume the property
\begin{equation} \label{Eq:Monotonicity} \tag{{\bf M}}
r_{i} \cdot \nabla \lambda_{i} \geq 0 \text{ in } \Omega,
\end{equation}
and we will refer to characteristic fields satisfying this condition as {\it monotone.} 
Naturally, fields satisfying $r_{i}\cdot \nabla \lambda_{i} \leq 0$ are also monotone, 
 since it suffices to reverse the direction of $r_{i}$. 
 Of course, this condition encompasses Lax's conditions of genuine nonlinearity 
 (which means $r_{i} \cdot \nabla \lambda_{i} > 0$ in $\Omega$ and is further abridged as GNL)
 and linear degeneracy (which means $r_{i}  \cdot \nabla \lambda_{i} = 0$  in $\Omega$
 and is further abridged as LD). 
We refer to \cite{BressanBook00,DafermosBook10} for general references on the topic of hyperbolic systems of conservation laws. \par
\medskip
Our goal in this paper is to build entropy solutions
 for such systems in the space of functions with finite $p$-variation class ${\mathcal W}_{p}(\R)$. 
We recall that the space $BV(\R;\R)$ can be defined as the space of functions $v: \R \rightarrow \R$
 for which
\begin{equation} \label{Eq:TV}
TV(v) := \sup_{N} \sup_{x_{1} < \dots < x_{N}} \sum_{k=0}^{N-1} |v(x_{k+1}) - v(x_{k})| < +\infty.
\end{equation}
For $1 \leq p < +\infty$, the $p$-variation of $v$  is defined by analogy by
\begin{equation} \label{Eq:TVp}
V_{p}(v) : = 
\left(\sup_{N} \sup_{x_{1} < \dots < x_{N}} \sum_{k=0}^{N-1} |v(x_{k+1}) - v(x_{k})|^p\right)^{1/p},
\end{equation}
and the class $\mathcal{W}_{p}(\R)$ of functions of finite $p$-variation is defined 
 as the set of functions $v: \R \rightarrow \R$ for which $V_{p}(u) < +\infty$. 
This space is naturally endowed with the norm:
\begin{equation} \label{NormeWp}
    \| v \|_{\mathcal{W}_p} := \| v \|_{\infty} + V_{p}(v).
\end{equation}
Let us briefly recall the definition of entropy solutions of \eqref{Eq:SCL}.
An entropy couple for a hyperbolic system of conservation laws \eqref{Eq:SCL}
 is defined as a couple of regular functions $(\eta,q) : \Omega \rightarrow \R$ satisfying:
\begin{equation} \label{Def:CoupleEntropie}
\forall u \in \Omega, \ \ D\eta(u) \cdot Df(u) = Dq(u).
\end{equation}
Then we use the following definition.
\begin{definition} \label{Def:SolutionEntropie}
A function $u \in L^\infty(0,T;\mathcal{W}_p(\R)) \cap C^{\frac{1}{p}}([0,T] ;L^p_{loc}(\R))$ is called an
 {\it entropy solution} of \eqref{Eq:SCL} when it is a weak solution in the sense of distributions
 and when for any entropy couple $(\eta,q)$, with $\eta$ convex, one has in the 
 sense of distributions
\begin{equation} \label{Def:SolEntrop1}
\eta(u)_t + q(u)_x \leq 0,
\end{equation}
that is, for all $\varphi \in {\mathcal D}((0,+\infty) \times \R)$ with $\varphi \geq 0$, 
\begin{equation} \label{InegEntropie}
\int_{\R_+ \times \R} \big( \eta(u(t,x)) \varphi_t(t,x) + q(u(t,x)) \varphi_x(t,x) \big)
 \, dx \, dt \geq 0.
\end{equation}
\end{definition}
\ \par
\noindent
In this paper, we prove the existence of entropy solutions of $2 \times 2$ systems \eqref{Eq:SCL} 
 in the class $\mathcal{W}_p(\R)$.
\begin{theorem} \label{Thm:Main}
Consider a strictly hyperbolic $2 \times 2$ system \eqref{Eq:SCL} satisfying \eqref{Eq:Monotonicity}. 
Let $p \in \left[1,\frac{3}{2}\right]$. 
Let $\overline{u} \in \Omega$. 
There exists $\varepsilon_0>0$ such that for any $u_0 \in \mathcal{W}_p(\R)$ such that
\begin{equation} \label{Eq:Smallness}
\| u_0 - \overline{u} \|_{\mathcal{W}_p} < \varepsilon_0,
\end{equation}
there exists $u \in L^\infty(\R_+;\mathcal{W}_p(\R)) \cap C^{1/p}([0,+\infty);L^p_{loc}(\R))$
 an entropy solution of \eqref{Eq:SCL} with initial condition $u(0,\cdot)=u_0$,
satisfying moreover that for all $s,t \geq 0$,
\begin{gather}
\label{Eq:EstVpGlobale}
V_p(u(t,\cdot))^p \leq V_p(u_0)^p + \mathcal{O}(1)V_p(u_0)^{2p}, \\
\label{Eq:TimeEstGlobale}
\int_{\R} |u(t,x) - u(s,x)|^p \, dx \leq \mathcal{O}(1) V_p(u_0)^p \, |t-s|.
\end{gather}
\end{theorem}
\begin{remark} \label{rem:CasNgeq3}
A straightforward adaptation of an example by Young \cite{RYoung99} proves
 that one cannot hope for such a result to hold for $n \times n$ hyperbolic systems, $n \geq 3$, at least
 in the linearly degenerate case.
\end{remark}
\begin{remark} \label{rem:OndesEngendreesBV}
The proof gives actually a slightly more precise picture of \eqref{Eq:EstVpGlobale}, at least at the discrete level of
 wave-front tracking approximations (as described below). 
Indeed, the total production of new waves (created by interactions of waves of the same characteristic family),
 measured by the total sum of their initial strengths in absolute value, is of order $\mathcal{O}(1) V_p(u_0)^{2p}$.
This somehow expresses that new waves are bounded in classical total variation.
\end{remark}
\noindent
These remarks are more precisely justified in Section~\ref{Sec:Remarks}. \par
\ \par
We now discuss several aspects of our result. \par
\ \par
\noindent
%
%
%
%
{\bf Important differences between $TV$ and $V_p$.}
First, let us emphasize some key differences between the usual total variation and the $p$-variation for $p>1$. 
To illustrate those differences, we will consider a piecewise constant function $f$ 
 with successive jumps $a_1, \ldots, a_n$ located at $x_1< \ldots < x_n$. 
We observe that the $p$-variation is:
\begin{itemize}
 \item {\it non-local}, 
        in the sense that the change in $p$-variation when modifying a jump $a_i$ does not merely 
        depend on $a_i$,
 \item {\it non-monotone},
        in the sense that reducing $a_i$ in absolute value while maintaining its sign does not 
        necessarily decrease the $p$-variation
        (taking for instance $(a_1,a_2,a_3)=(5,-2,5)$, $p=2$ and changing $a_2$ to $-1$),
\item {\it non-commutative with respect to jumps}, 
        in the sense that permuting the $a_i$ affects $V_p(f)$,
        as seen for instance on the above example, 
\item  {\it partially insensitive to cancellations}, in the sense that merging two successive 
        discontinuities with opposite monotonicity does not necessarily affect the $p$-variation,
        as seen on the same example as above,
\item {\it sensitive to a bounded but variable magnification of jumps}, 
        in the sense that multiplying jumps by variable bounded coefficients may change dramatically
        the $p$-variation (comparing for instance the cases of $a_k=(-1)^k$, $k=1,\ldots,n$,
        which gives a $p$-variation of order $\mathcal{O}(n^{1/p})$,
        and of $a_k b_k$ where $b_k=1+\frac{(-1)^k}{10}$,
        which gives a $p$-variation of order $\mathcal{O}(n)$),
\item {\it of fluctuating homogeneity with respect to changes in jumps}, 
        in the sense that for instance adding  a $(n+1)$-th discontinuity of size $h$ to the right 
        of $x_n$ can add to $V_p(f)^p$ a value ranging from $|h|^p$ to $p V_p(f)^{p-1}|h|$ 
        at leading order.
\end{itemize}
The fact that the opposite holds in the case of the total variation is frequently used when considering 
 $BV$ solutions of \eqref{Eq:SCL}. \par
\ \par
\noindent
%
%
%
%
%
{\bf Previous works.}
As mentioned in the introduction, in the genuinely nonlinear case, the existence of entropy solutions 
 for $2 \times 2$ systems at this level of regularity is not new: 
 this was proved by Glimm and Lax \cite{GlimmLax70} under an additional
 assumption that was later removed by Bianchini, Colombo and Monti \cite{BianchiniColomboMonti10}; see also the 
 related work by Cheverry \cite{Cheverry98}. 
In this context, Theorem~\ref{Thm:Main} should be considered as  a result of propagation of regularity
 in $\mathcal{W}_p$ spaces (even if, as far as we know, uniqueness is not established for solutions of this regularity). 
These results heavily rely on genuine nonlinearity and the associated effect 
 of decay of positive waves. 
The result here relies on a completely different aspect of those systems, namely the weaker coupling property
 of $2 \times 2$ systems. 
This explains why our result is valid under the broader assumption \eqref{Eq:Monotonicity} which allows 
 all sort of degeneracy of the characteristic speed. \par
\ \par
Hyperbolic conservation laws in the context of Wiener spaces have been considered in different references
 in the last ten years. 
Let us first mention references in the case of scalar conservation laws.
Up to our knowledge, Bourdarias, Gisclon and Junca \cite{BourdariasGisclonJunca14} were the first to consider
 $BV_s$ spaces in the context of conservation laws: they prove in particular the decay of $TV_s$ for solutions of
 scalar conservation laws and study the regularizing effect for convex fluxes.
A regularity result for general fluxes, relying on the more general notion of $\Phi$-variation, 
        was obtained by Marconi~\cite{Marconi2018}.
The notion of $\Phi$-variation was also used in Jenssen-Ridder~\cite{JenssenRidderPhi20} where precise time estimates
 on the solutions are given.
Let us also mention other recent studies in the field: 
 Ghoshal-Guelmame-Jana-Junca \cite{MR4130242} and Guelmame-Billel-Junca-Clamond~\cite{MR4069619}. \ \par
%
%
%
Concerning systems, there are fewer references.
In Bourdarias-Gisclon-Junca-Peng~\cite{MR3538367}, the authors consider a $2 \times 2$ triangular system related to chromatography,
composed of an uncoupled scalar conservation law and a transport equation; they study existence in $BV_s \times L^\infty$.
The reference Bourdarias-Choudhury-Guelmame-Junca~\cite{BourdariasChoudhuryGuelmameJunca22} considers more general $2 \times 2$ 
systems with this triangular structure.
%
\ \par
All of the above references concerning scalar conservation laws or systems of conservation laws in $BV_s$
 have a feature in common: they consider equations that are {\it $p$-variation diminishing} 
 in the field that is measured in $BV_s$. 
This is a natural property for scalar conservation laws (and in the discrete case this can be justified 
 by Proposition~\ref{Pro:ElemePropsVp}--\ref{Item:Fusion} below), but can only apply to very particular systems. 
The present paper is, up to our knowledge, the first to consider the case of general $2 \times 2$ systems
 where interactions can increase the $p$-variation, 
 and to measure this increase of $p$-variation through waves interactions, whether by amplification effect in the case of
 interactions of waves of different families, or by production of new waves in the case of interactions of waves 
 of the same family. \par
\ \par
\noindent
%
%
%
%
%
{\bf Ideas of proof.}
To describe the first difficulty met when considering systems in $BV_s$, let us start with a simple example.
Consider a $2 \times 2$ system where one characteristic field (say, with negative speed) is LD and the other one
 (say, with positive speed) is GNL. 
Consider as an initial state a shock of family $2$ on the left, and then a family of small oscillating
 contact discontinuities of family $1$, giving a small $p$-variation, but a large total variation, on the right.
When interacting with the shock, the strength of these contact discontinuities will be amplified (or damped) after crossing 
 by a multiplicative factor (depending on the size of the shock and on the local states) which in a first approximation can be 
 considered constant.
This involves in general that the difference between the original strength of these contact discontinuities and their new one
 will be {\it of large total variation}.
This has an important consequence: we will not be able to measure the differences in strength progressively
 in time in absolute value.
In particular, one cannot hope to use Glimm's approach in this situation, and we need to rely on
 {\it global-in-time} estimates, at least to understand the effect of interactions of waves of opposite families. \par
On the other hand, interactions of waves of the same family are responsible for the creation of new waves in the solution, 
 and for the complexity of the general wave pattern. 
It seems quite difficult to approach these interactions with global-in-time estimates due to this pattern complexity,
 and inter alia because the order (from left to right) of the newly created waves is difficult to predict;
 as mentioned earlier this is quite relevant when considering the $p$-variation. \par
This leads to the following idea: to include in the measure of wave strengths a ``preamplification factor''
 that partially predicts their future strengths after interactions with waves of the opposite family, and to treat these
 preamplified strengths by using a Glimm-type functional that measures their evolution through
 interactions with waves of the same family. 
(Actually, interactions of waves of opposite families will be considered as well, but this will concern
 higher order corrections.)
In particular in the process the new waves will be estimated using absolute values, which eventually leads
 to Remark~\ref{rem:OndesEngendreesBV} above. \par
Doing so, we face another difficulty: in our situation, we cannot use standard Glimm's functional 
\begin{equation*}
\sum_{\text{all waves}} |\sigma_\alpha| + C \sum_{\alpha < \beta} |\sigma_\alpha| |\sigma_\beta|, 
\end{equation*}
(where $\sigma_\alpha$ measures the jump across a wave $\alpha$)
because it is not uniformly bounded at initial time, since in general the initial data is not in $BV$.
A natural way to extend it would be to consider instead
\begin{equation*}
\sum_{\text{all waves}} |\sigma_\alpha|^p + C \sum_{\alpha < \beta} |\sigma_\alpha|^p |\sigma_\beta|^p, 
\end{equation*}
but:
\begin{itemize}
\item[--] the left term above does not control the $p$-variation,
\item[--] the gain obtained by means of the right term at an interaction point is of order 
        $|\sigma_\alpha|^p |\sigma_\beta|^p$, which does not allow controlling errors of size
        $|\sigma_\alpha||\sigma_\beta|^k$, whatever $k$.
\end{itemize}
Hence we introduce a functional that mimics Glimm's one, but uses a notion of {\it nonlocal strength}, which when 
 measuring the strength of a wave $\alpha$ does not only depend on $\alpha$, but is better adapted to the 
 $p$-variation.
Since this Glimm-type functional relies on nonlocal, preamplified strengths, the evolution of this functional
is more intricate, and we will need a priori assumptions on 
 our (wave-front front-tracking) approximations to prove that it decays 
 (this part corresponds to ``local-in-time estimates''). \par
Once this decay is obtained, we will have to prove that it indeed allows measuring the $p$-variation
 of the approximation, that is, to remove the preamplification factors.
This is not straightforward, since as mentioned earlier, the $p$-variation is very sensitive to multiplication.
In particular, going back to our first example of a shock interacting with small oscillating contact discontinuities,
 one has to ensure that the coefficients giving the amplification of these contact discontinuities across 
 the interaction do not vary too much.
(As we will see, this is connected to the so-called Love-Young inequality.) 
Hence, we will also need a priori assumptions on the approximation in this part of the proof,
 to measure the variation of interactions coefficients {\it along front lines}.
This is the part where the ``global-in-time'' estimates are operated. \par
Due to this need of a priori estimates, the proof takes the form of a large induction argument, where an a priori idea
 of the $p$-variation of the solution, not only for given times but also along front-lines,
 will allow getting a better estimate and go forward in time. \par
Finally, let us mention that time estimates in this case are not a direct consequence of the uniform in time 
 $V_p$ estimate that the above considerations allow obtaining.
We will need indeed a localized version of the estimates above to conclude. \par
\ \par
\noindent
%
%
%
%
%
{\bf Structure of the paper.}
In Section~\ref{Sec:PreliminaryMaterial}, we give some elementary properties for $V_p$
 and its discrete counterparts (which are simple but will be used all along the paper),
 and of monotone characteristic fields (for which an important part but not the entirety can be considered classical).
In Section~\ref{Sec:MainIngredients}, we present the main concepts that we will use to prove Theorem~\ref{Thm:Main}.
In Section~\ref{Sec:ProofOfTheMainTheorem}, we prove Theorem~\ref{Thm:Main} by relying on the objects constructed 
 in Section~\ref{Sec:MainIngredients}.
In Section~\ref{Sec:Remarks}, we justify Remarks~\ref{rem:CasNgeq3} and \ref{rem:OndesEngendreesBV}.
Finally, Section~\ref{Sec:Appendix} is an Appendix where we detail the proofs of the technical statements
omitted above. \par
\ \par
\noindent
%
%
%
%
%
{\bf General notation.}
As a notation for subsequent sections, we will use the notation ``$\mathcal{O}(1)$'' to describe 
 a bounded quantity depending on $f$, 
 but independent of $u_0$, the discretization parameter $\nu$, the time-horizon $T$ (as described below), etc. 
The same is valid for all constants $c$, $C$. 
In the same way, we will write $\lesssim (\cdots)$ for $\leq \mathcal{O}(1)(\cdots)$. \par
%
%
%
%
%
%
%
%
%
\section{Preliminary material}
\label{Sec:PreliminaryMaterial}
\subsection{Functions of $p$-bounded variations}
\subsubsection{Elementary properties of the $p$-variation}
Given a finite real sequence $x=(x_1,\ldots,x_n)$, we define its {\it $p$-variation}
 that we denote $v_{p}(x)$ as the following:
\begin{equation} \label{Def:vp}
v_{p}(x):= \left( \sup_{k \leq n} \ \sup_{0=i_{0}<\dots<i_{k}=n} \ 
    \sum_{j=0}^{k-1} \left|x_{i_{j+1}} - x_{i_{j}+1} \right|^{p} \right)^{1/p}.
\end{equation}
Correspondingly we define {\it maximal $p$-sum of $x$}, which we denote $s_p(x)$, the version for partial sums,
 that is:
\begin{equation} \label{Def:sp}
s_{p}(x):= \left( \sup_{k \leq n} \ \sup_{0=i_{0}<\dots<i_{k}=n} \ 
    \sum_{j=0}^{k-1} \left|\sum_{\ell=i_{j}+1}^{i_{j+1}} x_{\ell} \right|^{p} \right)^{1/p}.
\end{equation}
Clearly,
\begin{equation*}
v_p (x_1,\ldots,x_n) = s_p(x_2-x_1,\ldots,x_n-x_{n-1}).
\end{equation*}
The connection with the $p$-variation of functions is straightforward: 
given a step function $f$, its $p$-variation is given by
\begin{equation*}
V_p(f) = s_p([f]_{i = 1,\ldots,N}),
\end{equation*}
where $[f]_1,\ldots,[f]_N$ represent the jumps in $f$, from left to right. \par
\ \par
Concerning $s_p$ and $v_p$, we will use the following vocabulary. 
In \eqref{Def:vp} or \eqref{Def:sp},  we will refer to a family of indices $(i_{0},\ldots,i_{k})$ 
 satisfying $0=i_{0}<\dots<i_{k}=n$ as a {\it partition}, 
 and in such a partition we will call a set of indices $\{i_j+1,\ldots,i_{j+1}\}$ 
 for some $j \in \{0,\ldots,k-1 \}$ a {\it component} of this partition. 
We will say that such a partition is optimal when it achieves the supremum
 in \eqref{Def:vp} or \eqref{Def:sp}. \par
\ \par
In the next statements, we give a few elementary properties of $s_p$ and $v_p$.
\begin{proposition}[Elementary properties of $s_p$ and $v_p$] \label{Pro:ElemePropsVp} \ \par
\begin{enumerate}
\item \label{Item:SpSpPrime} Let $1 \leq p \leq p'$ and $a=(a_i)_{i=1\ldots k}$. 
Then $s_{p'}(a) \leq s_p(a)$ and $v_{p'}(a) \leq v_p(a)$.
\item \label{Item:TriangIneg} Let $a=(a_i)_{i=1\ldots k}$ and $b=(b_i)_{i=1\ldots k}$. 
Then $s_p(a + b) \leq s_p(a) + s_p(b)$ and $v_p(a + b) \leq v_p(a) + v_p(b)$.
\item \label{Item:DLVp} For some $C>0$ depending only on $p$ the following holds. 
Let $a=(a_i)_{i=1\ldots k}$ and $b=(b_i)_{i=1\ldots k}$. 
Then $s^p_p(a + b) \leq s^p_p(a) + C s_p(a)^{p-1}s_p(b) + 2 s^p_p(b)$. 
In particular, if $s_p(b) \leq \kappa s_p(a)$, 
then $s^p_p(a + b) \leq s^p_p(a) + (C + 2 \kappa^{p-1}) s_p(a)^{p-1}s_p(b)$. 
\item \label{Item:Fusion} Let $a=(a_i)_{i=1\ldots k}$. 
For $j=1,\ldots,k-1$, call $(b_i^j)_{i=1\ldots k-1}=(a_1,\ldots,a_{j-1},a_j + a_{j+1}, a_{j+2},\ldots,a_k)$. 
Then $s_p(b^j) \leq s_p(a)$ in all cases and $s_p(b^j) = s_p(a)$ when $a_j$ and $a_{j+1}$ have the same sign
(in particular swapping $a_j$ and $a_{j+1}$ does not change $s_p(a)$ in this case).
%
\item \label{Item:VpFusion} Let $a=(a_i)_{i=1\ldots k}$ and $b=(b_i)_{i=1\ldots k'}$. 
Denote $a \diamond b := (a_1,\ldots,a_k,b_1,\ldots,b_{k'})$. 
Then we have $s_p(a)^p + s_p(b)^p \leq s_p(a \diamond b)^p$.
\item \label{Item:Interleave} 
Let $a=(a_i)_{i=1\ldots k}$ and $b=(b_i)_{i=1\ldots k'}$. 
Let $c=(c_i)_{i=1\ldots k+k'}$ an interleaved sequence of $a$ and $b$, that is, 
 a sequence whose elements are the collection of all $a_i$ and $b_i$ 
 in such a way that the order of elements of $a$ and $b$ is preserved. 
Then $v_p(a) \leq v_p(c)$ and $|s_p(a) - s_p(b)| \leq s_p(c) \leq s_p(a) + s_p(b)$.
\item \label{Item:Repetition} 
Let $a=(a_i)_{i=1\ldots k}$. 
Consider $b=(b_i)_{i=1\ldots k'}$, whose first terms are $a_1$ repeated an arbitrary number of times, 
then $a_2$ repeated an arbitrary number of times, \ldots, until $a_k$ repeated an arbitrary number of times.
Then $v_p(b) \leq v_p(a)$. 
\item \label{Item:Ajoutde0} 
Let $a=(a_i)_{i=1\ldots k}$.
Then $v_p(0,a_1,\ldots,a_k) \leq |a_1| + v_p(a)$ and $v_p(a_1,\ldots,a_k,0) \leq |a_k| + v_p(a)$.
\item \label{Item:LipVp} Let $(-c,c)$ an interval in $\R$, $f \in W^{1,\infty}(-c,c)$
 and $a=(a_i)_{i=1\ldots k} \in (-c,c)^k$. 
Then 
\begin{equation} \nonumber 
v_p\left( f(a_i)\right) \leq  \|f'\|_\infty v_p(a).
\end{equation}
\item \label{Item:VpProd} Consider $n$ sequences $a^j=(a^j_i)_{i=1\ldots k}$ for $j=1,\ldots,n$, 
and let $c_i = \prod_{k=1}^{n}  a^k_i$. 
Then 
\begin{equation} \nonumber 
v_p\left( c_i \right) \leq \sum_{j=1}^n \left(v_p(a^j) \prod_{k \neq j} \|a^k\|_\infty\right).
\end{equation}

\end{enumerate}
\end{proposition}
\begin{proposition}[Sensitivity of $s_p$ with respect to endpoint values] 
\label{Pro:PropsVpEndpoints}
 \ \par
\begin{enumerate}
\item \label{Item:EndPoint1} Let $ n \geq 1$, $x_1, \ldots, x_n$ in $\R$. 
Then the map $t \mapsto s_p^p(x_1,\ldots,x_{n},t)$ is convex, decreasing on $\R_-$ 
and increasing on $\R_+$.
\item \label{Item:EndPoint2} Let $ n \geq 1$, $x_1, \ldots, x_n, x_{n+1}$ in $\R$. Then
\begin{equation} \label{Eq:ForceAGauche}
s_p^p(x_1,\ldots,x_{n+1}) - s_p^p(x_1,\ldots,x_n) \geq
\left\{ \begin{array}{l}
    p|x_n|^{p-1} |x_{n+1}| \ \text{ if } x_{n+1} \text{ and } x_n \text{ have the same sign,} 
                                                                                \medskip \\
    |x_{n+1}|^p \ \text{ in all cases.}
\end{array} \right. 
\end{equation}
\item \label{Item:EndPoint3} The equivalent properties are valid mutatis mutandis
 when varying the leftmost variable $x_1$ rather than the rightmost one $x_{n+1}$.
\end{enumerate}
\end{proposition}
The proofs of Propositions~\ref{Pro:ElemePropsVp} and \ref{Pro:PropsVpEndpoints} are postponed
 to Subsection~\ref{Subsec:ProofsVp}.
\subsubsection{The discrete Love-Young inequality and its consequence}
A central tool in this paper is the following discrete Love-Young inequality,
 see L.C.~Young\cite[(5.1)]{LCYoung36}.
\begin{proposition} \label{Pro:LYoriginal}
Let $p,q \geq 1$ such that
\begin{equation} \label{Eq:LYCondPQ}
\frac{1}{p} + \frac{1}{q} >1.
\end{equation}
Consider the $2n$ real numbers $x_{1}, \dots, x_{n}$ and $y_{1}, \dots, y_{n}$. Then one has
\begin{equation} \label{Eq:InegLY}
\left| \sum_{1\leq i \leq j \leq n} x_{i} y_{j} \right| 
\leq \left(1+ \zeta\left(\frac{1}{p} + \frac{1}{q}\right)\right) s_{p}(x) s_{p}(y),
\end{equation}
where
$\zeta(s):= \sum_{n \geq 1} \frac{1}{n^{s}}$ is Riemann's zeta function.
\end{proposition}
The importance of this inequality stems from the fact that it takes cancellations into account;
 it would not hold with absolute values inside the sum. 
It was introduced in \cite{LCYoung36} to generalize the Stieltjes integral. \par
\ \par
We will in particular use the following consequence of the discrete Love-Young inequality.
\begin{corollary} \label{Cor:MultiplSp}
Let $a=(a_i)_{i=1\ldots n}$ and $b=(b_i)_{i=1\ldots n}$. Call $ab := (a_1 b_1,\ldots,a_n b_{n})$.
Suppose that $p<2$.
Then for some constant $C>0$ depending only on $p$,
\begin{equation} \label{Eq:MultiplSp}
s_p(ab) \leq C s_p(a) \big( \|b\|_\infty + v_p(b) \big).
\end{equation}
\end{corollary}
Corollary~\ref{Cor:MultiplSp} is proved in Subsection~\ref{Subsec:ProofsDLY}.
\subsubsection{Classical properties of $BV$ that can be brought to $\mathcal{W}_p$}
We list here some properties of the Wiener spaces $\mathcal{W}_p$ that are counterparts
 of classical properties of $BV$. 
These are commonly used for the construction of $BV$ solutions of hyperbolic systems of conservation laws.
\begin{proposition} \label{Pro:Classical}
The following statements hold true in $\mathcal{W}_p$.
\begin{enumerate}
\item Functions in $\mathcal{W}_p$ are regulated.
\item (Approximation by piecewise constant functions)  Let $u \in \mathcal{W}_p$.
There exists a sequence $(u_{n})_{n \geq 0}$ of step functions such that $u_n \rightarrow u$ in $L^1_{loc}$
 and $V_p(u_n) \leq V_p(u)$, $\| u_n \|_{\infty} \leq \| u \|_{\infty}$.
\item (Helly's selection principle) The imbedding $\mathcal{W}_p(\R) \hookrightarrow L^1_{loc}(\R)$ is compact.
\end{enumerate}
\end{proposition}
These statements are rather classical;
 see \cite[Proposition 3.33]{DudleyNorvaivsa11} for first one and (the proof of) \cite[Theorem 2.1]{DudleyNorvaivsa11}
 for the second one. 
For Helly's principle, we refer to \cite{MusielakOrlicz59}. \par
%
%
%
%
%
%
%
%
\subsection{Monotone characteristic fields and interaction estimates}
\label{Subsec:CharacFields}
In this paragraph, we consider monotone characteristic fields, such as described in the introduction. 
We first see how the wave curves can be constructed as in the genuinely nonlinear 
 and linearly degenerate cases. 
Then we introduce interaction estimates that are slightly more precise
 than the usual ones (but restricted to $2 \times 2$ systems.)
\subsubsection{Wave curves for monotone characteristic fields}
We recall that Lax's $i$-th wave curve starting at $u_0 \in \Omega$ are defined 
 as the set of all (right) states that can be connected to $u_0$ by means of (a succession of) elementary waves
 (rarefaction/contact discontinuity/admissible shock) of that family. 
Due to \cite[Theorem 3.2]{Bianchini-Riemann03}, these are uniquely defined, even without any other
 assumption than strict hyperbolicity. \par
We first recall that for linearly degenerate fields ($\nabla \lambda_i \cdot r_i =0$),
 the Hugoniot locus coincides with the rarefaction curve, and they both correspond to contact discontinuities. 
It follows that in that case, one can define the Lax curve in the same way 
  as in the genuinely nonlinear case ($\nabla \lambda_i \cdot r_i >0$), 
  that is, by choosing the rarefaction curve for $s \geq 0$ and the shock curve for $s<0$. \par
The main result of this paragraph is that this remains true under the broader
 condition \eqref{Eq:Monotonicity}. 
Of course, the proof of existence of Lax curve under this assumption is not new, 
 since as mentioned before no condition is needed in addition to strict hyperbolicity. 
This part does not depend on the fact that the system is $2 \times 2$, so the next statement is written
 in the general $n \times n$ case. \par
We parameterize here the various $i$-curves $\Gamma$ starting at $u_0$ with
\begin{equation} \label{Eq:Parametrisation}
s=\ell_i(u_0) \cdot (\Gamma(s) - u_0), 
\end{equation}
where $(\ell_1(u),\ldots,\ell_n(u))$ is the dual basis of $(r_1(u),\ldots,r_n(u))$. \par
We denote $\mathcal{R}_{i} =\mathcal{R}_{i}(s,u_0)$ and $\mathcal{S}_{i}=\mathcal{S}_{i}(s,u_0)$ 
 the $i$-th rarefaction curve and the $i$-th Hugoniot locus starting from $u_0$. 
These are considered in both $s>0$ and $s<0$ cases. 
We have the following elementary statement.
\begin{proposition} \label{Pro:LaxCurves}
Consider a strictly hyperbolic $n \times n$ system. Under Assumption \eqref{Eq:Monotonicity}
 on the $i$-th characteristic field, given $u_0 \in \Omega$,
 Lax's $i$-th curve starting from $u_0$ is given by:
\begin{equation} \label{Eq:LaxCurves}
    \mathcal{T}_i(s;u_0) = \left\{ \begin{array}{l}
        \mathcal{R}_i(s;u_0) \text{ if } \ s \geq 0, \\
        \mathcal{S}_i(s;u_0) \text{ if } \ s \leq 0.
    \end{array} \right.
\end{equation}
\end{proposition}
The proof of Proposition~\ref{Pro:LaxCurves} is postponed to Subsection~\ref{Subsec:WCMontoneFields}. 
Though this is rather classical, we cannot apply the standard proof for GNL case 
due to the lack of usable Taylor series expansion in this case. \par
\subsubsection{The Riemann problem and its variants}
From now on, we go back to $2 \times 2$ systems, and work with Riemann coordinates. 
More precisely, since the system under view is $2 \times 2$, normalizing $r_1$ and $r_2$
so that $[r_1,r_2]=0$, we may introduce near the base point $\overline{u}$ coordinates $w=(w_1,w_2) = w[u]$ 
defined by
\begin{equation} \label{Eq:RiemannCoordinates}
\frac{\partial u}{\partial w_i} = r_i(u), \ \text{ for } i=1,2.
\end{equation}
Accordingly, we can use the jump of $w_1$ (respectively of $w_2$) to parameterize the $1$-wave curves 
(resp. the $2$-wave curves). 
In particular, the rarefaction curves take a simple form in these coordinates: 
 for $w^0=(w_1^0,w_2^0)$, we have
\begin{equation} \label{Eq:RarRC}
\mathcal{R}_1(s;w^0) = (w_1^0 + s ,w_2^0) 
\ \text{ and } \ 
\mathcal{R}_2(s;w^0) = (w_1^0  ,w_2^0+s).
\end{equation}
Note that we will also use the rarefaction curve for {\it negative} waves. 
We will refer to a wave between $u_l$ and $u_r = \mathcal{R}_i(\sigma;u_l)$, $\sigma <0$
 as a {\it compression wave}. \par
On the other side, due to the classical (at least) second-order tangency 
between rarefaction curves and the Hugoniot locus, one can find smooth functions $\mathcal{D}_i$, $i=1,2$
such that shock curves take the following form in Riemann coordinates:
\begin{equation} \label{Eq:ChocRC}
\mathcal{S}_1(s;w^0) = \big(w_1^0 + s ,w_2^0 + s^3 \mathcal{D}_1(s;w^0) \big) 
\ \text{ and } \ 
\mathcal{S}_2(s;w^0) = (w_1^0 +s^3 \mathcal{D}_2(s;w^0) ,w_2^0+s).
\end{equation}
Note that $\mathcal{D}_i$ is also defined for positive values (still corresponding to the Hugoniot locus). \par
\ \par
Since we are considering monotone characteristic fields, in accordance with Proposition~\ref{Pro:LaxCurves},
 we can form Lax curves $\mathcal{T}_i$ according to \eqref{Eq:LaxCurves}, which we parameterize by means 
 of the Riemann coordinates as above. 
Now in the sequel we will use {\it three} types of waves:
\begin{itemize}
\item shock waves (SW) corresponding to $(u_-,u_+)$ with 
        $u_+=\mathcal{S}_i(\sigma,u_-)$ with $\sigma \leq 0$,
\item rarefaction waves (RW) corresponding to $(u_-,u_+)$ with 
        $u_+=\mathcal{R}_i(\sigma,u_-)$ with $\sigma \geq 0$, 
\item compression waves (CW) corresponding to $(u_-,u_+)$ with 
         $u_+=\mathcal{R}_i(\sigma,u_-)$ with $\sigma \leq 0$.
\end{itemize}
All cases above include the possible case when the flux degenerates so that the wave
 is actually a contact discontinuity, but we keep the above vocabulary (SW/RW/CW)
 in that case. \par
\ \par
With these three types of waves, we will use in the construction several variants of the Riemann problem. 
All are gathered in the following statement.
\begin{proposition} \label{Prop:RiemannS}
There exists a neighborhood $U$ of $\overline{u}$ and $c>0$ such that the following holds.
Given $\mathcal{U}_1 \in \{ \mathcal{R}_1, \mathcal{T}_1\}$ and 
$\mathcal{U}_2 \in \{ \mathcal{R}_2, \mathcal{T}_2\}$, for any $u_l, u_r \in U$,
there exists unique $\sigma_1, \sigma_2 \in (-c,c)$ such that
\begin{equation} \label{Eq:RiemannS}
u_r = \mathcal{U}_2(\sigma_2; \mathcal{U}_1(\sigma_1;u_l)),
\end{equation}
with $\mathcal{U}_1(\sigma_1;u_l) \in \Omega$.
\end{proposition}
The proof is standard in the case $\mathcal{U}_1 =\mathcal{T}_1$ and $\mathcal{U}_2 =\mathcal{T}_2$,
and totally similar (simpler, even) in other cases, relying on the inverse mapping theorem. \par
\ \par
%
%
%
%
%
%
%
\subsubsection{Sharper interaction estimates}
\label{sss:SIE}
In what follows we will use interaction estimates that are stronger 
 (and of a slightly different form in the case of opposite families)
 than classical Glimm's estimates~\cite{Glimm65}, 
 relying in particular on the fact that the system is $2 \times 2$ 
 and that we can use Riemann invariants to measure wave strengths. 
A difference with classical interaction estimates 
(besides the fact that we treat variants of the Riemann problem) 
is that the error term is not the same for both outgoing waves.
Moreover, we treat all shocks in a similar fashion as what is done when
 studying strong shocks (see e.g. \cite{Schochet1991Feb}). \par
\ \par
For that purpose we distinguish between interactions of waves of the same family 
and interactions of waves of opposite families. Inside each of these propositions, we also consider 
the possibility of a compression wave. We will refer to shock and rarefaction waves as {\it classical} waves
(as opposed to compression waves.) \par
\ \par
\noindent
We will use the following notation: given a wave $u_+= \mathcal{U}_i(\sigma;u_-)$
with $\mathcal{U}_i \in \{\mathcal{T}_i,\mathcal{R}_i\}$, 
we define $\sigma_*$ as $\sigma_*=0$ unless $(u_-,u_+)$ is a shock in which case $\sigma_*=\sigma$. 
In other words, $\sigma_*=0$ if $\mathcal{U}_i = \mathcal{R}_i$ and $\sigma_*=\sigma_-$ 
for $\mathcal{U}_i = \mathcal{T}_i$. 
The open set $U$ and the constant $c>0$ mentioned in the next statements are the ones
 obtained in Proposition~\ref{Prop:RiemannS}. \par
\ \par
\noindent
{\it Interactions of waves of opposite families.} 
It is classical that interactions of waves of opposite family yield a quadratic error term. 
Here we will use that for $2 \times 2$ systems, when the waves are measured with Riemann invariants, 
 the error is actually of fourth order. 
This was already noticed in~\cite{BressanColombo95} (see the proof of Lemma~2 therein) 
 in the case of GNL fields, in a form that is slightly different from ours. \par
\begin{proposition}[Interaction of waves of opposite families] \label{Prop:InteractionDF} 
There exist two functions $\mathcal{C}^1=\mathcal{C}^1(u_m;s_1,s_2)$ 
 and $\mathcal{C}^2=\mathcal{C}^2(u_m;s_1,s_2)$ of Lipschitz class on $U \times (-c,c)^2$,
 such that was follows is valid. \par
\ \par
\noindent
{\bf 1.} (Classical waves) 
Suppose that $u_l, u_m, u_r$ are states in $U$ such that
\begin{equation*}
u_m=\mathcal{T}_2(\sigma_2;u_l) \ \text{ and } \ u_r=\mathcal{T}_1(\sigma_1;u_m).
\end{equation*}
Then one has $u_r = \mathcal{T}_2(\sigma'_2;\mathcal{T}_1(\sigma'_1;u_l))$ with
\begin{equation} \label{Eq:InteractionDF}
\sigma_1' = \sigma_1 \big( 1 + \mathcal{C}^1(u_m;\sigma_1,\sigma_2) (\sigma_2)_*^3\big)
 \ \text{ and } \ 
\sigma_2' = \sigma_2 \big( 1 + \mathcal{C}^2(u_m;\sigma_1,\sigma_2) (\sigma_1)_*^3\big).
\end{equation}
{\bf 2.} (Classical/compression waves). 
Suppose that $u_l, u_m, u_r$ are states in $U$ such that
\begin{equation*}
u_m=\mathcal{T}_2(\sigma_2;u_l) \ \text{ and } \ u_r=\mathcal{R}_1(\sigma_1;u_m).
\end{equation*}
Then one has $u_r = \mathcal{T}_2(\sigma'_2;\mathcal{R}_1(\sigma'_1;u_l))$ with
\begin{equation} \label{Eq:InteractionDF-CW}
\sigma_1' = \sigma_1 \big( 1 + \mathcal{C}^1(u_m;\sigma_1,\sigma_2) (\sigma_2)_*^3 
                                + \mathcal{O}(1) |\sigma_1|^2 (\sigma_2)_*^3\big)
 \ \text{ and } \ 
\sigma_2' = \sigma_2,
\end{equation}
and the equivalent holds when $u_m=\mathcal{R}_2(\sigma_2;u_l) 
                                    \ \text{ and } \ u_r=\mathcal{T}_1(\sigma_1;u_m)$. \par
\ \par
\noindent
{\bf 3.} (Compression and rarefaction waves).
Suppose that $u_l, u_m, u_r$ are states in $U$ such that
\begin{equation*}
u_m=\mathcal{R}_2(\sigma_2;u_l) \ \text{ and } \ u_r=\mathcal{R}_1(\sigma_1;u_m).
\end{equation*}
Then one has $u_r = \mathcal{R}_2(\sigma_2;\mathcal{R}_1(\sigma_1;u_l))$.
\end{proposition}
Of course, the above cases overlap a bit. Note that only the first two cases are somewhat new 
(in particular in the fact that the functions $\mathcal{C}^1$ and $\mathcal{C}^2$ are Lipschitz).
See e.g. Bressan-Colombo~\cite{BressanColombo95} for a close statement. \par
\ \par
\noindent
{\it Interactions of waves of the same family.}
Here classical interactions estimates give an error term of third order (see again e.g. \cite{BressanColombo95}). 
Relying again on the fact that our system is $2 \times 2$ and that we measure wave strengths with Riemann invariants, 
we can partially improve this order for the outgoing wave of the same family. \par
\begin{proposition}[Interaction of waves of the same family] 
\label{Prop:InteractionSF}
Let $i \in \{1,2\}$.
\ \par
\noindent
{\bf 1.} (Classical waves) 
Suppose that $u_l, u_m, u_r$ are states in $U$ such that 
$u_m=\mathcal{T}_i(\sigma_i;u_l)$ and  $u_r=\mathcal{T}_i(\sigma'_i;u_l)$.
Then one has $u_r = \mathcal{T}_2(\widetilde{\sigma}_2;\mathcal{T}_1(\widetilde{\sigma}_1;u_l))$ with
\begin{equation} \label{Eq:InteractionsSF}
\widetilde{\sigma}_i 
= \sigma_i + \sigma_i' + \mathcal{O}(1) |\sigma_i|^3|\sigma'_i|^3 (|\sigma_i| + |\sigma_i'|)^3
\ \text{ and } \ 
\widetilde{\sigma}_{3-i} = \mathcal{O}(1) |\sigma_i||\sigma'_i| (|\sigma_i| + |\sigma_i'|).
\end{equation}
\ \par
\noindent
{\bf 2.} (Shocks \& compression waves).
Suppose that $u_l, u_m, u_r$ are states in $U$ such that
 $u_m=\mathcal{T}_i(\sigma_i;u_l)$ and $u_r=\mathcal{R}_i(\sigma'_i;u_m)$,
 with $\sigma_i<0$ and $\sigma'_i<0$.
Then one has $u_r = \mathcal{T}_2(\sigma_2;\mathcal{T}_1(\widetilde{\sigma}_1;u_l))$ with 
\begin{equation} \label{Eq:InteractionsSF-CWvS}
\widetilde{\sigma}_i = \sigma_i + \sigma_i' + \mathcal{O}(1) |\sigma'_i|^3 (|\sigma_i| + |\sigma_i'|)^6
\ \text{ and } \ 
\widetilde{\sigma}_{3-i} = \mathcal{O}(1) |\sigma'_i| (|\sigma_i| + |\sigma_i'|)^2.
\end{equation}
The equivalent is valid when $u_m=\mathcal{R}_i(\sigma'_i;u_l)$ and $u_r=\mathcal{T}_i(\sigma_i;u_m)$. \par
\ \par
\noindent
{\bf 3.} (Compression and rarefaction waves).
Suppose that $u_l, u_m, u_r$ are states in $U$ such that
$u_m=\mathcal{R}_i(\sigma_i';u_l) \ \text{ and } \ u_r=\mathcal{R}_i(\sigma_i;u_m)$. 
Then one has $u_r = \mathcal{R}_i(\sigma_i + \sigma_i';u_l)$.
\end{proposition}
\begin{remark}
Notice the unusual $|\sigma'_i|^3$ error term in \eqref{Eq:InteractionsSF-CWvS}. 
This is due to the transformation of a compression wave in a shock wave in such a situation.
\end{remark}
\medskip
\noindent
These propositions are proved in Subsection~\ref{Subsec:ProofsShaperEstimates}.
We will call $C_*$ a bound for functions $\mathcal{C}^1$ and $\mathcal{C}^2$ in Proposition~\ref{Prop:InteractionDF}
 and for all ``$\mathcal{O}(1)$'' in Proposition~\ref{Prop:InteractionSF}. 
%
%
%
%
%
%
%
%
%
%
%
\section{Main ingredients}
\label{Sec:MainIngredients}
Given $\overline{u} \in \Omega$, we first introduce $r>0$ such that:
\begin{itemize}
\item $\overline{B}(\overline{u};r)$ is included in the neighborhood $U$ introduced in Proposition~\ref{Prop:RiemannS},
    so the different Riemann problems with states in $\overline{B}(\overline{u};r)$ are solvable
    with intermediate states in $\Omega$;
\item the intervals $\{ \lambda_1(u), u \in B(\overline{u};r) \}$ 
    and $\{ \lambda_2(u), u \in B(\overline{u};r) \}$ are disjoint.
\end{itemize}
We strengthen a bit the last requirement as follows.
Translating the reference frame by a constant speed if necessary,
 we can suppose without loss of generality that for constants $\Lambda_1^{max} < 0$ 
 and $\Lambda_2^{min}>0$, one has
\begin{equation} \label{Eq:SeparationVitesses}
\forall u \in \overline{B}(\overline{u};r), \ \ \lambda_1(u) \leq \Lambda_1^{max} < 0 < \Lambda_2^{min} \leq \lambda_2(u). 
\end{equation} 
It suffices indeed to consider the modified flux $\widetilde{f}(u) = f(u) - \lambda u$ for proper $\lambda$;
 then entropy solutions $u(t,x)$  of \eqref{Eq:SCL} with flux $f$ correspond in a one-to-one manner
 to entropy solutions $\widetilde{u}(t,x):= u(t,x + \lambda t)$ of \eqref{Eq:SCL} with flux $\widetilde{f}$. 
The requirement \eqref{Eq:SeparationVitesses} is not essential;
  it simplifies a bit the descriptions (and the figures) below. \par

As a first condition on $\varepsilon_0$ in \eqref{Eq:Smallness}, we suppose that 
\begin{equation} \label{Eq:1stSmallness}
    u_0(\R) \subset B(\overline{u},r/2)
    \ \text{ and } \ 
    V_p(u_0) \leq  \frac{1}{2}, 
\end{equation}
for any $u_0$ fulfilling \eqref{Eq:Smallness}. \par
Above and all along this paper, we measure the $p$-variation of a vector-valued function $u : \R \longrightarrow \R^2$
 along the Riemann coordinates, that is,
\begin{equation} \label{Eq:VpU}
V_p(u)^p := V_p(w_1(u))^p + V_p(w_2(u))^p .
\end{equation}
\ \par
\noindent
We now describe the main ingredients of the proof in the next paragraphs. \par
%
%
%
%
%
%
%
%
%
\subsection{Wave front-tracking algorithm}
\label{Subsec:FTA}
In this section, we describe a front-tracking algorithm which is a variant of the one of Bressan-Colombo~\cite{BressanColombo95};
we will use it to construct approximations of a solution. 
Given a threshold $\nu \in (0,\nu_0)$ (where $\nu_0>0$ is chosen small enough), we construct a piecewise constant function
 on a polygonal subdivision of $\R_+ \times \R$, where each line of discontinuity corresponds to a simple $1$ or $2$-wave. \par
\ \par
\noindent
{\bf Convention on the propagation speed.}
In what follows, all discontinuities will at first approximation travel at shock speed as defined below. 
However, we will further modify a bit this speed.
\begin{itemize}
\item For a wave of family $i$ separating $u_-$ and $u_+$ (either a shock, a rarefaction wave 
 or a compression wave), the shock speed $\overline{\lambda}_i(u_-,u_+)$ is defined as
\begin{equation} \label{Eq:VitesseChoc}
\overline{\lambda}_i(u_-,u_+) := \frac{\ell_i (u_-) \cdot \left(f(u^+) - f(u^-)\right)}
                                        {\ell_i (u_-) \cdot \left(u^+ - u^-\right)},
\end{equation}
which is consistent with \eqref{Eq:DefLambdaBar} below and exact in the case of a shock.
\item All waves will travel by default at {\it modified} shock speed, that is, for a wave of family $i$ separating
$u_-$ and $u_+$, at the speed $\overline{\lambda}^\nu_i(u_-,u_+)$ defined as
\begin{equation} \label{Eq:VitesseChocModifiee}
\overline{\lambda}^\nu_i(u_-,u_+) = \overline{\lambda}_i(u_-,u_+) + \nu (w_i[u_+] -w_i[\overline{u}]).
\end{equation}
This small modification is motivated by the possible local linear degeneracy of the characteristic speed. 
This is a way to ``genuinely nonlinearize'' the characteristic speed in order to treat all waves
 in a unified manner. 
This will simplify a bit the forthcoming analysis. \par
Taking $\nu_0$ small enough, and modifying $\Lambda_1^{max}$ and $\Lambda_2^{min}$ if necessary,
 we can ensure that the modified shock speed satisfies, for all $\nu \in (0,\nu_0)$,
\begin{multline} \label{Eq:SeparationVitesses2}
\forall u_-,u_+ \in \overline{B}(\overline{u};r), \ \ 
\overline{\lambda}^\nu_1(u_-,u_+)  \leq \Lambda_1^{max} < 0 
                    \ \text{ if } \ (u_-,u_+) \ \text{ is a } 1\text{-wave}, \\
\ \text{ and } \ 
0 < \Lambda_2^{min} \leq \overline{\lambda}^\nu_2(u_-,u_+)
                    \ \text{ if } \ (u_-,u_+) \ \text{ is a } 2\text{-wave}.
\end{multline}

\end{itemize}
We will possibly modify these speeds again by a small amount in the next algorithm. \par
\ \par
\noindent
{\bf Construction of front-tracking approximations.}
The construction of the approximations is a follows. \par
\ \par
\noindent
{\bf 1.} {\it (Approximation of the initial state.)} 
According to Proposition~\ref{Pro:Classical} we first introduce 
a family of piecewise constant approximations $(u^\nu_0)_{\nu \in (0,\nu_0)}$ satisfying 
\begin{multline} \label{Eq:ApproxInitiale}
\forall \nu \in (0,\nu_0), \  V_p (u^\nu_0) \leq V_p (u_0), \ 
\| u^\nu_0 - \overline{u} \|_{\infty} \leq \| u_0 - \overline{u} \|_{\infty} 
\ \text{ and } \ 
 u^\nu_0 \longrightarrow u_0  \ \text{ in } \ L^1_{loc}(\R) \text{ as } \nu \rightarrow 0^+.
\end{multline}
For the rest of the construction, $\nu$ is fixed in $(0,\nu_0)$. \par
\ \par
\noindent
{\bf 2.} {\it (Initial solver.)} At each discontinuity $x_1,\ldots,x_N$ of $u^\nu_0$ 
we solve the corresponding (standard) Riemann problem  $(u^\nu_0(x_i^-), u^\nu_0(x_i^+))$. 
We let each outgoing shock wave travel at modified shock speed, and we approximate outgoing rarefaction waves
 by rarefaction fans as follows: given an $i$-rarefaction wave between
$\omega_-$ and $\omega_+=\mathcal{R}_i(\sigma,\omega_-)$, $\sigma>0$, we consider the states:
\begin{equation} \label{Eq:RarFan}
\omega^0:=\omega_-,\ \ \omega^k:=\mathcal{R}_i\left(\frac{k}{m},\omega_-\right)
\text{ for } k=1, \ldots, m:=\left\lfloor \frac{\sigma}{\nu} \right\rfloor,
\ \ \omega^{m+1}:= \omega_+,
\end{equation}
and replace the rarefaction wave $(\omega_-,\omega_+)$ by the fan given by the states $\omega^0,\ldots,\omega^{m+1}$
 separated by straight lines traveling at modified shock speed $\overline{\lambda}^\nu_i(\omega^k,\omega^{k+1})$. 
Note that due to monotonicity of the characteristic field, one has 
 $\overline{\lambda}_i(\omega^k,\omega^{k+1}) \geq \overline{\lambda}_i(\omega^{k-1},\omega^{k})$
 (see e.g. \eqref{Eq:ShockSpeedRar}).
Hence the convention \eqref{Eq:VitesseChocModifiee} ensures that
 $\overline{\lambda}^\nu_i(\omega^k,\omega^{k+1}) > \overline{\lambda}^\nu_i(\omega^{k-1},\omega^{k})$. \par
If necessary, we further modify the speeds to ensure that no triplets
 of such discontinuities meet at the same point and that all interaction times are distinct. 
Doing so, we can moreover require that no discontinuities meet unless they would have met without this modification anyway,
 that is, if two discontinuities $(u_l,u_m)$ (on the left) and $(u_m,u_r)$ (on the right), 
 then the original (modified shock) speed of $(u_l,u_m)$ is strictly larger than the one of $(u_m,u_r)$. 
To obtain this result, we accelerate by a small amount the leftmost involved front, in such a way that
\begin{equation} \label{Eq:ModifSpeed}
| s - \overline{\lambda}^\nu_i(u_-,u_+) | \leq \nu |u_+ - u_-|.
\end{equation} \par
\ \par
\noindent%
{\bf 3.} {\it (Solvers at interaction points.)} 
We call discontinuity lines {\it fronts} and points where two fronts meet {\it interaction points}. 
At each interaction point, we consider the interaction of two fronts separating the states 
$(u_l,u_m)$ (on the left) and $(u_m,u_r)$ (on the right).

The way we approximate the outgoing Riemann solution
depends on the nature of the incoming fronts.
\begin{itemize}
\item For interactions of {\it opposite families}
(that is, when the two incoming fronts correspond to waves of different characteristic family), 
we approximate outgoing waves with {\it single fronts of the same nature as the incoming one},
 whether their strength is less than $\nu$ or not. Since the two fronts can correspond to rarefaction, 
 compression or shock waves, we  describe the situation as follows.
We write $u_m=\mathcal{U}_a(\sigma_2;u_l)$ with 
$\mathcal{U}_a \in \left\{ \mathcal{R}_2, \mathcal{T}_2\right\}$
and $u_r=\mathcal{U}_b(\sigma_1;u_m)$ with 
$\mathcal{U}_b \in \left\{ \mathcal{R}_1, \mathcal{T}_1\right\}$.
Then we solve the outgoing Riemann problem in the form 
$u_r= \mathcal{U}_b(\sigma_2';\mathcal{U}_a(\sigma_1';u_l))$,
and let $\widetilde{u}_m:= \mathcal{U}_a(\sigma_1';u_l)$. 
We let two outgoing fronts separate $u_l$ and $\widetilde{u}_m$, and $\widetilde{u}_m$ and $u_r$,
 respectively, and travel at the modified shock speed described above.
\item For interactions of fronts of {\it the same family}, there are several possibilities. 
Call $i \in \{1,2\}$ this family.
Again, the two fronts correspond to rarefaction, compression or shock waves,
and we write $u_m=\mathcal{U}_a(\sigma_i;u_l)$ and $u_r=\mathcal{U}_b(\sigma'_i;u_m)$ 
 with $\mathcal{U}_a, \mathcal{U}_b \in \left\{ \mathcal{R}_i, \mathcal{S}_i\right\}$.
Then we choose the relevant Riemann problem according to the following rules:
\begin{itemize}
\item For the outgoing front of family $i$,
       if $\mathcal{S}_i \in \{\mathcal{U}_a,\mathcal{U}_b\}$,
    or if $\sigma_i \sigma_i'>0$ with $\sigma_i + \sigma_i' \leq -2{\nu}$,
    we let $\mathcal{U}_i=\mathcal{T}_i$ 
    and we otherwise let $\mathcal{U}_i=\mathcal{R}_i$,
\item For the outgoing front of family $3-i$, 
    if $C_*|\sigma_i \sigma_i'| (|\sigma_i |+|\sigma_i'|)\geq {\nu}$,
    we let $\mathcal{U}_{3-i}=\mathcal{T}_{3-i}$ 
    and we otherwise let $\mathcal{U}_{3-i}=\mathcal{R}_{3-i}$.
    Recall $C_*$ is a bound for all ``$\mathcal{O}(1)$'' in Proposition~\ref{Prop:InteractionSF}.
\end{itemize}
Then we solve the corresponding outgoing Riemann problem in the form 
 $u_r= \mathcal{U}_2(\widetilde{\sigma}_2;\mathcal{U}_1(\widetilde{\sigma}_1;u_l))$.
We let a single front separate the states bounding the wave of family $i$, 
 and approximate the outgoing wave of the other family by a single front 
 if it is shock or a compression wave, 
 by a rarefaction fan as in \eqref{Eq:RarFan} if it is a rarefaction.
 Again, we let each of these fronts travel at modified shock speed.
\end{itemize}
We introduce no front in the family $3-i$ if the strength of the corresponding wave is zero,
 but we extend the front of family $i$ even if it is of zero strength (in which case we call this 
 front {\it trivial}). \par
Again, we possibly modify the speeds of propagation of fronts to avoid three-fronts interactions 
 and to obtain that all interaction times are distinct, still requiring \eqref{Eq:ModifSpeed} 
 and the fact that all meetings of fronts would have taken place anyway. 
(Should the leftmost involved front be trivial, we accelerate the next one.) \par
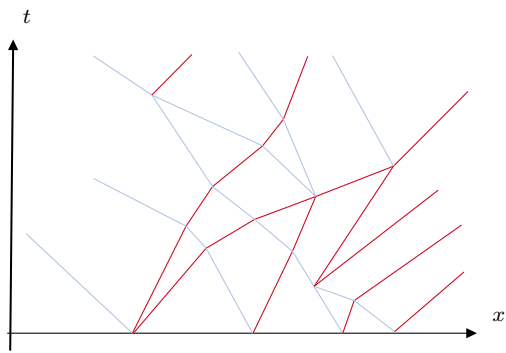
\begin{figure}[ht]
\centering
\begin{tikzpicture}[x=0.75pt,y=0.75pt,yscale=-1,xscale=1]
\draw    (98,250) -- (329.5,250) ;
\draw [shift={(332.5,250)}, rotate = 180] [fill={rgb, 255:red, 0; green, 0; blue, 0 }  ][line width=0.08]  [draw opacity=0] 
    (5.36,-2.57) -- (0,0) -- (5.36,2.57) -- cycle    ;
\draw [line width=0.75]    (99.5,258.75) -- (100.97,105.25) ;
\draw [shift={(101,102.25)}, rotate = 90.55] [fill={rgb, 255:red, 0; green, 0; blue, 0 }  ][line width=0.08]  [draw opacity=0] (5.36,-2.57) -- (0,0) -- (5.36,2.57) -- cycle    ;
\draw [color={rgb, 255:red, 208; green, 2; blue, 27 }  ,draw opacity=1 ]   (161,250) -- (197.29,207.29) -- (221.5,192.75) -- (252,181.25) ;
\draw [color={rgb, 255:red, 208; green, 2; blue, 27 }  ,draw opacity=1 ]   (160.5,250) -- (187.29,196.14) -- (200.5,176.25) -- (225.5,155.75) -- (235.86,142.43) -- (248,110.75) ;
\draw [color={rgb, 255:red, 208; green, 2; blue, 27 }  ,draw opacity=1 ]   (220.6,250.2) -- (240.4,208.8) -- (252,181.25) -- (290.8,166) -- (328,128.4) ;
\draw [color={rgb, 255:red, 177; green, 196; blue, 228 }  ,draw opacity=1 ]   (170.2,130.2) -- (184.87,152.49) -- (200.5,176.25) -- (221.5,192.75) -- (240.4,208.8) -- (265.5,250) ;
\draw [color={rgb, 255:red, 177; green, 196; blue, 228 }  ,draw opacity=1 ]   (170.2,130.2) -- (225.5,155.75) -- (252,181.25) ;
\draw [color={rgb, 255:red, 177; green, 196; blue, 228 }  ,draw opacity=1 ]   (213.57,108.71) -- (235.86,142.43) -- (252,181.25) ;
\draw [color={rgb, 255:red, 177; green, 196; blue, 228 }  ,draw opacity=1 ]   (141,172.14) -- (187.29,196.14) -- (197.29,207.29) -- (220.6,250.2) ;
\draw [color={rgb, 255:red, 177; green, 196; blue, 228 }  ,draw opacity=1 ]   (107.5,200) -- (160.5,250) ;
\draw [color={rgb, 255:red, 208; green, 2; blue, 27 }  ,draw opacity=1 ]   (324.8,195.6) -- (271.2,233.6) -- (265.5,250) ;
\draw [color={rgb, 255:red, 208; green, 2; blue, 27 }  ,draw opacity=1 ]   (290.8,166) -- (251.2,226.4) ;
\draw [color={rgb, 255:red, 177; green, 196; blue, 228 }  ,draw opacity=1 ]   (260.4,110.4) -- (290.8,166) ;
\draw [color={rgb, 255:red, 177; green, 196; blue, 228 }  ,draw opacity=1 ]   (141,110.6) -- (170.2,130.2) ;
\draw [color={rgb, 255:red, 208; green, 2; blue, 27 }  ,draw opacity=1 ]   (190.2,109.8) -- (170.2,130.2) ;
\draw [color={rgb, 255:red, 177; green, 196; blue, 228 }  ,draw opacity=1 ]   (251.2,226.4) -- (271.2,233.6) -- (291.4,249.2) ;
\draw [color={rgb, 255:red, 208; green, 2; blue, 27 }  ,draw opacity=1 ]   (313.2,178) -- (251.2,226.4) ;
\draw [color={rgb, 255:red, 208; green, 2; blue, 27 }  ,draw opacity=1 ]   (326,219.2) -- (291.4,249.2) ;
\draw (104.6,86.3) node [anchor=north west][inner sep=0.75pt]  [font=\footnotesize]  {$t$};
\draw (339,237.9) node [anchor=north west][inner sep=0.75pt]  [font=\footnotesize]  {$x$};
\end{tikzpicture}
\caption{A front-tracking approximation}
\label{Fig:WFTA}
\end{figure}
\begin{remark}
This algorithm can be seen as a limit case of the Bressan-Colombo algorithm~\cite{BressanColombo95}.
Indeed in this paper, the authors consider an approximate wave curve joining the rarefaction curve to
the shock curve by
\begin{equation*}
\mathcal{T}_i^\nu(s,u) = (1-\varphi(x/\sqrt{\nu})) \mathcal{R}_i(s,u) 
                            + \varphi(x/\sqrt{\nu}) \mathcal{S}_i(s,u) ,
\end{equation*}
where $\varphi$ is a smooth function equal to $1$ on $(-\infty,-2]$ and to $0$ on $[-1,+\infty)$. 
The algorithm above can be considered as a limit case, where we replace $\varphi$ by a Heaviside function. 
Indeed, using such a function $\varphi$ would yield poor estimates in Proposition~\ref{Prop:InteractionSF}.
The downside of using a Heaviside function is that this could naturally provoke an instability 
 in the pattern between compression waves and shocks. 
This motivates the additional rules that we enforced in the algorithm (such as the fact that a shock
 is always extended as a shock across interactions).
\end{remark}

\ \par
This allows to define an approximation $u^\nu$ up to a maximal time $T_\nu >0$ 
 where either the number of fronts becomes infinite,
 or interaction points accumulate or states leave the domain $U$. 
A representation of a front-tracking approximation is given in Figure~\ref{Fig:WFTA}.
We will prove later that for small enough $V_p(u_0)$, we actually have $T_\nu = +\infty$. \par 
%
%
%
%
%
%
%
%
%
\subsection{Time horizon, front lines and modified $p$-variation}
\label{Subsec:WaveTracing}
For each $\nu \in (0,\nu_0)$, we consider the corresponding front-tracking approximation $u^\nu$ 
on $[0,T_\nu) \times \R$. \par
\ \par
\noindent
{\bf Time horizon.} For the forthcoming analysis, it will be useful to introduce a {\it time horizon} $T \in (0,T_\nu)$
 so that we focus on the restriction on $u^\nu$ in $[0,T] \times \R$. 
Of course, $T$ is intended to be arbitrarily close to $T_\nu$.
The interest of such a time horizon is that at first we cannot rule out the possibility of an unbounded number of fronts
 or an accumulation of them as $t \rightarrow T_{\nu}^-$.
On $[0,T]$, we know that the number of fronts at play is finite.
The time horizon is central in the forthcoming Glimm-type functional. \par
\ \par
\noindent
{\bf Front lines.} 
An important feature of the above front-tracking algorithm (made possible by condition \eqref{Eq:Monotonicity}),
 is that each front $\alpha$ existing at time $t$ in $u^\nu$ can be uniquely followed forward in time until time $T_\nu$,
 by following the outgoing front of the same characteristic family at each interaction point.
There is indeed one (and only one) outgoing front of family $i$, when this family is represented among the incoming
 fronts.
We call such a continuous succession of fronts a {\it front line}.
Note that each front line is unequivocally associated with a characteristic family. \par
When restricted to a time interval $[0,T]$, maximal front lines (for inclusion) end up at time $T$
 and start either at $t=0$, or at an interaction point of fronts of the same family.
We will refer to the point where the front line was created as the front line's {\it birthplace}. 
When a front $\alpha'$ belongs to the front line of a front $\alpha$, we will say that $\alpha$ is an
 {\it ancestor} of $\alpha'$. \par
%
%
%
An important feature of these front lines is that:
\begin{itemize}
\item Front lines of the same characteristic family cannot cross: when they meet, they merge for further times
(actually a good image would be that they propagate side by side, keeping their left/right order).
\item Front lines of different characteristic family cross at most once: this is due to the strict separation of
characteristic speeds.
\end{itemize}
\ \par
\noindent
{\bf Ordering and families of fronts.} 
For each family $i \in \{1,2\}$, and at each time $t< T_\nu$, we can gather all $i$-fronts that cross $\{t \} \times \R$.
These fronts are naturally ordered from left to right.
Since front lines of the same family cannot cross, the order (in the large sense) of two front lines of the same family remains
 the same for all times $t < T_\nu$ at which they both exist.  
Given a front $\alpha$ existing at time $t$, we denote $x_\alpha(t)$ its position at that time as well.
For two fronts existing at time $t$, we will abusively write $\alpha < \beta$ for
$x_\alpha(t) < x_\beta(t)$. \par 
For $t \in [0,T]$, we call $\mathcal{A}^1(t)$ (respectively $\mathcal{A}^2(t)$) 
 the set of fronts of family $1$ (resp. of family $2$) existing at time $t$, and 
 $\mathcal{A}(t):=\mathcal{A}^1(t) \cup \mathcal{A}^2(t)$ the set of all fronts at time $t$.
Given a front $\gamma$ we will denote $\mathcal{A}^k_{<\gamma}(t)$
 (respectively $\mathcal{A}^k_{>\gamma}(t)$) the set of fronts of family $k$ existing at time $t$,
 strictly to the left (resp. right) of $\gamma$ at time $t$. 
 \par
\ \par
\noindent
{\bf Modified $p$-variation.} 
It will be helpful in order to measure the $p$-variation of front-tracking approximations, to introduce a slight 
 variant of \eqref{Eq:VpU}. Given a wave-front tracking approximation $u^\nu$ and $t \in [0,T_\nu)$, we can list 
 from left to right all $i$-fronts $\gamma^i_1,\ldots,\gamma^i_{n_i}$ that exist at time $t$. 
Then we set 
\begin{equation} \label{Eq:DefVtilde}
\widetilde{V}_p(u^\nu(t,\cdot)) := \Big( s_p^p(\sigma_{\gamma^1_1},\ldots,\sigma_{\gamma^1_{n_1}}) + s_p^p(\sigma_{\gamma^2_1},\ldots,\sigma_{\gamma^2_{n_2}}) \Big)^{1/p}.
\end{equation}
The difference with ${V}_p(u^\nu(t,\cdot))$ comes from the fact that $\widetilde{V}_p(u^\nu(t,\cdot))$ only measures $i$-waves 
in the $i$-th Riemann coordinate. 
We have the following elementary result, whose proof is delayed to Subsection~\ref{Subsec:ProofVVtilde}.
\begin{lemma} \label{Lem:VVtilde}
For some $C>0$, one has for any wave-front tracking approximation $u^\nu$ and $t \in (0,T_\nu)$,
\begin{equation} \label{Eq:VVtilde}
\left| \widetilde{V}_p(u^\nu(t,\cdot)) - V_p(u^\nu(t,\cdot)) \right| 
        \leq C \min \Big( V_p(u^\nu(t,\cdot))^3 , \widetilde{V}_p(u^\nu(t,\cdot))^3 \Big). 
\end{equation}
\end{lemma} 
%
%
%
%
%
%
%
%
%
%
%
%
%
%
\subsection{Wave measure curves}
\label{Subsec:APA}
In this paragraph, we introduce special curves in the $[0,T_\nu) \times \R$ strip, along which 
 we will measure the $p$-variation of the approximations $u^\nu$, or more precisely the maximal $p$-sums
 of the strengths of their fronts. \par
\ \par
\noindent
{\bf Lower wave measure curves.}
We will call {\it $1$-lower wave measure curve} a piecewise affine curve included in $[0,T_\nu) \times \R$
composed of three parts (from left to right): 
 a horizontal segment $(-\infty,a) \times \{t_{\max}\}$,
 a part of a $1$-front line between times $t_{\max}$ and $t_{\min}$ 
 (followed infinitesimally below all along), 
 and a horizontal segment $(b,+\infty) \times \{t_{\min}\}$, with $0 \leq t_{\min} \leq t_{\max} < T_\nu$.  \par
In the same way, a {\it $2$-lower wave measure curve} is a piecewise affine curve composed of three parts
 (from left to right): 
 a horizontal segment $(-\infty,a) \times \{t_{\min}\}$,
 a part of a $2$-front line between times $t_{\min}$ and $t_{\max}$ 
 (again followed below all along),
 and a horizontal segment $(b,+\infty) \times \{t_{\max}\}$, with $0 \leq t_{\min} \leq t_{\max} < T_\nu$. \par
\tikzset{every picture/.style={line width=0.75pt}}
\definecolor{myblue}{RGB}{181,202,234}  
\definecolor{myred}{RGB}{208,2,27}
\begin{figure}[ht]
\hspace{-0.5cm}
\begin{subfigure}{0.45\textwidth}
    \centering
    \begin{tikzpicture}[x=0.75pt,y=0.75pt,yscale=-1,xscale=1]
%
%
\draw    (111,250) -- (375,249.8) ;
\draw [shift={(378,249.8)}, rotate = 179.96] [fill={rgb, 255:red, 0; green, 0; blue, 0 }  ][line width=0.08]  [draw opacity=0] (5.36,-2.57) -- (0,0) -- (5.36,2.57) -- cycle    ;

\draw [line width=0.75]    (120,252.6) -- (120.2,104.8) ;
\draw [shift={(120.2,101.8)}, rotate = 90.08] [fill={rgb, 255:red, 0; green, 0; blue, 0 }  ][line width=0.08]  [draw opacity=0] 
	(5.36,-2.57) -- (0,0) -- (5.36,2.57) -- cycle    ;


\draw [color={myblue}, draw opacity=1 ]   (202.53,130.53) -- (229.17,141.06) -- (267.27,168.2) ;

\draw [color={myblue}, draw opacity=1 ]   (264.79,110.25) -- (292.67,160) ;

\draw [color={myblue}, draw opacity=1 ]   
	(129.52,127.71) -- (209.27,192.2) -- (219.62,196.48) -- (229.62,207.62) -- (253.35,249.75) ;

\draw [color={myblue}, draw opacity=1 ]   (129.81,184.86) -- (173.07,220.4) -- (192.73,249.75) ;

\draw [color={myblue}, draw opacity=1 ]   (305.04,110) -- (324.17,136) ;

\draw [color={myblue}, draw opacity=1 ]   (173.79,110) -- (202.53,130.53) ;

\draw [color={myblue}, draw opacity=1 ]   (283.53,226.73) -- (303.53,233.93) -- (323.73,249.53) ;

\draw [color={myblue}, draw opacity=1 ]   (220.81,149.86) -- (226.67,185.2) ;

\draw [color={myblue}, draw opacity=1 ]   
	(202.53,130.53) -- (243.47,174) -- (253.83,193.08) -- (276.75,208.83) -- (283.53,226.73) -- (297.73,249.5) ;


\draw [color={myred}, draw opacity=1 ]
	(192.73,249.75) -- (229.62,207.62) -- (253.83,193.08) -- (267.27,168.2) ;

\draw [color={myred}, draw opacity=1 ]   
	(153,250) -- (173,220.4) -- (209,192) -- (226.67,185) -- (243.47,174) -- (267.27,168) ;

\draw [color={myred}, draw opacity=1 ]
	(253,250) -- (277.58,208) -- (292.67,160) -- (324,136) -- (359.57,118.4) ;

\draw [color={myred}, draw opacity=1 ]   (360,192.5) -- (303.53,234) -- (297.73,249.5) ;

\draw [color={myred}, draw opacity=1 ]   (324,136) -- (283.5,226.73) ;

\draw [color={myred}, draw opacity=1 ]   (222.5,110.13) -- (202.53,130.53) ;

\draw [color={myred}, draw opacity=1 ]   (359.5,160.5) -- (283.53,226.73) ;

\draw [color={myred}, draw opacity=1 ]   (359.79,209.75) -- (323.73,249.53) ;

\draw [color={myred}, draw opacity=1 ]   (192.73,249.75) -- (219.62,196.48) -- (226.67,185.2) ;

\draw [color={myred}, draw opacity=1 ]   (267.27,168.2) -- (292.67,160) ;

\draw [color={myred}, draw opacity=1 ]   (249.79,110) -- (229.17,141.06) -- (220.81,149.86) ;


\draw  [dash pattern={on 3pt off 3pt}]  (114,120) -- (182,120) -- (199,131) -- (240,174) -- (252,195) -- (274,210) -- (281,228) -- (286,237) -- (378,237) ;


\draw (124,86) node [anchor=north west][inner sep=0.75pt]  [font=\footnotesize]  {$t$};
\draw (384,247) node [anchor=north west][inner sep=0.75pt]  [font=\footnotesize]  {$x$};
\draw (375,222) node [anchor=north west][inner sep=0.75pt]  [font=\footnotesize]  {$t_{\min}$};
\draw (87,106) node [anchor=north west][inner sep=0.75pt]  [font=\footnotesize]  {$t_{\max}$};

\end{tikzpicture}
\caption{A $1$-lower measure curve}
\label{Subfig:L1MC}        
\end{subfigure}
\hspace{1cm}
\begin{subfigure}{.45\textwidth}
    \centering
    \begin{tikzpicture}[x=0.75pt,y=0.75pt,yscale=-1,xscale=1]
%
%
\draw    (111,250) -- (375,249.8) ;
\draw [shift={(378,249.8)}, rotate = 179.96] [fill={rgb, 255:red, 0; green, 0; blue, 0 }  ][line width=0.08]  [draw opacity=0] 
	(5.36,-2.57) -- (0,0) -- (5.36,2.57) -- cycle    ;

\draw [line width=0.75]    (120,252.6) -- (120.2,104.8) ;
\draw [shift={(120.2,101.8)}, rotate = 90.08] [fill={rgb, 255:red, 0; green, 0; blue, 0 }  ][line width=0.08]  [draw opacity=0] 
	(5.36,-2.57) -- (0,0) -- (5.36,2.57) -- cycle    ;


\draw [color={myblue}, draw opacity=1 ]   
	(202.5,130.5) -- (243.5,174) -- (254,193) -- (276.75,209) -- (283.5,227) -- (298,249.5) ;

\draw [color={myblue}, draw opacity=1 ]   (202.53,130.53) -- (229.17,141.06) -- (267.27,168.2) ;

\draw [color={myblue}, draw opacity=1 ]   (264.79,110.25) -- (292.67,160) ;

\draw [color={myblue}, draw opacity=1 ]   
	(129.52,127.71) -- (209.27,192.2) -- (219.62,196.48) -- (229.62,207.62) -- (253.35,249.75) ;

\draw [color={myblue}, draw opacity=1 ]   (129.81,184.86) -- (173.07,220.4) -- (192.73,249.75) ;

\draw [color={myblue}, draw opacity=1 ]   (305.04,110) -- (324.17,136) ;

\draw [color={myblue}, draw opacity=1 ]   (173.79,110) -- (202.53,130.53) ;

\draw [color={myblue}, draw opacity=1 ]   (283.53,226.73) -- (303.53,233.93) -- (323.73,249.53) ;

\draw [color={myblue}, draw opacity=1 ]   (220.81,149.86) -- (226.67,185.2) ;


\draw [color={myred}, draw opacity=1 ]
	(192.73,249.75) -- (229.62,207.62) -- (253.83,193.08) -- (267.27,168.2) ;

\draw [color={myred}, draw opacity=1 ]   
	(153,250) -- (173,220.4) -- (209,192) -- (226.5,185) -- (243.5,174) -- (267,168) ;

\draw [color={myred}, draw opacity=1 ]   
	(253.5,250) -- (277.5,208) -- (293,160) -- (324,136) -- (359.5,118.5) ;

\draw [color={myred}, draw opacity=1 ]   (360.04,192.5) -- (303.53,233.93) -- (297.73,249.5) ;

\draw [color={myred}, draw opacity=1 ]   (324.17,136) -- (283.53,226.73) ;

\draw [color={myred}, draw opacity=1 ]   (222.53,110.13) -- (202.53,130.53) ;

\draw [color={myred}, draw opacity=1 ]   (359.54,160.5) -- (283.53,226.73) ;

\draw [color={myred}, draw opacity=1 ]   (359.79,209.75) -- (323.73,249.53) ;

\draw [color={myred}, draw opacity=1 ]   (192.73,249.75) -- (219.62,196.48) -- (226.67,185.2) ;

\draw [color={myred}, draw opacity=1 ]   (267.27,168.2) -- (292.67,160) ;

\draw [color={myred}, draw opacity=1 ]   (249.79,110) -- (229.17,141.06) -- (220.81,149.86) ;

\draw  [dash pattern={on 3pt off 3pt}]  (349,150.13) -- (308.75,150.38) -- (292.5,162.88) -- (245.75,176.13) -- (229.5,186.13) -- (222,197.88) -- (215.75,210) -- (112,210) ;

\draw (124.27,86.3) node [anchor=north west][inner sep=0.75pt]  [font=\footnotesize]  {$t$};
\draw (384.2,247.26) node [anchor=north west][inner sep=0.75pt]  [font=\footnotesize]  {$x$};
\draw (90.17,195) node [anchor=north west][inner sep=0.75pt]  [font=\footnotesize]  {$t_{\min}$};
\draw (350,137) node [anchor=north west][inner sep=0.75pt]  [font=\footnotesize]  {$t_{\max}$};

\end{tikzpicture}
    \caption{A $2$-lower measure curve}
\label{Subfig:L2MC}        
\end{subfigure}
\caption{Lower measure curves}
\label{Fig:LWC}
\end{figure}
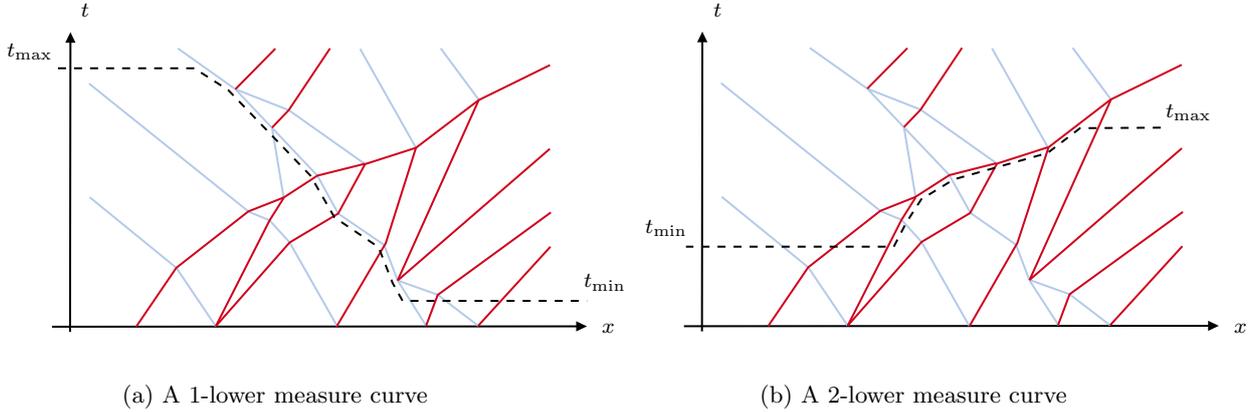
An important property of these curves is that a front line can cross a lower measure curve at most once.
This is due to the fact that front lines of different families can only cross once, and that
a front line meeting another front line of the same family stays in it. \par
%
We remark that thanks to \eqref{Eq:SeparationVitesses2}, $1$- and $2$- lower wave measure curves 
 are graphs of functions in the $(t,x)$ half-plane. 
Lower wave curves are represented in Figure~\ref{Fig:LWC}. \par
\ \par
\noindent
{\bf Upper wave measure curves.}
Now we introduce an ``upper'' version of the wave measure curves above.
This extension is slightly less straightforward than expected. 
Moreover an upper curve of family $i$ will only be used to measure waves of family $i$. \par
We will call {\it $1$-upper wave measure curve} a piecewise affine curve composed of three parts: 
 a horizontal segment $(-\infty,a) \times \{t_{\min}\}$  (from left to right),
 a part of a $1$-front line between times $t_{\min}$ and $t_{\max}$ 
 (followed infinitesimally above all along, from {\it right to left}), 
 and a horizontal segment $(b,+\infty) \times \{t_{\max}\}$, with $0 \leq t_{\min} \leq t_{\max} < T_\nu$.  \par
In the same way, {\it $2$-upper wave measure curve} is a piecewise affine curve composed of three parts: 
 a horizontal segment $(-\infty,a) \times \{t_{\max}\}$ (from left to right),
 a part of a $2$-front line between times $t_{\max}$ and $t_{\min}$ 
 (followed infinitesimally above all along, from {\it right to left}),
 and a horizontal segment $(b,+\infty) \times \{t_{\min}\}$, with $0 \leq t_{\min} \leq t_{\max} < T_\nu$. \par
\tikzset{every picture/.style={line width=0.75pt}} 
\begin{figure}[ht]
\hspace{-0.5cm}
\begin{subfigure}{0.45\textwidth}
    \centering
    \begin{tikzpicture}[x=0.75pt,y=0.75pt,yscale=-1,xscale=1]
%
%
\draw    (111,250) -- (375,249.8) ;
\draw [shift={(378,249.8)}, rotate = 179.96] [fill={rgb, 255:red, 0; green, 0; blue, 0 }  ][line width=0.08]  [draw opacity=0] 
	(5.36,-2.57) -- (0,0) -- (5.36,2.57) -- cycle    ;

\draw [line width=0.75]    (120,252.6) -- (120.2,104.8) ;
\draw [shift={(120.2,101.8)}, rotate = 90.08] [fill={rgb, 255:red, 0; green, 0; blue, 0 }  ][line width=0.08]  [draw opacity=0] 
	(5.36,-2.57) -- (0,0) -- (5.36,2.57) -- cycle    ;


\draw [color={myblue}, draw opacity=1 ]   
	(202.53,130.53) -- (243.47,174) -- (253.83,193.08) -- (276.75,208.83) -- (283.53,226.73) -- (297.73,249.5) ;

\draw [color={myblue}, draw opacity=1 ]   (202.53,130.53) -- (229.17,141.06) -- (267.27,168.2) ;

\draw [color={myblue}, draw opacity=1 ]   (264.79,110.25) -- (292.67,160) ;

\draw [color={myblue}, draw opacity=1 ]   
	(129.52,127.71) -- (209.27,192.2) -- (219.62,196.48) -- (229.62,207.62) -- (253.35,249.75) ;

\draw [color={myblue}, draw opacity=1 ]   (129.81,184.86) -- (173.07,220.4) -- (192.73,249.75) ;

\draw [color={myblue}, draw opacity=1 ]   (305.04,110) -- (324.17,136) ;

\draw [color={myblue}, draw opacity=1 ]   (173.79,110) -- (202.53,130.53) ;

\draw [color={myblue}, draw opacity=1 ]   (283.53,226.73) -- (303.53,233.93) -- (323.73,249.53) ;

\draw [color={myblue}, draw opacity=1 ]   (220.81,149.86) -- (226.67,185.2) ;


\draw [color={myred}, draw opacity=1 ]   
	(192.73,249.75) -- (229.62,207.62) -- (253.83,193.08) -- (267.27,168.2) ;

\draw [color={myred}, draw opacity=1 ]   
	(153,250) -- (173,220.5) -- (209,192) -- (226.5,185) -- (243.5,174) -- (267,168) ;

\draw [color={myred}, draw opacity=1 ]   
	(253.35,249.75) -- (277.58,208) -- (292.67,160) -- (324.17,136) -- (360,114) ;

\draw [color={myred}, draw opacity=1 ]   (360,192.5) -- (303.53,233.93) -- (297.73,249.5) ;

\draw [color={myred}, draw opacity=1 ]   (324.17,136) -- (283.53,226.73) ;

\draw [color={myred}, draw opacity=1 ]   (222.53,110.13) -- (202.53,130.53) ;

\draw [color={myred}, draw opacity=1 ]   (360,160.5) -- (283.53,226.73) ;

\draw [color={myred}, draw opacity=1 ]   (360,209.75) -- (323.73,249.53) ;

\draw [color={myred}, draw opacity=1 ]   (192.73,249.75) -- (219.62,196.48) -- (226.67,185.2) ;

\draw [color={myred}, draw opacity=1 ]   (267.27,168.2) -- (292.67,160) ;

\draw [color={myred}, draw opacity=1 ]   (249.79,110) -- (229.17,141.06) -- (220.81,149.86) ;

\draw  [dash pattern={on 3pt off 3pt}]  (368.29,119.29) -- (192.68,120.28) -- (205.83,129.99) -- (246.86,173.29) -- (255.43,189.86) -- (278,205.86) -- (286,224.14) -- (294.29,237.86) -- (114.86,237.29) ;

\draw (124,86) node [anchor=north west][inner sep=0.75pt]  [font=\footnotesize]  {$t$};
\draw (384,247) node [anchor=north west][inner sep=0.75pt]  [font=\footnotesize]  {$x$};
\draw (92,221.5) node [anchor=north west][inner sep=0.75pt]  [font=\footnotesize]  {$t_{\min}$};
\draw (372,107) node [anchor=north west][inner sep=0.75pt]  [font=\footnotesize]  {$t_{\max}$};

\end{tikzpicture}
\caption{A $1$-upper measure curve}
\label{Subfig:U1MC}        
\end{subfigure}
\hspace{1cm}
\begin{subfigure}{.45\textwidth}
    \centering
    \begin{tikzpicture}[x=0.75pt,y=0.75pt,yscale=-1,xscale=1]
%
%
\draw    (111,250) -- (375,249.8) ;
\draw [shift={(378,249.8)}, rotate = 179.96] [fill={rgb, 255:red, 0; green, 0; blue, 0 }  ][line width=0.08]  [draw opacity=0] 
	(5.36,-2.57) -- (0,0) -- (5.36,2.57) -- cycle    ;

\draw [line width=0.75]    (120,252.6) -- (120.2,104.8) ;
\draw [shift={(120.2,101.8)}, rotate = 90.08] [fill={rgb, 255:red, 0; green, 0; blue, 0 }  ][line width=0.08]  [draw opacity=0] 
	(5.36,-2.57) -- (0,0) -- (5.36,2.57) -- cycle    ;


\draw [color={myblue}, draw opacity=1 ]   
	(202.53,130.53) -- (243.47,174) -- (253.83,193.08) -- (276.75,208.83) -- (283.53,226.73) -- (297.73,249.5) ;

\draw [color={myblue}, draw opacity=1 ]   (202.53,130.53) -- (229.17,141.06) -- (267.27,168.2) ;

\draw [color={myblue}, draw opacity=1 ]   (264.79,110.25) -- (292.67,160) ;

\draw [color={myblue}, draw opacity=1 ]   
	(129.52,127.71) -- (209.27,192.2) -- (219.62,196.48) -- (229.62,207.62) -- (253.35,249.75) ;

\draw [color={myblue}, draw opacity=1 ]   (129.81,184.86) -- (173.07,220.4) -- (192.73,249.75) ;

\draw [color={myblue}, draw opacity=1 ]   (305.04,110) -- (324.17,136) ;

\draw [color={myblue}, draw opacity=1 ]   (173.79,110) -- (202.53,130.53) ;

\draw [color={myblue}, draw opacity=1 ]   (283.53,226.73) -- (303.53,233.93) -- (323.73,249.53) ;

\draw [color={myblue}, draw opacity=1 ]   (220.81,149.86) -- (226.67,185.2) ;


\draw [color={myred}, draw opacity=1 ]   (192.73,249.75) -- (229.62,207.62) -- (253.83,193.08) -- (267.27,168.2) ;

\draw [color={myred}, draw opacity=1 ]   
	(153,250) -- (173,220.5) -- (209,192) -- (226.5,185) -- (243.5,174) -- (267,168) ;

\draw [color={myred}, draw opacity=1 ]   
	(253.35,249.75) -- (277.58,208) -- (292.67,160) -- (324.17,136) -- (359.57,118.4) ;

\draw [color={myred}, draw opacity=1 ]   (360.04,192.5) -- (303.53,233.93) -- (297.73,249.5) ;

\draw [color={myred}, draw opacity=1 ]   (324.17,136) -- (283.53,226.73) ;

\draw [color={myred}, draw opacity=1 ]   (222.53,110.13) -- (202.53,130.53) ;

\draw [color={myred}, draw opacity=1 ]   (359.54,160.5) -- (283.53,226.73) ;

\draw [color={myred}, draw opacity=1 ]   (359.79,209.75) -- (323.73,249.53) ;

\draw [color={myred}, draw opacity=1 ]   (192.73,249.75) -- (219.62,196.48) -- (226.67,185.2) ;

\draw [color={myred}, draw opacity=1 ]   (267.27,168.2) -- (292.67,160) ;

\draw [color={myred}, draw opacity=1 ]   (249.79,110) -- (229.17,141.06) -- (220.81,149.86) ;

\draw  [dash pattern={on 3pt off 3pt}]  (108,144.86) -- (307,144) -- (290,158) -- (265,166) -- (252,190.29) -- (226,206) -- (212,223) -- (360,223) ;

\draw (124.27,86.3) node [anchor=north west][inner sep=0.75pt]  [font=\footnotesize]  {$t$};

\draw (384.2,247.26) node [anchor=north west][inner sep=0.75pt]  [font=\footnotesize]  {$x$};

\draw (366.45,213.97) node [anchor=north west][inner sep=0.75pt]  [font=\footnotesize]  {$t_{\min}$};

\draw (83,130) node [anchor=north west][inner sep=0.75pt]  [font=\footnotesize]  {$t_{\max}$};

\end{tikzpicture}
    \caption{A $2$-upper measure curve}
\label{Subfig:U2MC}        
\end{subfigure}
\caption{Upper measure curves}
\label{Fig:UWC}
\end{figure}
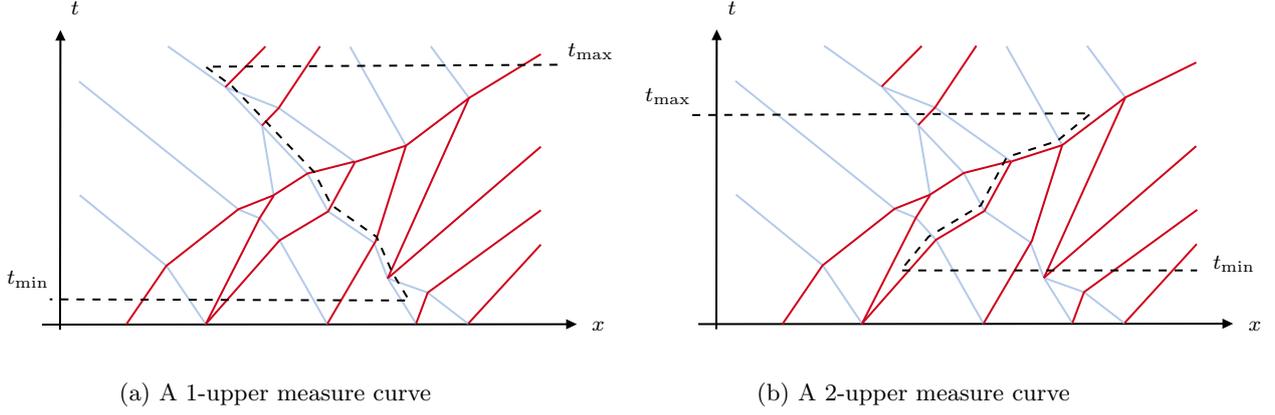
With this definition, upper measure curves are no longer graphs of functions.
Note that a front of family $i$ can only cross an $i$-upper measure curve only once, for the same reason as above.
This is not the case for fronts of family $3-i$, but as mentioned before,
 we will not measure $(3-i)$-fronts on such a curve. 
Upper wave curves are represented in Figure~\ref{Fig:UWC}. \par
\ \par
\noindent
{\bf Intermediate measure curves.} 
Given a lower or upper wave measure curve $\Gamma$ as above, we define intermediate measure curves $\Gamma_t$ 
 for times between $t=0$ and $t=t_{\max}$ as follows.
For $0 \leq t \leq t_{\min}$, $\Gamma_t$ is just the horizontal line $\{t\} \times \R$. 
For $t \in [t_{\min},t_{\max}]$, $\Gamma_t$ is composed of the part of $\Gamma$ in $[0,t] \times \R$ 
 (including the horizontal half-line at time $t_{\min}$)
 and of the horizontal half-line $\{t\} \times (-\infty,\Gamma(t)]$ (for $1$-lower and $2$-upper wave measure curves) 
 or $\{t\} \times [\Gamma(t),+\infty)$ (for $1$-upper and $2$-lower wave measure curves). 
In particular, $\Gamma_{t_{\max}}=\Gamma$. \par
%
%
Intermediate measure curves are illustrated in the case of family $1$ in Figure~\ref{Fig:IMC}. \par
\ \par
\noindent
{\bf Measuring the $p$-variation on curves.} Given a measure curve $\Gamma$ and $t \in [0,t_{\max}]$,
 we can measure the maximal $p$-sum of strengths of fronts that cross $\Gamma_t$. 
Call indeed $\gamma^1_1,\ldots,\gamma^1_{n_1}$ (resp. $\gamma^2_1,\ldots,\gamma^2_{n_2}$) the $1$-fronts
 (resp. $2$-fronts) crossing $\Gamma_t$, ordered by following the line from $t_{\min}$ to $t_{\max}$ as described above. 
Then we set
\begin{multline} \label{Eq:DefVpGamma}
\widetilde{V}_p[u^\nu;\Gamma](t) := \left( s_p(\sigma_{\gamma^1_1},\ldots,\sigma_{\gamma^1_{n_1}})^p 
                + s_p(\sigma_{\gamma^2_1},\ldots,\sigma_{\gamma^2_{n_2}})^p \right )^{\frac{1}{p}} 
\ \text{ for lower measure curves,} \\
\text{ and } \ 
\widetilde{V}_p[u^\nu;\Gamma](t) := s_p(\sigma_{\gamma^i_1},\ldots,\sigma_{\gamma^i_{n_i}})
\ \text{ for an } i-\text{upper measure curve.}
\end{multline}
\ \par
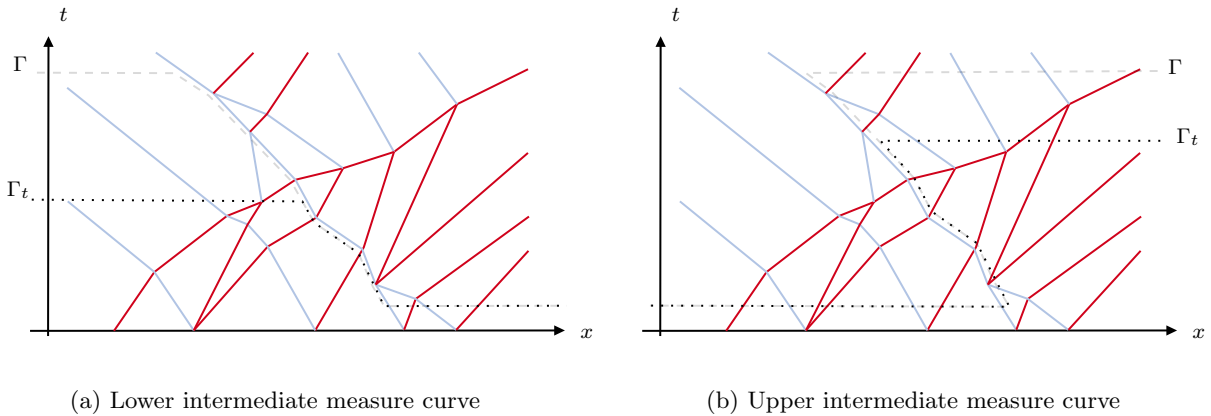
\begin{figure}[!ht]
\hspace{-0.5cm}
\begin{subfigure}{0.45\textwidth}
    \centering
    \begin{tikzpicture}[x=0.75pt,y=0.75pt,yscale=-1,xscale=1]

\draw    (111,250) -- (375,249.8) ;
\draw [shift={(378,249.8)}, rotate = 179.96] [fill={rgb, 255:red, 0; green, 0; blue, 0 }  ][line width=0.08]  [draw opacity=0] (5.36,-2.57) -- (0,0) -- (5.36,2.57) -- cycle    ;

\draw [line width=0.75]    (120,252.6) -- (120.2,104.8) ;
\draw [shift={(120.2,101.8)}, rotate = 90.08] [fill={rgb, 255:red, 0; green, 0; blue, 0 }  ][line width=0.08]  [draw opacity=0] (5.36,-2.57) -- (0,0) -- (5.36,2.57) -- cycle    ;

\draw [color={rgb, 255:red, 208; green, 2; blue, 27 }  ,draw opacity=1 ]   (192.73,249.75) -- (229.62,207.62) -- (253.83,193.08) -- (267.27,168.2) ;

\draw [color={rgb, 255:red, 208; green, 2; blue, 27 }  ,draw opacity=1 ]   (152.9,250) -- (173.07,220.4) -- (209.27,192.2) -- (226.67,185.2) -- (243.47,174) -- (267.27,168.2) ;

\draw [color={rgb, 255:red, 208; green, 2; blue, 27 }  ,draw opacity=1 ]   (253.35,249.75) -- (277.58,208) -- (292.67,160) -- (324.17,136) -- (359.57,118.4) ;

\draw [color={rgb, 255:red, 177; green, 196; blue, 228 }  ,draw opacity=1 ]   (202.53,130.53) -- (243.47,174) -- (253.83,193.08) -- (276.75,208.83) -- (283.53,226.73) -- (297.73,249.5) ;

\draw [color={rgb, 255:red, 177; green, 196; blue, 228 }  ,draw opacity=1 ]   (202.53,130.53) -- (229.17,141.06) -- (267.27,168.2) ;

\draw [color={rgb, 255:red, 177; green, 196; blue, 228 }  ,draw opacity=1 ]   (264.79,110.25) -- (292.67,160) ;

\draw [color={rgb, 255:red, 177; green, 196; blue, 228 }  ,draw opacity=1 ]   (129.52,127.71) -- (209.27,192.2) -- (219.62,196.48) -- (229.62,207.62) -- (253.35,249.75) ;

\draw [color={rgb, 255:red, 177; green, 196; blue, 228 }  ,draw opacity=1 ]   (129.81,184.86) -- (173.07,220.4) -- (192.73,249.75) ;

\draw [color={rgb, 255:red, 208; green, 2; blue, 27 }  ,draw opacity=1 ]   (360.04,192.5) -- (303.53,233.93) -- (297.73,249.5) ;

\draw [color={rgb, 255:red, 208; green, 2; blue, 27 }  ,draw opacity=1 ]   (324.17,136) -- (283.53,226.73) ;

\draw [color={rgb, 255:red, 177; green, 196; blue, 228 }  ,draw opacity=1 ]   (305.04,110) -- (324.17,136) ;

\draw [color={rgb, 255:red, 177; green, 196; blue, 228 }  ,draw opacity=1 ]   (173.79,110) -- (202.53,130.53) ;

\draw [color={rgb, 255:red, 208; green, 2; blue, 27 }  ,draw opacity=1 ]   (222.53,110.13) -- (202.53,130.53) ;

\draw [color={rgb, 255:red, 177; green, 196; blue, 228 }  ,draw opacity=1 ]   (283.53,226.73) -- (303.53,233.93) -- (323.73,249.53) ;

\draw [color={rgb, 255:red, 208; green, 2; blue, 27 }  ,draw opacity=1 ]   (359.54,160.5) -- (283.53,226.73) ;

\draw [color={rgb, 255:red, 208; green, 2; blue, 27 }  ,draw opacity=1 ]   (359.79,209.75) -- (323.73,249.53) ;

\draw [color={rgb, 255:red, 208; green, 2; blue, 27 }  ,draw opacity=1 ]   (192.73,249.75) -- (219.62,196.48) -- (226.67,185.2) ;

\draw [color={rgb, 255:red, 208; green, 2; blue, 27 }  ,draw opacity=1 ]   (267.27,168.2) -- (292.67,160) ;

\draw [color={rgb, 255:red, 177; green, 196; blue, 228 }  ,draw opacity=1 ]   (220.81,149.86) -- (226.67,185.2) ;

\draw [color={rgb, 255:red, 208; green, 2; blue, 27 }  ,draw opacity=1 ]   (249.79,110) -- (229.17,141.06) -- (220.81,149.86) ;

\draw [color={rgb, 255:red, 0; green, 0; blue, 0 }  ,draw opacity=0.15 ] [dash pattern={on 3pt off 3pt}]  (114.42,120.13) -- (182.92,120.38) -- (200.17,131.63) -- (240.92,174.13) -- (252.92,195.63) -- (274.92,210.88) -- (282.17,227.88) -- (287.76,237.55) -- (378.85,237.36) ;

\draw  [dash pattern={on 0.75pt off 3pt}]  (111.86,184) -- (246.92,184.88) -- (252.92,195.63) -- (274.92,210.88) -- (287.76,237.55) -- (378.85,237.36) ;

\draw (124.27,86.3) node [anchor=north west][inner sep=0.75pt]  [font=\footnotesize]  {$t$};

\draw (384.2,247.26) node [anchor=north west][inner sep=0.75pt]  [font=\footnotesize]  {$x$};

\draw (101,110.4) node [anchor=north west][inner sep=0.75pt]  [font=\footnotesize]  {$\Gamma$};

\draw (97.97,173.2) node [anchor=north west][inner sep=0.75pt]  [font=\footnotesize]  {$\Gamma_{t}$};

\end{tikzpicture}
\caption{Lower intermediate measure curve}
\label{Subfig:I1MC}        
\end{subfigure}
\hspace{1cm}
\begin{subfigure}{.45\textwidth}
    \centering
    \begin{tikzpicture}[x=0.75pt,y=0.75pt,yscale=-1,xscale=1]

\draw    (111,250) -- (375,249.8) ;
\draw [shift={(378,249.8)}, rotate = 179.96] [fill={rgb, 255:red, 0; green, 0; blue, 0 }  ][line width=0.08]  [draw opacity=0] (5.36,-2.57) -- (0,0) -- (5.36,2.57) -- cycle    ;

\draw [line width=0.75]    (120,252.6) -- (120.2,104.8) ;
\draw [shift={(120.2,101.8)}, rotate = 90.08] [fill={rgb, 255:red, 0; green, 0; blue, 0 }  ][line width=0.08]  [draw opacity=0] (5.36,-2.57) -- (0,0) -- (5.36,2.57) -- cycle    ;

\draw [color={rgb, 255:red, 208; green, 2; blue, 27 }  ,draw opacity=1 ]   (192.73,249.75) -- (229.62,207.62) -- (253.83,193.08) -- (267.27,168.2) ;

\draw [color={rgb, 255:red, 208; green, 2; blue, 27 }  ,draw opacity=1 ]   (152.9,250) -- (173.07,220.4) -- (209.27,192.2) -- (226.67,185.2) -- (243.47,174) -- (267.27,168.2) ;

\draw [color={rgb, 255:red, 208; green, 2; blue, 27 }  ,draw opacity=1 ]   (253.35,249.75) -- (277.58,208) -- (292.67,160) -- (324.17,136) -- (359.57,118.4) ;

\draw [color={rgb, 255:red, 177; green, 196; blue, 228 }  ,draw opacity=1 ]   (202.53,130.53) -- (243.47,174) -- (253.83,193.08) -- (276.75,208.83) -- (283.53,226.73) -- (297.73,249.5) ;

\draw [color={rgb, 255:red, 177; green, 196; blue, 228 }  ,draw opacity=1 ]   (202.53,130.53) -- (229.17,141.06) -- (267.27,168.2) ;

\draw [color={rgb, 255:red, 177; green, 196; blue, 228 }  ,draw opacity=1 ]   (264.79,110.25) -- (292.67,160) ;

\draw [color={rgb, 255:red, 177; green, 196; blue, 228 }  ,draw opacity=1 ]   (129.52,127.71) -- (209.27,192.2) -- (219.62,196.48) -- (229.62,207.62) -- (253.35,249.75) ;

\draw [color={rgb, 255:red, 177; green, 196; blue, 228 }  ,draw opacity=1 ]   (129.81,184.86) -- (173.07,220.4) -- (192.73,249.75) ;

\draw [color={rgb, 255:red, 208; green, 2; blue, 27 }  ,draw opacity=1 ]   (360.04,192.5) -- (303.53,233.93) -- (297.73,249.5) ;

\draw [color={rgb, 255:red, 208; green, 2; blue, 27 }  ,draw opacity=1 ]   (324.17,136) -- (283.53,226.73) ;

\draw [color={rgb, 255:red, 177; green, 196; blue, 228 }  ,draw opacity=1 ]   (305.04,110) -- (324.17,136) ;

\draw [color={rgb, 255:red, 177; green, 196; blue, 228 }  ,draw opacity=1 ]   (173.79,110) -- (202.53,130.53) ;

\draw [color={rgb, 255:red, 208; green, 2; blue, 27 }  ,draw opacity=1 ]   (222.53,110.13) -- (202.53,130.53) ;

\draw [color={rgb, 255:red, 177; green, 196; blue, 228 }  ,draw opacity=1 ]   (283.53,226.73) -- (303.53,233.93) -- (323.73,249.53) ;

\draw [color={rgb, 255:red, 208; green, 2; blue, 27 }  ,draw opacity=1 ]   (359.54,160.5) -- (283.53,226.73) ;

\draw [color={rgb, 255:red, 208; green, 2; blue, 27 }  ,draw opacity=1 ]   (359.79,209.75) -- (323.73,249.53) ;

\draw [color={rgb, 255:red, 208; green, 2; blue, 27 }  ,draw opacity=1 ]   (192.73,249.75) -- (219.62,196.48) -- (226.67,185.2) ;

\draw [color={rgb, 255:red, 208; green, 2; blue, 27 }  ,draw opacity=1 ]   (267.27,168.2) -- (292.67,160) ;

\draw [color={rgb, 255:red, 177; green, 196; blue, 228 }  ,draw opacity=1 ]   (220.81,149.86) -- (226.67,185.2) ;

\draw [color={rgb, 255:red, 208; green, 2; blue, 27 }  ,draw opacity=1 ]   (249.79,110) -- (229.17,141.06) -- (220.81,149.86) ;

\draw [color={rgb, 255:red, 0; green, 0; blue, 0 }  ,draw opacity=0.15 ] [dash pattern={on 3pt off 3pt}]  (368.29,119.29) -- (192.68,120.28) -- (205.83,129.99) -- (246.86,173.29) -- (255.43,189.86) -- (278,205.86) -- (286,224.14) -- (294.29,237.86) -- (114.86,237.29) ;

\draw  [dash pattern={on 0.75pt off 3pt}]  (369.6,154.4) -- (230,154.4) -- (246.86,173.29) -- (255.43,189.86) -- (278,205.86) -- (294.29,237.86) -- (114.86,237.29) ;

\draw (124.27,86.3) node [anchor=north west][inner sep=0.75pt]  [font=\footnotesize]  {$t$};
\draw (384.2,247.26) node [anchor=north west][inner sep=0.75pt]  [font=\footnotesize]  {$x$};
\draw (372,112) node [anchor=north west][inner sep=0.75pt]  [font=\footnotesize]  {$\Gamma$};
\draw (376.26,146.63) node [anchor=north west][inner sep=0.75pt]  [font=\footnotesize]  {$\Gamma_{t}$};

\end{tikzpicture}
    \caption{Upper intermediate measure curve}
\label{Subfig:I2MC}        
\end{subfigure}
\caption{Intermediate measure curves}
\label{Fig:IMC}
\end{figure}
\ \par
%
%
%
\noindent
{\bf Components of $(\R_+ \times \R) \setminus \Gamma$.} 
Wave measure curves cut the half-plane $\R_+ \times \R$ in two parts. 
We will refer to the component containing times larger than $t_{max}$ as the {\it unlimited component}. 
Note that for both lower and upper measure curves, this component contains the front line that the curve
 follows between $t_{\min}$ and $t_{\max}$, and that on the contrary, for all intermediate wave curves $\Gamma_t$,
 the part at time $t$ belongs to the other component of $\R_+ \times \R$.
%
%
%
%
%
%
%
%
%
%
%
%
\subsection{An oracle Glimm-type functional}
\label{Subsec:GlimmTypeFunctional}
As mentioned in the introduction, we construct a functional resembling Glimm's one,
 but using data from the future of the solution. The functional is constructed as follows. \par
\ \par
We let $\nu \in (0,\nu_0)$ and associate the approximation $u^{\nu}$.
First, we fix a time horizon $T \in (0,T_\nu)$ (again intended to be close to $T_\nu$).
All the forthcoming notions depend on this time horizon and on the restriction of $u^\nu$ on $[0,T]$. 
In many cases, the dependence on $T$, $t$ and $\nu$ will not be reflected on the notation, 
 for the sake of brevity. \par
\ \par
\noindent
{\bf Convention on interaction times.}
For interaction times, we will correspondingly distinguish times $t^-$ and $t^+$.
Given $\alpha$ a front crossing $\{t\} \times \R$, we write $\sigma_\alpha(t)$ for its strength at time $t$,
 taken immediately {\it before} the interaction in case of ambiguity. \par
\ \par
\noindent
{\bf Future interaction coefficients.}
Consider two fronts $\alpha$ (of family $1$) and $\beta$ (of family $2$) existing at time $t>0$
 such that their front lines cross between times $t$ and $T$.
We can associate the corresponding unique interaction point, say, separating $u_l$, $u_m$ and $u_r$,
  the $2$-front on the left separating $u_l$ from $u_m$ belonging to the front line of $\beta$
 (say, of strength $\sigma_2$), 
  the $1$-front on the right separating $u_m$ from $u_r$ belonging to the front line of $\alpha$
 (say, of strength $\sigma_1$). 
In that case, we define the {\it future interaction coefficients} of $\alpha$ and $\beta$ as follows:
\begin{equation} \label{Eq:FutureIntCoef}
\mathcal{C}^1_{\alpha \beta} = \mathcal{C}^1(u_m;0,\sigma_2)
\ \text{ and } \ 
\mathcal{C}^2_{\alpha \beta} = \mathcal{C}^2(u_m;\sigma_1,0)
\end{equation}
where the functions $\mathcal{C}^1$ and $\mathcal{C}^2$ were introduced in Proposition~\ref{Prop:InteractionDF}.
In the case where $\alpha$ and $\beta$ do no cross between times $t$ and $T$, we simply put 
\begin{equation*}
\mathcal{C}^1_{\alpha \beta} = \mathcal{C}^2_{\alpha \beta} =0.
\end{equation*}
We emphasize that these coefficients are those appearing at the time of the interaction; 
 they are determined in advance and used in the following functionals for times $t$ 
 before the interaction actually occurs (hence the expression ``oracle'' functional below). 
We also underline the fact that these coefficients depend not only on $t$ 
 but also on the time horizon $T$.
 \par
Note also that the coefficients $\mathcal{C}^i_{\alpha \beta}$ are independent of the fronts $\alpha$ (or $\beta$) 
 whose front lines have merged before the time of interaction. \par
Finally, we will write $\mathcal{C}^i_{\alpha\beta}$ and $\mathcal{C}^i_{\beta\alpha}$ interchangeably for two fronts $\alpha$ and $\beta$ of opposite families. \par
\ \par
\noindent
{\bf Preamplified strengths (of fronts).}
Let $t>0$. For a $1$-front $\alpha$ existing a time $t \leq T$, we define its preamplified strength as
 (using the notation of Proposition~\ref{Prop:InteractionDF} for $\sigma_*$):
\begin{equation} \label{Eq:PreampStr1}
\widehat{\sigma}_{\alpha}(t) 
:= {\sigma}_{\alpha}(t) \ 
\prod_{\beta \in \mathcal{A}_{< \alpha}^2(t)}
    (1 + \mathcal{C}^1_{\alpha\beta} (\sigma_\beta(t))_*^3).
\end{equation}
In the same way, for a $2$-front $\beta$ existing a time $t \leq T$,
 we define its preamplified strength as: \par
\begin{equation} \label{Eq:PreampStr2}
\widehat{\sigma}_{\beta}(t) 
:= {\sigma}_{\beta}(t) \ 
\prod_{\alpha \in \mathcal{A}_{>\beta}^1(t)}
(1 + \mathcal{C}^2_{\alpha \beta} (\sigma_\alpha(t))_*^3).
\end{equation}
We will use the following notation for the preamplification factor:
\begin{equation} \label{Eq:AmplificationFactor}
\mathcal{M}_{\alpha}(t) 
= \prod_{\beta \in \mathcal{A}_{< \alpha}^2(t)}
(1 + \mathcal{C}^1_{\alpha\beta} (\sigma_\beta(t))_*^3)
\ \text{ and } \ 
\mathcal{M}_{\beta}(t) =
\prod_{\alpha \in \mathcal{A}_{>\beta}^1(t)}
(1 + \mathcal{C}^2_{\alpha \beta} (\sigma_\alpha(t))_*^3).
\end{equation}
Notice that, while the interaction coefficients $\mathcal{C}^i_{\alpha\beta}$ are computed
 at the future time of the interaction, the strengths $\sigma_\alpha(t)$ and $\sigma_\beta(t)$
 are the ones at time $t$. \par
Remark also that due to the convention for future interaction coefficients,
 the indices $<\alpha$ and $>\beta$ on $\mathcal{A}^i$ are in fact not needed. \par
\ \par
\noindent
{\bf Total and partial strengths (of $u^\nu$).}
Given a time $t$, we describe the sets $\mathcal{A}^1(t)$ and $\mathcal{A}^2(t)$ 
 of $1$- and $2$- fronts in increasing order (that is, from left to right):
\begin{equation} \label{Eq:DefASets}
\mathcal{A}^1 = \{ \gamma^1_1, \ldots, \gamma^1_M \} 
\ \text{ and } \ 
\mathcal{A}^2 = \{ \gamma^2_1, \ldots, \gamma^2_N \},
\end{equation}
with ${\gamma^1_i} < {\gamma^1_j}$ and ${\gamma^2_i} < {\gamma^2_j}$ for meaningful $i<j$. \par
Then we measure the total (preamplified) strength of family $1$ and $2$ by
\begin{equation} \label{Eq:TotalStrengthes}
\mathcal{V}^1(t) = s^p_p(\widehat{\sigma}_{\gamma^1_1}(t), \ldots, \widehat{\sigma}_{\gamma^1_M}(t) )
\ \text{ and } \ 
\mathcal{V}^2(t) = s^p_p (\widehat{\sigma}_{\gamma^2_1}(t), \ldots, \widehat{\sigma}_{\gamma^2_N}(t) ).
\end{equation}
Given a front $\alpha$ of either family, we also introduce the {\it partial} strengths $\mathcal{V}^1(>\alpha)$ 
and $\mathcal{V}^2(>\alpha)$ as the maximal $p$-sums above limited to fronts strictly to the right of $\alpha$, 
so that for instance for fronts of family~$1$, 
$$
\mathcal{V}^1(>\gamma^1_i) 
   = s^p_p\big(\widehat{\sigma}_{\gamma^1_{i+1}}(t), \ldots, \widehat{\sigma}_{\gamma^1_M}(t) \big).
$$
We can define analogously $\mathcal{V}^i(< \alpha)$, $\mathcal{V}^i(\geq \alpha)$ and $\mathcal{V}^i(\leq \alpha)$. 
Clearly, apart from interaction times, if $\alpha$ is of family $k$, for the other family $3-k$ one has 
$\mathcal{V}^{3-k}(>\alpha) = \mathcal{V}^{3-k}(\geq \alpha)$. \par
\ \par
\noindent
{\bf Nonlocal strengths (of fronts).}
Given a front $\alpha$ of family $k$, 
we define its left (respectively right) {\it nonlocal strength} as follows:
\begin{equation} \label{Eq:NonLocalStrength}
\mathfrak{s}^k_l(\alpha) := \mathcal{V}^k( \leq \alpha) - \mathcal{V}^k( < \alpha) \ \ (\text{resp. }  
\mathfrak{s}^k_r(\alpha) := \mathcal{V}^k( \geq \alpha) - \mathcal{V}^k( > \alpha) ).
\end{equation}
We emphasize that, despite the notation, these strengths do not merely depend on $\alpha$, 
 but also on other fronts either on the left or on the right of $\alpha$,
 on the future interactions of $\alpha$, on $t$ and $T$. \par
Notice that
\begin{equation} \label{Eq:SommeForcesNonLocales}
\sum_{\alpha \in \mathcal{A}^k(t)} \mathfrak{s}^k_l(\alpha)
=\sum_{\alpha \in \mathcal{A}^k(t)} \mathfrak{s}^k_r(\alpha)
= \mathcal{V}^k(t).
\end{equation}
\ \par
\noindent
{\bf Interaction potentials.}
We can now define various interaction potentials associated with $u^\nu$ at time $t$. 
Precisely we let for $k \in \{ 1,2\}$
\begin{gather}
\label{Eq:Q1}
\mathcal{Q}_1^k(t):= \sum_{\substack{ {\alpha,\beta \in \mathcal{A}^k} \\ {\alpha < \beta} }} 
            \mathfrak{s}^k_r(\alpha) \, \mathfrak{s}^k_l(\beta), \\
%
\label{Eq:Q2}
\mathcal{Q}_2(t):= \sum_{\substack{ {\alpha \in \mathcal{A}^2} \\ 
               {\beta \in \mathcal{A}^1} \\ {\alpha < \beta } } }
\mathfrak{s}^k_r(\alpha) \, \mathfrak{s}^k_l(\beta), \\ 
\label{Eq:Q3}
\mathcal{Q}_3^k(t):= \sum_{\alpha \in \mathcal{A}^k} |{\sigma}_\alpha|^3 (1+\delta_{CW}(\alpha)),
\end{gather}
where $\delta_{CW}(\alpha)=1$ if $\alpha$ is a compression front 
 and $\delta_{CW}(\alpha)=0$ otherwise. \par
In these formulas, $\mathcal{Q}_1^k$ measures the potential interactions between waves of family $k$, 
 $\mathcal{Q}_2$ measures the potential interactions between waves of different families,
 and $\mathcal{Q}_3$ might not look like an interaction term, 
 but will play a key role for interactions of waves of the same family with opposite signs.
We will write in short
\begin{equation*}
\mathcal{V}:=\mathcal{V}^1 + \mathcal{V}^2, \ \  \mathcal{Q}_1 := \mathcal{Q}_1^1 + \mathcal{Q}_1^1,
            \ \text{ and } \   \mathcal{Q}_3 := \mathcal{Q}_3^1 + \mathcal{Q}_3^2.
\end{equation*}
\smallskip
\ \par
\noindent
{\bf Main Glimm-type functional.} 
We can now define the following functional:
\begin{equation} \label{Eq:DefGlimmFunc}
\Upsilon(t) := \mathcal{V}(t) + C_2 \big[ C_1 \mathcal{Q}_1(t) + \mathcal{Q}_2(t) +\mathcal{Q}_3(t) \big],
\end{equation}
for $C_1,C_2 \geq 1$ large enough. \par
\ \par
\noindent
{\bf Additional terms on measure curves.} 
We consider $\Gamma$ a $1-$ or $2-$, lower or upper, measure curve $\Gamma$
 (with its $t_{\min}$ and $t_{\max}$), in the strip $[0,T] \times \R$. 
We parameterize $\Gamma$ from left to right, as it is described in Subsection~\ref{Subsec:APA}. 
We associate to it the intermediate curves $\Gamma_t$ for $t \in [0,t_{\max}]$.
We now describe additional functionals to measure waves on $\Gamma_t$. \par
For each $i \in \{1,2\}$ and each $t \in [0,t_{\max}]$, we can obtain the list 
$\gamma^i_1,\ldots,\gamma^i_{n_i}$ 
 of all $i$-fronts that cross $\Gamma_t$, ordered according to the above parameterization. 
We then set:
\begin{equation} \label{Eq:ViGamma}
\mathcal{V}^i_\Gamma (t) = s_p^p(\widehat{\sigma}_{\gamma^i_1},\ldots,\widehat{\sigma}_{\gamma^i_{n_i}}),
\end{equation}
where the preamplified strengths $\widehat{\sigma}_{\gamma^i_j}$ written above are calculated at the time 
they cross $\Gamma_t$ (including their preamplification factor). 
In particular, all terms apart from those that correspond to an  intersection with (the horizontal line)
$\Gamma_t \setminus \Gamma$ at time $t$ are ``frozen'', that is, do not evolve for larger $t$. 
Then we set
\begin{equation} \label{Eq:VGamma}
\mathcal{V}_\Gamma (t) = 
\left\{ \begin{array}{l}
\mathcal{V}^1_\Gamma(t) + \mathcal{V}^2_\Gamma(t) \ \text{ if } \Gamma \text{ is a lower measure curve,} 
\medskip \\
\mathcal{V}^i_\Gamma(t) \ \text{ if } \Gamma \text{ is an } i \text{-upper measure curve.} 
\end{array} \right. 
\end{equation}
In other words, as mentioned before, we measure fronts of both families on lower curves, 
 but merely those of family $i$ on an $i$-upper curve. \par
\ \par
\noindent
We will also measure {\it variations} of the preamplification factors along $\Gamma$ by using intermediate measure curves.
However, we will measure preamplification factors on $\Gamma$ that are different from those of \eqref{Eq:AmplificationFactor} in two ways:
\begin{itemize}
\item we take the same convention as for $\mathcal{V}_\Gamma$ that they are calculated at the time the fronts crosses $\Gamma_t$
     and then stay frozen,
\item we only take into account future interaction coefficients on the unlimited component of $(\R_+ \times \R) \setminus \Gamma$
 and before $t_{\max}$.
\end{itemize}
More precisely, we define here for $\alpha$ a $1$-front and $\beta$ a $2$-front existing at time $t$ and $i \in \{1,2\}$,
\begin{equation} \label{Eq:DefCTilde}
\tilde{\mathcal{C}}^{i,\Gamma}_{\alpha\beta} = \left\{ \begin{array}{l}
\mathcal{C}^i_{\alpha\beta} \text{ if } \alpha \text{ and } \beta \text{ meet in the unlimited component of } (\R_+ \times \R) \setminus \Gamma \text{ at a time } \leq t_{\max}, \\
0 \text{ otherwise,}
\end{array} \right. 
\end{equation}
and
\begin{equation} \label{Eq:AmplificationFactorGamma}
\widetilde{\mathcal{M}}^\Gamma_{\alpha}(t) 
= \prod_{\beta \in \mathcal{A}^2(t)}
(1 + \tilde{\mathcal{C}}^{1,\Gamma}_{\alpha\beta} (\sigma_\beta(t))_*^3)
\ \text{ and } \ 
\widetilde{\mathcal{M}}^\Gamma_{\beta}(t) =
\prod_{\alpha \in \mathcal{A}^1(t)}
(1 + \tilde{\mathcal{C}}^{2,\Gamma}_{\alpha \beta} (\sigma_\alpha(t))_*^3).
\end{equation}
Then we let 
\begin{equation} \label{Eq:MiGamma}
\mathfrak{M}^i_\Gamma (t) 
= v_p^p \left(\widetilde{\mathcal{M}}^\Gamma_{\gamma^i_1}(t^i_1),\ldots,
                \widetilde{\mathcal{M}}^\Gamma_{\gamma^i_{n_i}}(t^i_{n_i})\right),
\end{equation}
where $t^i_j$ is the time of intersection of $\gamma^i_j$ with $\Gamma$. 
Finally we define
\begin{equation} \label{Eq:MGamma}
\mathfrak{M}_\Gamma (t) = 
\left\{ \begin{array}{l}
\mathfrak{M}^1_\Gamma(t) + \mathfrak{M}^2_\Gamma(t) \ \text{ if } \Gamma \text{ is a lower measure curve,} 
\medskip \\
\mathfrak{M}^i_\Gamma(t) \ \text{ if } \Gamma \text{ is a } i \text{-upper measure curve.} 
\end{array} \right. 
\end{equation}
%
%
%
%
%
%
%
\section{Proof of Theorem~\ref{Thm:Main}}
\label{Sec:ProofOfTheMainTheorem}
%
%
%
%
%
%
%
%
\subsection{Structure of the proof and induction hypothesis}
\ \par
As for Glimm's theorem, an important part of the proof consists in proving uniform-in-time estimates 
 in $\mathcal{W}_p(\R)$ for the wave-front tracking approximations $u^\nu$ 
 when $V_p(u_{0})$ is sufficiently small. 
It will follow that $T_\nu = +\infty$.
Then an additional step will be needed to obtain time-regularity estimates and deduce
 that the sequence $(u^\nu)$ is relatively compact in $L^1_{loc}(\R_+ \times \R)$. 
It will remain to prove that a limit is indeed an entropy solution of the system with initial condition $u_0$. \par
To prove the uniform estimate in $\mathcal{W}_p$, we will prove the decay 
 of the functional $\Upsilon$ for a given time horizon under a suitable induction assumption,
 and explain how to use this decay to obtain estimates of $u^\nu$ in $\mathcal{W}_p$ 
 (which is not immediate here). Then we will show how to propagate the induction hypothesis.
The model for the induction is the following elementary statement. 
\begin{quote}
{\it 
Let $f : [a,b) \rightarrow \R$ a càdlàg function. Suppose that for some $C>0$, one has
\begin{itemize}
\item $f(a) \leq C$,
\item for any $x \in [a,b]$, $f(x^+) - f(x^-) \leq C/2$,
\item for any $t \in (a,b)$,  $\big( f_{|[a,t)} \leq 2C \Rightarrow f_{|[a,t)} \leq C \big)$.
\end{itemize}
Then $f \leq C$ on $[a,b)$.}
\end{quote}
%
%
\ \par
The difficult assumption to check in our context is the equivalent of the last one. 
The general idea is that if for some time-horizon $T$ we have an a priori estimate on the $p$-variation of 
 the waves in $[0,T]$ (for each time and along the above family of curves), then studying the evolution of 
 these waves and their interactions on $[0,T]$, we can obtain a better estimate on their $p$-variation for 
 $V_p(u_0)$ small enough, and then are able to propagate an estimate uniformly 
 in time (independently of $\nu$). \par
Let us now state the assumption which will be at the core of our induction procedure. 
We recall the notations \eqref{Eq:DefVtilde} for $\widetilde{V}_p(u^\nu(t))$ and
 \eqref{Eq:DefVpGamma} for $\widetilde{V}_p[u^\nu;\Gamma](t)$. \par
\ \par
\noindent
{\bf Assumption $H(\tau)$.}
For $\tau \in (0,T_\nu)$, we will say that $u^\nu$ satisfies the assumption $H(\tau)$ if: 
 for any wave measure curve $\Gamma$ included the strip $[0,\tau] \times \R$,
 and for all $t \in [0,\tau]$, one has
\begin{equation} \label{Eq:H1tau}
    \widetilde{V}_p(u^\nu(t)) \leq 4 \, V_p(u_0) 
    \ \text{ and } \ 
    \widetilde{V}_p[u^\nu;\Gamma](t) \leq 4 \, V_p(u_0).
\end{equation} 
\ \par
The goal of the next two subsections is to establish the following.
\begin{proposition} \label{Pro:Induction}
There exists $c>0$ such that the following holds. 
Suppose that $V_p(u_0)<c$ and let $\nu \in (0,\nu_0)$.
Then the front-tracking approximation $u^\nu$ satisfies $H(\tau)$ for all $\tau \in (0,T_\nu)$. 
\end{proposition}
The core of the proof of Proposition~\ref{Pro:Induction} is divided in two main propositions, 
that intuitively manage the (largest part of) interactions of waves of opposite families and interactions within a family, 
respectively. This corresponds respectively to ``global'' and ``local'' (in time) estimates. \par 
\begin{proposition} \label{Pro:EstimeesVetQ}
There exist two constants $c,C>0$ such that the following is true.
Given a front-tracking approximation $u^\nu$ given as previously satisfying the assumption $H(\tau)$
 for some $\tau < T_{\nu}$, one has for any $t \in [0,\tau]$:
\begin{gather}
\label{Eq:EquivVVp}
\widetilde{V}_p^p(u^\nu(t,\cdot)) \left(1 -c V_p(u_0)^{3}\right) \leq \mathcal{V}(t) 
  \leq \widetilde{V}_p^p(u^\nu(t,\cdot)) \left(1 +C V_p(u_0)^{3}\right), \\
\label{Eq:QQuadratique}
 \mathcal{Q}_1(t) + \mathcal{Q}_2(t) +\mathcal{Q}_3(t)\leq C \, [\mathcal{V}(t)]^2.
\end{gather}
Moreover, for any measure curve $\Gamma$ in the strip $[0,\tau] \times \R$, one has, when the time horizon is $\tau$,
\begin{equation} \label{Eq:MGamma0}
\mathfrak{M}_\Gamma(0) \leq C V_p(u_0)^3.
\end{equation}
\end{proposition}
\begin{proposition} \label{Pro:DecayGlimm} 
There exists $C_1,C_2 \geq 1$ and $c_0>0$ such that the following is true.
Suppose that $V_p(u_0) < c_0$.
Consider a front-tracking approximation $u^\nu$ constructed by means 
 of the algorithm of  Subsection~\ref{Subsec:FTA}, and let $\tau < T_\nu$.
Under the assumption $H(\tau)$, the functional $\Upsilon$ associated with the time horizon $T=\tau$ satisfies
\begin{equation} \label{Eq:DecayGlimm}
\Upsilon(t) \ \text{ is non-increasing on } \ [0,\tau].
\end{equation}
Moreover, for all measure curve $\Gamma$ included in the strip $[0,\tau] \times \R$,
we also have, for the same time horizon~$\tau$,
\begin{equation} \label{Eq:DecayGlimm2}
\mathcal{V}_\Gamma(t) + \mathcal{Q}(t) 
\ \text{ and } \ 
\mathfrak{M}_\Gamma(t) + \mathcal{Q}(t) \ \text{ are non-increasing on } \ [0,\tau].
\end{equation}
\end{proposition}       
A central lemma which will be useful for both the proofs of Propositions~\ref{Pro:EstimeesVetQ}
and \ref{Pro:DecayGlimm} is the following.
\begin{lemma} \label{Lem:VpCoefsInteraction}
There exists a constant $C>0$ such that the following is true.
Consider a front-tracking approximation $u^\nu$, satisfying the assumption $H(\tau)$
 for some $\tau < T_{\nu}$.
Let $t \in [0,\tau]$ and consider an $i$-front $\alpha$ (for $i \in \{1,2\}$) existing at time $t$.
Then we have for the future interaction coefficients computed with time horizon $T=\tau$:
\begin{equation} \label{Eq:VpCoefsInteraction}
v_p \big[ (\mathcal{C}^{3-i}_{\alpha\beta})_{\beta \in \mathcal{A}^{3-i}(t)} \big]
\leq C.
\end{equation}
\end{lemma}
\noindent
Recall that $\mathcal{C}^{3-i}_{\alpha\beta}$ is taken as $0$ when $\alpha$ and $\beta$ do not interact in $[0,\tau]$. \par
\ \par
Lemma~\ref{Lem:VpCoefsInteraction} and Proposition~\ref{Pro:EstimeesVetQ} will be established
in Subsection~\ref{Subsec:DeduceVpEst} while
Proposition~\ref{Pro:DecayGlimm} will be proved in Subsection~\ref{Subsec:Decay}. \par
\ \par
For the next two subsections we consider a front-tracking approximation $u^\nu$ for fixed $\nu \in (0,\nu_0)$.
We let $\tau \in (0,T_\nu)$, and we suppose $H(\tau)$ to be satisfied. \par
As a first remark, we notice that under the assumption $H(\tau)$ and for sufficiently small $V_p(u_0)$,
 all preamplification factors are uniformly bounded in $[0,\tau]$, and more precisely satisfy for any 
 front $\alpha$ existing at time $t \in [0,\tau]$,
\begin{equation} \label{Eq:Mentre2et12}
\frac{1}{2} \leq \mathcal{M}_{\alpha}(t) \leq 2.
\end{equation}
This is a direct consequence of \eqref{Eq:AmplificationFactor}, $H(\tau)$, the boundedness of coefficients
 $\mathcal{C}^1$ and $\mathcal{C}^2$, and of
\begin{equation} \label{Eq:ProdExp}
\exp\left(-\sum_{i=1}^n |x_i| \right) \leq \prod_{i=1}^n (1+x_i) \leq \exp\left(\sum_{i=1}^n |x_i| \right) .
\end{equation}
\ \par
We emphasize that in Subsections~\ref{Subsec:DeduceVpEst} and \ref{Subsec:Decay}, the precise speed of propagation
 of fronts does not play a role apart from the strict separation of speeds of $1$- and $2$- fronts 
 (which involves in particular that fronts lines of different families can cross at most once),
 and the fact that front-lines of the same family merge when they meet.
The central property that we use is the fact that at an interaction point, outgoing fronts are based on one of
 the above Riemann problems. \par
%
%
%
%
%

%
%
%
%
\subsection{Global estimates: managing the preamplification factors}
\label{Subsec:DeduceVpEst}
In this section, we prove of Lemma~\ref{Lem:VpCoefsInteraction} and of Proposition~\ref{Pro:EstimeesVetQ}.
\subsubsection{Proof of Lemma~\ref{Lem:VpCoefsInteraction}} %
\label{ssub:PrLemCI}
\ \par
\noindent
{\bf 1.}
We extend the front $\alpha$ forward in time as a front line that we follow
 from time $t$ to time $\tau$. 
This extension is unique, and we call it ${a}$.
We associate to this front line ${a}$ an $i$-lower wave measure curve 
 $\Gamma_a^{\downarrow}$ and an $i$-upper wave measure curve $\Gamma_a^{\uparrow}$
 whose middle (non-horizontal) part corresponds to ${a}$ between time $t$
 and time $\tau$ (hence $t_{\min}=t$ and $t_{\max}=\tau$ for these curves). \par
To simplify the presentation, we suppose that $\alpha$ is a $1$-front, so that we consider 
 the coefficients $\mathcal{C}^2_{\alpha \beta}$ for $2$-fronts $\beta$ at time $t$. \par
Now we consider all $2$-fronts meeting ${a}$ between times $t$ and $\tau$.
These fronts are naturally ordered from right to left, that is, increasingly in terms 
 of their time of interaction with ${a}$.
We call them $\gamma_1$,\ldots,$\gamma_N$ (in that order). 
Note that each $\gamma_j$ meets $a$ once (due to separation of speeds between $\lambda_1$ and $\lambda_2$).
We call the corresponding future interaction coefficient $\mathcal{C}^2_{a\gamma_j}$. \par
\ \par
\noindent
{\bf 2.} We first show that 
\begin{equation} \label{Eq:CdeTversA}
v_p \big[ (\mathcal{C}^2_{\alpha\beta})_{\beta \in \mathcal{A}^{2}(t)} \big] 
\leq 
v_p (0,\mathcal{C}^2_{a\gamma_1},\ldots,\mathcal{C}^2_{a\gamma_N},0) .
\end{equation}
Call $\beta_1,\ldots,\beta_K$ all $2$-fronts at time $t$ to the left of $a$, ordered from left to right. 
Between $1$ and $K$, we can introduce $K_1$ as the largest integer so that the line fronts corresponding to fronts of
 index between $1$ and $K_1$ do not meet ${a}$ before time $\tau$.
If $K_1 = K$, then the left-hand side of \eqref{Eq:CdeTversA} is zero, since the convention is that
 $\mathcal{C}^2_{a\beta}=0$ when $a$ and the front line of $\beta$ do not cross.
If $K_1 < K$, since line fronts of the same family cannot cross,
 this involves that all line fronts for fronts between $\beta_{K_1+1}$ and $\beta_{K}$ included meet ${a}$ before time~$\tau$. \par
Now some of these line fronts $\beta_i$, $i=K_1+1,\ldots,K$, merge before reaching ${a}$, and some of the fronts
 $\gamma_i$, $i=1,\ldots,N$, may not have ancestors among $\beta_1,\ldots,\beta_K$, because they correspond 
 to fronts that were created after time $t$.
This involves that the sequence $(\mathcal{C}^2_{a\beta_{K_1+1}},\ldots,\mathcal{C}^2_{a\beta_K})$ can be obtained from 
the sequence $(\mathcal{C}^2_{a\gamma_1},\ldots,\mathcal{C}^2_{a\gamma_N})$ by forgetting some of the terms, 
and then repeating some of the coefficients.
Using Proposition~\ref{Pro:ElemePropsVp}--\ref{Item:Interleave}\&\ref{Item:Repetition}, we deduce that
\begin{equation*}
v_p(\mathcal{C}^2_{a\beta_{K_1+1}},\ldots,\mathcal{C}^2_{a\beta_K})
    \leq v_p(\mathcal{C}^2_{a\gamma_1},\ldots,\mathcal{C}^2_{a\gamma_N}).
\end{equation*}
Then \eqref{Eq:CdeTversA} is a consequence of  the fact that $\mathcal{C}^2_{\alpha\beta_k}=0$ for $k \leq K_1$
 and $\mathcal{C}^2_{\alpha\beta}=0$ for fronts $\beta$ to the right of $\alpha$ at time $t$. \par
\ \par
\noindent
{\bf 3.}
Now let us now prove that
\begin{equation} \label{Eq:CAdeVpA}
v_p (\mathcal{C}^2_{a\gamma_1},\ldots,\mathcal{C}^2_{a\gamma_N}) 
\lesssim \widetilde{V}_p(u^\nu;\Gamma^{\downarrow}_a)^p + \widetilde{V}_p(u^\nu;\Gamma^{\uparrow}_a)^p + \mathcal{O}(\nu^3).
\end{equation}
We consider an optimal partition generating $v_p (\mathcal{C}^2_{a\gamma_1},\ldots,\mathcal{C}^2_{a\gamma_N})$:
\begin{equation} \label{Eq:VpAOpt}
v_p^p (\mathcal{C}^2_{a\gamma_1},\ldots,\mathcal{C}^2_{a\gamma_N})
= \sum_{j=0}^{k-1} \left|\mathcal{C}^2_{a\gamma_{i_{j+1}}} - \mathcal{C}^2_{a\gamma_{i_{j}+1}} \right|^{p},
\end{equation}
with $i_1=0 < \dots < i_k=N$.
We consider $j \in \{1,\ldots,k-1\}$ and estimate  
    $\mathcal{C}^2_{a\gamma_{i_{j+1}}} - \mathcal{C}^2_{a\gamma_{i_{j}+1}}$.
Recall that the various coefficients $\mathcal{C}^2_{a \gamma_i}$ are computed precisely 
    at the moment when $\gamma_i$ crosses $a$. 
Call $u_m^i$ the value of $u^\nu$ on the left of $a$ and on the right of $\gamma_i$,
    at the interaction point between $a$ and $\gamma_i$, and call $\mathfrak{S}_i$ the strength
    of the $1$-front containing $a$ at this interaction (just before it).
With these notations we have $\mathcal{C}^2_{a\gamma_{i}} = \mathcal{C}^2(u_m^i;\mathfrak{S}_i,0)$.
Using the fact that $\mathcal{C}^2$ is Lipschitz, it follows that
\begin{equation*} 
|\mathcal{C}^2_{a\gamma_{i_{j+1}}} - \mathcal{C}^2_{a\gamma_{i_{j}+1}}|
\lesssim |u_{m}^{i_{j+1}} - u_{m}^{i_{j}+1}| + |\mathfrak{S}_{i_{j+1}} - \mathfrak{S}_{i_{j}+1}|.
\end{equation*}
We estimate separately the two terms in the right-hand side. \par
\ \par
\noindent
$\circ$ 
{\it The $|u_{m}^{i_{j+1}} - u_{m}^{i_{j}+1}|$ term.}
The difference between $u_{m}^{i_{j}+1}$ and $u_{m}^{i_{j+1}}$ is due to the fronts meeting ${a}$ 
    between the crossing points of ${a}$ with $\gamma_{i_{j}+1}$ and $\gamma_{i_{j+1}}$ (including the former),
    under ${a}$ (that is, on its left).
Call $b_1,\ldots,b_m$ these fronts, and call $\mathcal{B}^j_1$ (respectively $\mathcal{B}^j_2$) the set of indices
    $k \in \{1,\ldots,m\}$ for which $b_k$ is of family $1$ (resp. $2$).
Then we estimate $|u_{m}^{i_{j+1}} - u_{m}^{i_{j}+1}|$ by
\begin{equation*}
|u_{m}^{i_{j+1}} - u_{m}^{i_{j}+1}| \leq |w_1(u_{m}^{i_{j+1}}) - w_1(u_{m}^{i_{j}+1})|
+ |w_2(u_{m}^{i_{j+1}}) - w_2(u_{m}^{i_{j}+1})|.
\end{equation*}
Using a bound ``$\mathcal{O}(1)$'' for functions $\mathcal{D}_1$ and $\mathcal{D}_2$ in \eqref{Eq:ChocRC} 
we have that for $u \in U$ and $\sigma \in (-c,c)$ (introduced in Proposition~\ref{Prop:RiemannS}),
\begin{equation*}
w_1(\Phi_i(\sigma,u)) - w_1(u)  =
\left\{ \begin{array}{l} 
\sigma \text{ if } i =1, \\
\mathcal{O}(1) |\sigma|^3 \text{ if } i =2, 
\end{array} \right.
\text{ and } \ 
w_2(\Phi_i(\sigma,u)) - w_2(u) =
\left\{ \begin{array}{l} 
\mathcal{O}(1) |\sigma|^3 \text{ if } i =1, \\
\sigma \text{ if } i =2.
\end{array} \right.
\end{equation*}
We deduce that
\begin{multline*}
|w_1(u_{m}^{i_{j+1}}) - w_1(u_{m}^{i_{j}+1})| 
\leq \left|\sum_{k \in \mathcal{B}^j_1} \sigma_{b_k}\right| + \mathcal{O}(1) \sum_{k \in \mathcal{B}^j_2} |\sigma_{b_k}|^3
\\ \text{ and } \ 
|w_2(u_{m}^{i_{j+1}}) - w_2(u_{m}^{i_{j}+1})| 
\leq \left|\sum_{k \in \mathcal{B}^j_2} \sigma_{b_k}\right| + \mathcal{O}(1) \sum_{k \in \mathcal{B}^j_1} |\sigma_{b_k}|^3.
\end{multline*}
Consequently, using $p<3$, Proposition~\ref{Pro:ElemePropsVp}--\ref{Item:SpSpPrime},
$H(\tau)$ on $\Gamma_a^{\downarrow}$ and \eqref{Eq:1stSmallness}, we find
\begin{eqnarray} \nonumber
|u_{m}^{i_{j+1}} - u_{m}^{i_{j}+1}| 
&\lesssim& s_p(b_k; k \in \mathcal{B}^j_1) + s_p(b_k; k \in \mathcal{B}^j_2) 
        + s^3_p(b_k; k \in \mathcal{B}^j_1) + s^3_p(b_k; k \in \mathcal{B}^j_2) \\
\nonumber 
&\lesssim& s_p(b_k; k \in \mathcal{B}^j_1) + s_p(b_k; k \in \mathcal{B}^j_2).
\end{eqnarray} \par
\ \par
\noindent
$\circ$ {\it The $|\mathfrak{S}_{i_{j+1}} - \mathfrak{S}_{i_{j}+1}|$ term.}
Here the evolution of $\mathfrak{S}$ between indices $i_{j+1}$ and $i_{j}+1$ is due to all fronts interacting
 with ${a}$ on both sides (since there can be fronts of family $1$ joining ${a}$ from the right).
So changing a bit the notations we call here $\widehat{\mathcal{B}}^j_1:=\{d_1,\ldots,d_K\}$ the fronts 
 of family $1$ on both sides (ordered as before)
 that join ${a}$ between the crossing points of ${a}$ with $\gamma_{i_j+1}$ and with $\gamma_{i_{j+1}}$.
We keep the notation  $\gamma_{i_j+1},\ldots,\gamma_{i_{j+1}-1}$ for fronts of family~$2$ between these points;
 we note $\widehat{\mathcal{B}}^j_2:=\{i_j+1,\ldots,i_{j+1}-1\}$. 
The situation is represented in Figure~\ref{Fig:FollowingA}. \par
In what follows we will use the following notation: for $k \in \{1,\ldots,K\}$
 and $p \in \widehat{\mathcal{B}}^j_2$, $d_k < \gamma_p$ means that the front $d_k$ meets $a$ 
 before the front $\gamma_p$. 
Finally, for $k \in \widehat{\mathcal{B}}^j_1$ we call $\mathfrak{e}_k$ the ``$\mathcal{O}(1)$'' error term of
 Proposition~\ref{Prop:InteractionSF} for the outgoing wave of the family $1$
 at the interaction point between ${a}$ and $d_k$ (which are $1$-fronts). \par 
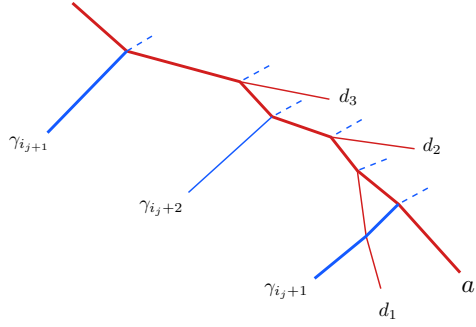
\begin{figure}[ht]
\centering
\resizebox{.4\linewidth}{!}{\tikzset{every picture/.style={line width=0.75pt}}
\begin{tikzpicture}[x=0.75pt,y=0.75pt,yscale=-1,xscale=1]

\draw [color={rgb, 255:red, 210; green, 30; blue, 30 } ,draw opacity=1 ][line width=1.5] (473,220) -- (431,173) -- (403,150) -- (385,127) -- (345,113) -- (323,89) -- (246,68) -- (209,35) ;

\draw [color={rgb, 255:red, 210; green, 30; blue, 30 } ,draw opacity=1 ]   (385,127) -- (442,135) ;
\draw [color={rgb, 255:red, 210; green, 30; blue, 30 } ,draw opacity=1 ]   (419,231) -- (409,195) -- (403,150) ;
\draw [color={rgb, 255:red, 210; green, 30; blue, 30 } ,draw opacity=1 ]   (323,89) -- (384,101) ;

\draw [color={rgb, 255:red, 20; green, 90; blue, 255 }  ,draw opacity=1 ][line width=1.5]    (431,173) -- (409,195) -- (373.91,224) ;
\draw [color={rgb, 255:red, 20; green, 90; blue, 255 }  ,draw opacity=1 ][line width=1.5]    (246,68) -- (192,124) ;
\draw [color={rgb, 255:red, 20; green, 90; blue, 255 }  ,draw opacity=1 ]   (345,113) -- (288,165) ;
\draw [color={rgb, 255:red, 20; green, 90; blue, 255 }  ,draw opacity=1 ] [dash pattern={on 3pt off 3pt}]  (246,68) -- (266,56) ;
\draw [color={rgb, 255:red, 20; green, 90; blue, 255 }  ,draw opacity=1 ] [dash pattern={on 3pt off 3pt}]  (365,102) -- (345,113) ;
\draw [color={rgb, 255:red, 20; green, 90; blue, 255 }  ,draw opacity=1 ] [dash pattern={on 3pt off 3pt}]  (451,162) -- (431,173) ;
\draw [color={rgb, 255:red, 20; green, 90; blue, 255 }  ,draw opacity=1 ] [dash pattern={on 3pt off 3pt}]  (423,142) -- (403,150) ;
\draw [color={rgb, 255:red, 20; green, 90; blue, 255 }  ,draw opacity=1 ] [dash pattern={on 3pt off 3pt}]  (406,115) -- (385,127) ;
\draw [color={rgb, 255:red, 20; green, 90; blue, 255 }  ,draw opacity=1 ] [dash pattern={on 3pt off 3pt}]  (343,78) -- (323,89) ;

\draw (485,224) node [anchor=north east][inner sep=0.75pt] [xscale=1.3, yscale=1.3]{$a$};
\draw (371,226) node [anchor=north east][inner sep=0.75pt]  {$\gamma_{{i_j}+1}$};
\draw (195,125) node [anchor=north east][inner sep=0.75pt]  {$\gamma_{i_{j+1}}$};
\draw (286,169) node [anchor=north east][inner sep=0.75pt]  {$\gamma_{{i_j}+2}$};
\draw (432,239) node [anchor=north east][inner sep=0.75pt]  {$d_{1}$};
\draw (462,126) node [anchor=north east][inner sep=0.75pt]  {$d_{2}$};
\draw (405,92) node [anchor=north east][inner sep=0.75pt]  {$d_{3}$};
\end{tikzpicture}}
\caption{Following the front-line $a$}
\label{Fig:FollowingA}
\end{figure}
Then an induction shows the following Duhamel-type formula: for $\ell \geq i_j+1$,
\begin{multline} \label{Eq:Duhamel}
\mathfrak{S}_\ell 
= \mathfrak{S}_{i_j+1}
    \prod_{\substack{p \in \widehat{\mathcal{B}}^j_2 \\ p < \ell}}
    \left(1+ \mathcal{C}^1(u_m^p;\mathfrak{S}_p,\sigma_{\gamma_p}) (\sigma_{\gamma_p})_*^3\right) \\
+ \sum_{\substack{k \in \widehat{\mathcal{B}}^j_1 \\ \text{s.t. } d_k < \gamma_\ell}}  ( \sigma_{d_k} + \mathfrak{e}_k)
    \prod_{\substack{p \in \widehat{\mathcal{B}}^j_2  \\ p <\ell \\ \text{s.t. } d_k < \gamma_p}}
    \left(1+ \mathcal{C}^1(u_m^p;\mathfrak{S}_p,\sigma_{\gamma_p}) (\sigma_{\gamma_p})_*^3\right) .
\end{multline}
This allows to estimate $\mathfrak{S}_{i_{j+1}} - \mathfrak{S}_{i_{j}+1}$ as follows.
We use \eqref{Eq:Duhamel} with $\ell = {i_{j+1}}$.
We estimate the first term by using
\begin{equation*}
\left|\prod_{\substack{p \in \widehat{\mathcal{B}}^j_2}}
    \left(1+ \mathcal{C}^1(u_m^p;\mathfrak{S}_p,\sigma_{\gamma_p}) (\sigma_{\gamma_p})_*^3\right) - 1 \right|
    \lesssim \sum_{\substack{p \in \widehat{\mathcal{B}}^j_2}} (\sigma_{\gamma_p})_*^3,
\end{equation*}
and deduce that
\begin{equation*}
\left|\mathfrak{S}_{i_j+1} 
    \prod_{p \in \widehat{\mathcal{B}}^j_2 }
    \left(1+ \mathcal{C}^1(u_m^p;\mathfrak{S}_p,\sigma_{\gamma_p}) (\sigma_{\gamma_p})_*^3\right) 
    - \mathfrak{S}_{i_j+1} \right| 
\lesssim \left|\mathfrak{S}_{i_j+1} \right|
    \ \sum_{\substack{p \in \widehat{\mathcal{B}}^j_2}} (\sigma_{\gamma_p})_*^3
\lesssim \sum_{\substack{p \in \widehat{\mathcal{B}}^j_2}} (\sigma_{\gamma_p})_*^3.
\end{equation*}
For the second term we will use Corollary~\ref{Cor:MultiplSp}. 
Calling
\begin{equation} \label{Eq:CoefPreamp}
m_k =  \prod_{\substack{p \in \widehat{\mathcal{B}}^j_2  \\ p <\ell \\ \text{s.t. } d_k < \gamma_p}}
    \left(1+ \mathcal{C}^1(u_m^p;\mathfrak{S}_p,\sigma_{\gamma_p}) (\sigma_{\gamma_p})_*^3\right),
\end{equation}
we have, using a uniform bound on $\mathcal{C}^1$, Proposition~\ref{Pro:ElemePropsVp}--\ref{Item:SpSpPrime}
 and $H(\tau)$ on $\Gamma_a^{\downarrow}$,
\begin{equation} \label{Eq:SpCoefPreamp}
v_p(m_k; k \in \widehat{\mathcal{B}}^j_1)
  \leq v_1(m_k; k \in \widehat{\mathcal{B}}^j_1) 
  \lesssim \sum_{\substack{p \in \widehat{\mathcal{B}}^j_2}} (\sigma_{\gamma_p})_*^3
\ \text{ and } \ 
\| m_k\|_\infty \lesssim 1.
\end{equation}
On the other hand, one has $s_p(\sigma_{d_k} + \mathfrak{e}_k) \leq s_p(\sigma_{d_k}) + s_p(\mathfrak{e}_k)$.
Moreover, one has $|\mathfrak{e}_k| \lesssim |\sigma_{d_k}|^3$, with the only possible exception
 that $a$ at this interaction is a CW turning into a shock, with $\sigma_{d_k} < |\sigma_a|$ where 
 $\sigma_a$ is the strength of the front carried by $a$ at that time. 
But this can happen only once (since a shock never turns in a CW), and in that case $|\mathfrak{e}_k| \lesssim \nu^3$
 because all CW are of size $\mathcal{O}(\nu)$ (see the proof p.~\pageref{Place:EstCW} that is valid under $H(\tau)$).
It follows that, using again Proposition~\ref{Pro:ElemePropsVp}--\ref{Item:SpSpPrime},
\begin{equation*}
s_p(\mathfrak{e}_k, k \in \widehat{\mathcal{B}}^j_1) \leq s_1(\mathfrak{e}_k, k \in \widehat{\mathcal{B}}^j_1)
    \leq s_1(|\sigma_{d_k}|^3 , k \in \widehat{\mathcal{B}}^j_1) +\mathcal{O}(\nu^3)
    \leq s_p(\sigma_{d_k}, k \in \widehat{\mathcal{B}}^j_1)^3 +\mathcal{O}(\nu^3).
\end{equation*}
Hence Corollary~\ref{Cor:MultiplSp} gives:
\begin{eqnarray} \nonumber 
\left|\sum_{k \in \widehat{\mathcal{B}}^j_1} (\sigma_{d_k} + \mathfrak{e}_k)
    \prod_{\substack{p \in \widehat{\mathcal{B}}^j_2  \\ p <\ell \\ \text{s.t. } d_k < \gamma_p}}
    \left(1+ \mathcal{C}^1(u_m^p;\mathfrak{S}_p,\sigma_{\gamma_p}) (\sigma_{\gamma_p})_*^3\right) \right|
&\lesssim& \left(s_p(\sigma_{d_k}; k \in \widehat{\mathcal{B}}^j_1) +\mathcal{O}(\nu^3) \right)
   \left( 1 + \sum_{\substack{p \in \widehat{\mathcal{B}}^j_2}} (\sigma_{\gamma_p})_*^3 \right) \\
&\lesssim& s_p(\sigma_{d_k}; k \in \widehat{\mathcal{B}}^j_1) +\mathcal{O}(\nu^3), 
\end{eqnarray}
where we used again $H(\tau)$ on $\Gamma_a^{\downarrow}$. \par
\ \par
\noindent
$\circ$ {\it Back to the proof of \eqref{Eq:CAdeVpA}.}
Gathering the estimates above we deduce 
\begin{eqnarray*}
|\mathcal{C}^2_{a\gamma_{i_{j+1}}} - \mathcal{C}^2_{a\gamma_{i_{j}+1}}| 
&\lesssim& s_p(\sigma_{b_k}; k \in \mathcal{B}^j_1) + s_p(\sigma_{b_k}; k \in \mathcal{B}^j_2)
    + s_p(\sigma_{d_k}; k \in \widehat{\mathcal{B}}^j_1) 
    + \sum_{\substack{p \in \widehat{\mathcal{B}}^j_2}} (\sigma_{\gamma_p})_*^3 +\mathcal{O}(\nu^3)\\
&\lesssim& s_p(\sigma_{b_k}; k \in \mathcal{B}^j_1) + s_p(\sigma_{b_k}; k \in \mathcal{B}^j_2)
    + s_p(\sigma_{d_k}; k \in \widehat{\mathcal{B}}^j_1) +\mathcal{O}(\nu^3),
\end{eqnarray*}
where the error $\mathcal{O}(\nu^3)$ can only occur once along $a$.
Raising to the power $p$, summing over $j$ and using Proposition~\ref{Pro:ElemePropsVp}-\ref{Item:VpFusion} 
 and \eqref{Eq:VpAOpt},b we arrive at
\begin{eqnarray*}
v_p^p (\mathcal{C}^2_{a\gamma_1},\ldots,\mathcal{C}^2_{a\gamma_N})
&\lesssim& \mathcal{O}(\nu^3) + \sum_{j=0}^{k-1} \left(s^p_p(\sigma_{b_k}; k \in \mathcal{B}^j_1) + s^p_p(\sigma_{b_k}; k \in \mathcal{B}^j_2)
    + s^p_p(\sigma_{\gamma_k}; k \in \widehat{\mathcal{B}}^j_1) \right) \\
& \lesssim & \widetilde{V}_p(u^\nu;\Gamma^{\downarrow}_a)^p + \widetilde{V}_p(u^\nu;\Gamma^{\uparrow}_a)^p + \mathcal{O}(\nu^3),
\end{eqnarray*}
that is \eqref{Eq:CAdeVpA}. \par
\ \par
\noindent
{\bf 4.} 
Using $H(\tau)$, we infer 
$v_p^p (\mathcal{C}^2_{a\gamma_1},\ldots,\mathcal{C}^2_{a\gamma_N}) \lesssim V_p(u_0)^p + \mathcal{O}(\nu^3)$.
Going back to \eqref{Eq:CdeTversA} and using Proposition~\ref{Pro:ElemePropsVp}--\ref{Item:Ajoutde0},
 we obtain \eqref{Eq:VpCoefsInteraction}. 
This ends the proof of Lemma~\ref{Lem:VpCoefsInteraction}.
\hfill \qedsymbol \par
\ \par
\noindent
\subsubsection{Proof of Proposition~\ref{Pro:EstimeesVetQ}} 
\label{ssub:PrProVQ}
\ \par
\noindent
{\bf Proof of \eqref{Eq:EquivVVp}.} 
We first consider $\mathcal{V}^1(t)$. Call $\beta_1,\ldots,\beta_K$ the list of all $1$-fronts at time $t$, from left
 to right, and $\alpha_1,\ldots,\alpha_N$ the list of all $2$-fronts at time $t$, from left
 to right.
For any $i \in \{1,\ldots,N\}$ and $k \in \{1,\ldots,K\}$, we can introduce the front lines
 starting from $\alpha_i$ and $\beta_k$ at time $t$. \par
Now we estimate 
$\mathcal{V}^1(t)
 = s^p_p( \widehat{\sigma}_{\beta_1},\ldots, \widehat{\sigma}_{\beta_K} )
= s^p_p( \mathcal{M}_{\beta_1} \sigma_{\beta_1},\ldots, \mathcal{M}_{\beta_K} \sigma_{\beta_K} )$. 
To that purpose, we estimate $v_p( \mathcal{M}_{\beta_1} ,\ldots, \mathcal{M}_{\beta_K} )$.
Due to the convention that $\mathcal{C}^i_{\alpha_i \beta_k}=0$
 when the front lines of $\alpha_i$ and $\beta_k$ do not cross, we have
\begin{equation} \label{Eq:EcritureGlobaleM}
\mathcal{M}_{\beta_k}(t) 
= \prod_{\alpha \in \mathcal{A}^2(t)} (1 + \mathcal{C}^1_{\beta_k\alpha} (\sigma_\alpha(t))_*^3)
= \prod_{i=1}^N \left(1 + \mathcal{C}^1_{\beta_k\alpha_i} (\sigma_{\alpha_i}(t))_*^3\right).
\end{equation}
Now we use Proposition~\ref{Pro:ElemePropsVp}--\ref{Item:VpProd} to estimate
 $v_p( \mathcal{M}_{\beta_1} ,\ldots, \mathcal{M}_{\beta_K} )$:
\begin{equation*}
v_p( \mathcal{M}_{\beta_1} ,\ldots, \mathcal{M}_{\beta_K} ) 
\leq \sum_{i=1}^N v_p \bigg( 1+ \mathcal{C}^1_{\beta_k\alpha_i} (\sigma_{\alpha_i}(t))_*^3 ; \ k =1\ldots K\bigg)
\prod_{j \neq i} \max_k \Big| 1+ \mathcal{C}^1_{\beta_k\alpha_j} (\sigma_{\alpha_j}(t))_*^3 \Big|.
\end{equation*}
Since the coefficients $\mathcal{C}^1_{\beta_k\alpha_j}$ are uniformly bounded, using \eqref{Eq:ProdExp} and $H(\tau)$,
 we easily infer that
\begin{equation} \label{Eq:MLinfty}
\prod_{j \neq i} \max_k \Big| 1+ \mathcal{C}^1_{\beta_k\alpha_j} (\sigma_{\alpha_j}(t))_*^3 \Big|
\lesssim \exp \left( \widetilde{V}_p(u^\nu(t))^3 \right)
\lesssim 1.
\end{equation}
Hence we find
\begin{eqnarray*}
v_p( \mathcal{M}_{\beta_1} ,\ldots, \mathcal{M}_{\beta_K} ) 
&\lesssim& \sum_{i=1}^N v_p \Big( 1+ \mathcal{C}^1_{\beta_k\alpha_i} (\sigma_{\alpha_i}(t))_*^3 ; \ k =1\ldots K\Big) \\
&=& \sum_{i=1}^N v_p \Big( \mathcal{C}^1_{\beta_k\alpha_i} (\sigma_{\alpha_i}(t))_*^3 ; \ k =1\ldots K\Big) \\
&= & \sum_{i=1}^N |(\sigma_{\alpha_i}(t))_*|^3 \, v_p \Big( \mathcal{C}^1_{\beta_k\alpha_i} ; \ k =1\ldots K\Big).
\end{eqnarray*}
Due to Lemma~\ref{Lem:VpCoefsInteraction} and $H(\tau)$, we find
\begin{equation} \label{Eq:EstVpM}
v_p( \mathcal{M}_{\beta_1} ,\ldots, \mathcal{M}_{\beta_K} ) 
\lesssim \sum_{i=1}^N |(\sigma_{\alpha_i}(t))_*|^3 
\lesssim \widetilde{V}_p(u^\nu(t))^3 \lesssim V_p(u_0)^3.
\end{equation}
Now as for \eqref{Eq:MLinfty} the boundedness of the future interaction coefficients easily yields
\begin{equation} \label{Eq:EstLinftyM}
\left\| (\mathcal{M}_{\beta_1}-1,\ldots,\mathcal{M}_{\beta_N}-1) \right\|_\infty
\lesssim s_3(\sigma_{\alpha_1},\ldots,\sigma_{\alpha_N})^3
\lesssim V_p(u_0)^{3}.
\end{equation}
Using Corollary~\ref{Cor:MultiplSp}, we finally find that
\begin{eqnarray*}
\mathcal{V}^1(t)^{\frac{1}{p}} 
&\leq& s_p(\sigma_{\beta_1},\ldots,\sigma_{\beta_N}) 
    + s_p\left( (\mathcal{M}_{\beta_1}-1)\sigma_{\beta_1},\ldots,(\mathcal{M}_{\beta_N}-1)\sigma_{\beta_N} \right) \\
&\leq& s_p(\sigma_{\beta_1},\ldots,\sigma_{\beta_N}) \left(1 + \mathcal{O}(1) V_p(u_0)^{3}\right).
\end{eqnarray*}
In the other direction, we use the fact that all preamplification factors satisfy \eqref{Eq:Mentre2et12}.
Using Proposition~\ref{Pro:ElemePropsVp}--\ref{Item:LipVp} with $f(x) = 1/x$, we find that 
\begin{equation} \label{Eq:Mmoins1}
v_p (\mathcal{M}^{-1}_{\beta_1},\ldots,\mathcal{M}^{-1}_{\beta_N}) \lesssim V_p(u_0)^{3}.
\end{equation}
In the same way the factors $\mathcal{M}^{-1}_{\beta_k}$ satisfy \eqref{Eq:EstLinftyM}.
Consequently we also obtain
\begin{equation} \label{Eq:spVsp1}
 s_p(\sigma_{\beta_1},\ldots,\sigma_{\beta_N}) \leq \mathcal{V}^1(t)^{\frac{1}{p}} \left(1 + \mathcal{O}(1) V_p(u_0)^{3}\right).
\end{equation}
The same can be done for family $2$, so
\begin{equation*}
\mathcal{V}^2(t)^{\frac{1}{p}} \left(1 - \mathcal{O}(1) V_p(u_0)^{3}\right)
\leq s_p(\sigma_{\alpha_1},\ldots,\sigma_{\alpha_K})
\leq \mathcal{V}^2(t)^{\frac{1}{p}} \left(1 + \mathcal{O}(1) V_p(u_0)^{3}\right) .
\end{equation*}
Raising the above inequalities to the power $p$ and summing them, we reach \eqref{Eq:EquivVVp}. \par
\ \par
\noindent{\bf Proof of \eqref{Eq:QQuadratique}.} 
The estimate $\mathcal{Q}_1(t) + \mathcal{Q}_2(t) \leq C \, [\mathcal{V}(t)]^2$ is immediate
 noting that $\sum_{\alpha \in \mathcal{A}^k(t)} \mathfrak{s}_l^k(\alpha) = \mathcal{V}^k(t)$
 and  $\sum_{\alpha \in \mathcal{A}^k(t)} \mathfrak{s}_r^k(\alpha) = \mathcal{V}^k(t)$.
For $\mathcal{Q}_3$, we have 
\begin{equation*}
\mathcal{Q}_3(t) 
\leq s_3^3(\sigma_\alpha; \alpha \in \mathcal{A}(t) ) \\
\leq s_p^3(\sigma_\alpha; \alpha \in \mathcal{A}(t) ) \\
\lesssim \mathcal{V}(t)^{3/p} \lesssim \mathcal{V}(t)^{2},
\end{equation*}
where we used Proposition~\ref{Pro:ElemePropsVp}--\ref{Item:SpSpPrime}, \eqref{Eq:EquivVVp}, $H(\tau)$ and $p \leq \frac{3}{2}$. \par
\ \par
\noindent{\bf Proof of \eqref{Eq:MGamma0}.} 
The proof of \eqref{Eq:MGamma0} is similar to the one of \eqref{Eq:EstVpM}. 
One replaces $\mathcal{M}_{\beta}$ with $\widetilde{\mathcal{M}}^{\Gamma}_{\beta}$
 and $\mathcal{C}^1_{\beta\alpha}$ with $\tilde{\mathcal{C}}^{1,\Gamma}_{\beta\alpha}$
 in \eqref{Eq:EcritureGlobaleM}. 
The main point is that we have for coefficients $\tilde{\mathcal{C}}^{1,\Gamma}_{\beta\alpha}$
 the same uniform bound as for coefficients ${\mathcal{C}}^{1}_{\beta\alpha}$ in Lemma~\ref{Lem:VpCoefsInteraction},
 namely that $v_p \Big( \tilde{\mathcal{C}}^{1,\Gamma}_{\beta_k\alpha_i} ; \ k =1\ldots K\Big)$ is bounded.
This is due to the fact that if $\beta_k$ and $\beta_{k'}$ both meet $\alpha_i$ in the unlimited component
 of $(\R_+ \times \R) \setminus \Gamma$ (for $k<k'$), then the same is valid for $\beta_j$, $k<j<k'$, since
 $\alpha_i$ can cross $\Gamma$ at most once between $t_{\min}$ and $t_{\max}$, because $\Gamma$ coincides
 with a front line for such times. 
Notice also that the coefficients $\mathcal{C}^1_{\beta\alpha}$ do not evolve during $[0,T]$ and are consequently 
not affected by the freezing procedure on preamplification factors $\widetilde{\mathcal{M}}^{\Gamma}_{\beta}$. \par
It follows that the sequence $( \tilde{\mathcal{C}}^{1,\Gamma}_{\beta_k\alpha_i} ; \ k =1\ldots K)$
 can be written as the multiplication of the sequence
 $( {\mathcal{C}}^{1}_{\beta_k\alpha_i} ; \ k =1\ldots K)$ 
 with a sequence of the form $(0,\ldots,0,1,\ldots,1,0,\ldots,0)$,
and the conclusion follows with Proposition~\ref{Pro:ElemePropsVp}--\ref{Item:Ajoutde0}. \par
\ \par
This ends the proof of Proposition~\ref{Pro:EstimeesVetQ}.
\hfill \qedsymbol \par
%
%
%
%
%
%
%
%
%
%
\subsection{Local estimates: decay of the oracle Glimm-type functional}
\label{Subsec:Decay}
In this section, we prove Proposition~\ref{Pro:DecayGlimm}. \par
\ \par
We introduce a measure curve $\Gamma$ in the strip $[0,\tau] \times \R$,
 and we associate the various functionals (in particular $\Upsilon$) to the time-horizon $\tau$ and 
 to the curve $\Gamma$. \par
Now the structure of the proof is as follows. 
We first consider the evolution of the functional $\Upsilon$ and prove \eqref{Eq:DecayGlimm}.
Since $\Upsilon(t)$ is clearly constant between interaction points, 
one only has to check that for any interaction time $t$, one has
\begin{equation} \label{Eq:DecayUpsilon}
\Upsilon(t^+) \leq \Upsilon(t^-) .
\end{equation}
We will do so by considering the three possible types of interactions at time $t$. \par
Then in a second time, we prove \eqref{Eq:DecayGlimm2}, by showing that the decay of $\mathcal{Q}$
 is sufficient to control the possible growths of $\mathcal{V}_\Gamma$ and $\mathfrak{M}_\Gamma$. \par
\ \par
\noindent
We recall that at an interaction time, only two fronts are involved. 
Hence we consider an interaction time $t$ and two fronts $\alpha$ (on the left)
 and $\beta$ (on the right) that interact at time $t^-$. 
We call $u_l$ (respectively $u_r$) the leftmost (resp. rightmost) state at the interaction,
 and $u_m$ the state between the two fronts $\alpha$ and $\beta$, before they interact.
We discuss \eqref{Eq:DecayUpsilon} according to the nature of the waves 
 carried by $\alpha$ and $\beta$.
%
%
%
%
%
%
%
%
\subsubsection{Interactions of opposite families}
\label{Subsec:IOF}
Here we suppose that $\alpha$ and $\beta$ correspond to waves of different families. 
Since $\alpha$ and $\beta$ interact, this means that $\alpha$ is of family $2$ 
 and $\beta$ of family $1$. 
We call $\beta'$ the outgoing front of family $1$ and $\alpha'$ the outgoing front of family $2$
 (we recall that at interaction points of opposite families, only single fronts emerge for each family). 
\par
According to Proposition~\ref{Eq:InteractionDF}, we have
\begin{equation} \label{Eq:EvolForcesDeBase}
\sigma_{\beta'} = \sigma_{\beta} 
                \big( 1 + \mathcal{C}^1(u_m;\sigma_\beta,\sigma_\alpha) (\sigma_{\alpha})_*^3 + \mathfrak{e}_{\beta} \big)
 \ \text{ and } \ 
\sigma_{\alpha'} = \sigma_{\alpha} 
                \big( 1 + \mathcal{C}^2(u_m;\sigma_\beta,\sigma_\alpha) (\sigma_{\beta})_*^3 + \mathfrak{e}_{\alpha} \big),
\end{equation}
where $\mathfrak{e}_{\beta} = \mathcal{O}(1) (\sigma_{\alpha})_*^3 |\sigma_{\beta}|^2$ and 
$\mathfrak{e}_{\alpha} = \mathcal{O}(1) (\sigma_{\beta})_*^3 |\sigma_{\alpha}|^2$ are the possible
additional error terms from \eqref{Eq:InteractionDF-CW} when $\beta$ or $\alpha$ is a compression wave.
On the other side, due to \eqref{Eq:AmplificationFactor}, we have
\begin{equation*}
\mathcal{M}_{\beta'} = \mathcal{M}_{\beta} 
                       \big( 1 + \mathcal{C}^1(u_m;0,\sigma_\alpha) (\sigma_{\alpha})_*^3\big)^{-1}
 \ \text{ and } \ 
\mathcal{M}_{\alpha'} = \mathcal{M}_{\alpha} 
                        \big( 1 + \mathcal{C}^2(u_m;\sigma_\beta,0) (\sigma_{\beta})_*^3\big)^{-1}.
\end{equation*}
Using that $\mathcal{C}^1$ and $\mathcal{C}^2$ are Lipschitz, we deduce that
\begin{equation} \label{Eq:EvolForcesAmplifiees}
\widehat{\sigma}_{\beta'} = \mathcal{M}_{\beta'} \sigma_{\beta'} 
    = \widehat{\sigma}_{\beta} + \mathcal{O}(1) |(\sigma_\alpha)_*|^3 |\sigma_\beta|^2
 \ \text{ and } \ 
\widehat{\sigma}_{\alpha'} = \mathcal{M}_{\alpha'} \sigma_{\alpha'}
    = \widehat{\sigma}_{\alpha} + \mathcal{O}(1) |\sigma_\alpha|^2 |(\sigma_\beta)_*|^3.
\end{equation}
Now we measure the error induced by the transformation of $\alpha$ and $\beta$ 
on $\mathcal{V}$, $\mathcal{Q}_1$, $\mathcal{Q}_2$ and $\mathcal{Q}_3$. \par
\ \par \noindent
{\bf Evolution of $\mathcal{V}$.}
The functional $\mathcal{V}$ is affected in two ways by the interaction of $\alpha$ and $\beta$:
\begin{itemize}
\item the preamplified forces of $\alpha$ and $\beta$ are modified according to \eqref{Eq:EvolForcesAmplifiees};
\item the preamplified strengths of other fronts change as well,
 due to the modification of $\sigma_\alpha$ and $\sigma_\beta$ that affects
 the preamplification factors $\mathcal{M}_\gamma$ for fronts $\gamma$ approaching $\alpha$ and $\beta$.
\end{itemize}
\noindent
{\it Update of preamplification factors of other fronts.}
Call $\widetilde{\mathcal{V}}^1(t)$  (respectively $\widetilde{\mathcal{V}}^2(t)$)
the value obtained in \eqref{Eq:TotalStrengthes} by updating the preamplification factors
of $1$- (resp. $2$-) fronts $\gamma$ at time $t$ to take the modification of $\sigma_\beta$ 
 (resp. $\sigma_\alpha$) into account, but not yet the modification of the preamplified strength
 of the involved front $\widehat{\sigma}_\beta$ (resp. $\widehat{\sigma}_\alpha$) itself. \par
Let us start with the estimate of $\widetilde{\mathcal{V}}^1(t)$.
Consider a $1$-front $\gamma$ approaching $\alpha$.
The preamplified strength of $\gamma$ evolves according to
\begin{equation} \label{Eq:ChgtPreampFac1}
\widehat{\sigma}_{\gamma}(t^+) 
= \widehat{\sigma}_{\gamma}(t^-) 
\frac{ (1 + \mathcal{C}^1_{\alpha \gamma} (\sigma_{\alpha'})_*^3) }
    { (1 + \mathcal{C}^1_{\alpha \gamma} (\sigma_\alpha)_*^3)  }.
\end{equation}
Notice that this is valid included in the case where $\mathcal{C}^1_{\alpha \gamma}=0$. \par
We apply Proposition~\ref{Pro:ElemePropsVp}-\ref{Item:LipVp} with 
\begin{equation*}
f(x) = \frac{ (1 + x (\sigma_{\alpha'})_*^3) }{ (1 + x (\sigma_\alpha)_*^3)  } -1
\ \text{ for } x \in [-C_*,C_*],
\end{equation*}
where $C_*$ was introduced at the end of Paragraph~\ref{sss:SIE}.
We note that since the nature of fronts do not change across interactions of opposite families and due to \eqref{Eq:EvolForcesDeBase},
 we have
\begin{equation*}
\|f\|_{W^{1,\infty}} 
= \mathcal{O}(1) \big( (\sigma_{\alpha'})_*^3 - (\sigma_{\alpha})_*^3 \big)
= \mathcal{O}(1) |\sigma_\alpha|^3|\sigma_\beta|^3.
\end{equation*}
Consequently,
\begin{equation} \label{Eq:EstOscilCoefs}
v_p(f(\mathcal{C}^1_{\alpha \gamma})_{\gamma \in \mathcal{A}^1(t)})
\lesssim |\sigma_\alpha|^3|\sigma_\beta|^3 
            v_p((\mathcal{C}^1_{\alpha \gamma})_{\gamma \in \mathcal{A}^1(t)}).
\end{equation}
Now we use Corollary~\ref{Cor:MultiplSp} to deduce that
\begin{equation*}
s_p(( \widehat{\sigma}_{\gamma}(t^+) - \widehat{\sigma}_{\gamma}(t^-) )_{\gamma \in \mathcal{A}^1(t)}) 
\lesssim |\sigma_\alpha|^3|\sigma_\beta|^3 
    v_p((\mathcal{C}^1_{\alpha \gamma})_{\gamma \in \mathcal{A}^1(t)})
    s_p(( \widehat{\sigma}_{\gamma}(t^-) )_{\gamma \in \mathcal{A}^1(t)}) .
\end{equation*}
The preamplification factors of $1$-fronts to the left of $\alpha$ at time $t$
 do not change across this interaction.
Consequently, using Proposition~\ref{Pro:ElemePropsVp}-\ref{Item:DLVp} (and noticing that
 the above term is $\lesssim \mathcal{V}^1(t^-)$), we have 
\begin{eqnarray}
\nonumber 
\widetilde{\mathcal{V}}^1(t) 
&\leq& \mathcal{V}^1(t^-) + C \mathcal{V}^1(t^-)^{\frac{p-1}{p}} 
\Big[s_p(( \widehat{\sigma}_{\gamma}(t^+) - \widehat{\sigma}_{\gamma}(t^-) )_{\gamma \in \mathcal{A}^1(t)}) 
\Big] \\
\label{Eq:V1updatePreamp}
&\leq& \mathcal{V}^1(t^-) + C \mathcal{V}^1(t^-)^{\frac{p-1}{p}} 
\Big[|\sigma_\alpha|^3|\sigma_\beta|^3 
    v_p((\mathcal{C}^1_{\alpha \gamma})_{\gamma \in \mathcal{A}^1(t)})
    s_p(( \widehat{\sigma}_{\gamma}(t^-) )_{\gamma \in \mathcal{A}^1(t)})
\Big].
\end{eqnarray}
Reasoning in the same way for family $2$, we have
\begin{equation}
\label{Eq:V2updatePreamp}
\widetilde{\mathcal{V}}^2(t) 
\leq \mathcal{V}^2(t^-) + C \mathcal{V}^2(t^-)^{\frac{p-1}{p}} 
\Big[|\sigma_\alpha|^3|\sigma_\beta|^3 
    v_p((\mathcal{C}^2_{\beta \gamma})_{\gamma \in \mathcal{A}^2(t)})
    s_p(( \widehat{\sigma}_{\gamma}(t^-) )_{\gamma \in \mathcal{A}^2(t)})
\Big].
\end{equation}
Using Lemma~\ref{Lem:VpCoefsInteraction} (since $H(\tau)$ is valid), we have
 $v_p((\mathcal{C}^2_{\beta \gamma})_{\gamma \in \mathcal{A}^2(t)}) =\mathcal{O}(1)$
and
 $v_p((\mathcal{C}^1_{\alpha \gamma})_{\gamma \in \mathcal{A}^1(t)}) =\mathcal{O}(1)$.
Moreover we have
 $s_p(( \widehat{\sigma}_{\gamma}(t^-) )_{\gamma \in \mathcal{A}^2(t)}) \leq \mathcal{V}(t^-)^{1/p}$
and
 $s_p(( \widehat{\sigma}_{\gamma}(t^-) )_{\gamma \in \mathcal{A}^1(t)}) \leq \mathcal{V}(t^-)^{1/p}$.
Hence we find
\begin{equation} \nonumber 
\widetilde{\mathcal{V}}^1(t)  + \widetilde{\mathcal{V}}^2(t) 
 = \mathcal{V}(t^-) 
+ \mathcal{O}(1) \mathcal{V}(t^-) |\sigma_\alpha|^3 |\sigma_\beta|^3 .
\end{equation}
We note in passing that $\widetilde{\mathcal{V}}^1(t) \lesssim \mathcal{V}^1(t^-)$ and 
$\widetilde{\mathcal{V}}^2(t) \lesssim \mathcal{V}^2(t^-)$. \par
\ \par
\noindent
{\it Update of $\widehat{\sigma}_\alpha$ and $\widehat{\sigma}_\beta$.} 
The difference $\mathcal{V}^1(t^+) -\widetilde{\mathcal{V}}^1(t)$ 
 (respectively $\mathcal{V}^2(t^+) -\widetilde{\mathcal{V}}^2(t)$)
 is due to the change $\widehat{\sigma}_\beta(t^-)$ in $\widehat{\sigma}_\beta(t^+)$ 
 (resp. $\widehat{\sigma}_\alpha(t^-)$ in $\widehat{\sigma}_\alpha(t^+)$),
 for which we have \eqref{Eq:EvolForcesAmplifiees}.
Using Proposition~\ref{Pro:ElemePropsVp}--\ref{Item:DLVp}, we find that
\begin{equation*}
\mathcal{V}^1(t^+) -\widetilde{\mathcal{V}}^1(t)
    = \mathcal{O}(1) \widetilde{\mathcal{V}}^1(t)^{\frac{p-1}{p}} 
        |\widehat{\sigma}_\beta(t^-) - \widehat{\sigma}_\beta(t^+)|
    = \mathcal{O}(1) \mathcal{V}^1(t^-)^{\frac{p-1}{p}} 
          |\sigma_\beta|^2 |(\sigma_\alpha)_*|^3,
\end{equation*}
and in the same way
\begin{equation*}
\mathcal{V}^2(t^+) -\widetilde{\mathcal{V}}^2(t)
    = \mathcal{O}(1) \widetilde{\mathcal{V}}^2(t)^{\frac{p-1}{p}} 
          |\sigma_\alpha|^2 |(\sigma_\beta)_*|^3.
\end{equation*} \par
\ \par
\noindent
{\it Final estimate of $\mathcal{V}(t^+) - \mathcal{V}(t^-)$.} 
Putting these estimate together we obtain
\begin{equation} \label{Eq:EstVInterDiff}
\mathcal{V}(t^+) = \mathcal{V}(t^-) + \mathcal{O}(1) \mathcal{V}(t^-)^{\frac{p-1}{p}} 
        |\sigma_\alpha|^2 |\sigma_\beta|^2 (|\sigma_\alpha| + |\sigma_\beta|).
\end{equation} \par
\ \par
%
\noindent
{\bf Evolution of $\mathcal{Q}_1$.}
As for $\mathcal{V}$, the quadratic term $\mathcal{Q}_1$ is modified through the change in $\widehat{\sigma}_{\alpha}$
and $\widehat{\sigma}_{\beta}$, and because the preamplified factors of other fronts change as well. 
We write
\begin{equation*}
\mathcal{Q}_1^k(t^+) - \mathcal{Q}_1^k(t^-)
= \sum_{\substack{ {\mu,\nu \in \mathcal{A}^k} \\ {\nu > \mu} }}
   \mathfrak{s}^k_r(\mu)(t^+) \, \Big[ \mathfrak{s}^k_l(\nu)(t^+) - \mathfrak{s}^k_l(\nu)(t^-) \Big]
+ \sum_{\substack{ {\mu,\nu \in \mathcal{A}^k} \\ {\nu > \mu} }}
   \Big[ \mathfrak{s}^k_r(\mu)(t^+) - \mathfrak{s}^k_r(\mu)(t^-) \Big] \, \mathfrak{s}^k_l(\nu)(t^-),
\end{equation*}
Recalling $\mathfrak{s}^k_l(\gamma) := \mathcal{V}^k( \leq \gamma) - \mathcal{V}^k( < \gamma)$,
 we find for the first sum that
\begin{equation*}
\sum_{\substack{ {\mu,\nu \in \mathcal{A}^k} \\ {\nu > \mu} }} \mathfrak{s}^k_r(\mu)(t^+)
    \, \Big[ \mathfrak{s}^k_l(\nu)(t^+) - \mathfrak{s}^k_l(\nu)(t^-) \Big] 
= \sum_{\mu \in \mathcal{A}^k} \mathfrak{s}^k_r(\mu)(t^+)
    \,  \Big[ \mathcal{V}^ k(t^+)-\mathcal{V}^k(\leq \mu)(t^+) - \mathcal{V}^ k(t^-) + \mathcal{V}^k(\leq \mu)(t^-) \Big] .
\end{equation*}
Now following the same lines as for \eqref{Eq:EstVInterDiff}, we see that we can estimate 
$\mathcal{V}^k(\leq \mu; t^+) -\mathcal{V}^k(\leq \mu; t^-)$ in the same way as 
$\mathcal{V}^k(t^+) - \mathcal{V}^k(t^-)$. 
Namely, the only differences are as follows. The update of preamplification factors only concern fronts 
 to the left of $\mu$. 
 Hence, the errors in \eqref{Eq:V1updatePreamp} and \eqref{Eq:V2updatePreamp} are replaced here with
\begin{multline*}
\mathcal{O}(1) \, \mathcal{V}_\mu^1(t^-)^{\frac{p-1}{p}} 
\Big[|\sigma_\alpha|^3|\sigma_\beta|^3 
    v_p((\mathcal{C}^2_{\alpha \gamma})_{\gamma \in  \mathcal{A}^1_{\leq \mu}})
    s_p(( \widehat{\sigma}_{\gamma}(t^-) )_{\gamma \in  \mathcal{A}^1_{\leq \mu}})
\Big]
\\ \text{ and } \ 
\mathcal{O}(1) \, \mathcal{V}_\mu^2(t^-)^{\frac{p-1}{p}} 
\Big[|\sigma_\alpha|^3|\sigma_\beta|^3 
    v_p((\mathcal{C}^1_{\beta \gamma})_{\gamma \in \mathcal{A}^2_{\leq \mu}})
    s_p(( \widehat{\sigma}_{\gamma}(t^-) )_{\gamma \in \mathcal{A}^2_{\leq \mu}})
\Big]. 
\end{multline*}
The update in the preamplified forces $\widehat{\sigma}_\alpha$ and $\widehat{\sigma}_\beta$
 is the same, but only apply when $\alpha$ and $\beta$ are to the left of $\mu$.
Hence we get in the same way that 
\begin{equation*} 
\mathcal{V}^k(\leq \mu; t^+) = \mathcal{V}^k(\leq \mu; t^-) + \delta \mathcal{V}(t),
\end{equation*}
with
\begin{equation*}
\delta \mathcal{V}(t)= 
\mathcal{O}(1) \mathcal{V}(t^-)^{\frac{p-1}{p}} 
            |\sigma_\alpha|^2 |\sigma_\beta|^2 (|\sigma_\alpha| + |\sigma_\beta|).
\end{equation*}
Hence we find 
\begin{equation*}
\mathcal{Q}_1^k(t^+) - \mathcal{Q}_1^k(t^-)
\leq \delta \mathcal{V}(t)  \sum_{\mu \in \mathcal{A}^k} \mathfrak{s}^k_r(\mu)(t^+) \\
=  \delta \mathcal{V}(t) \, \mathcal{V}^ k(t^+)
\leq 2 \delta \mathcal{V}(t) \, \mathcal{V}^ k(t^-).
\end{equation*}
The other sum can be treated likewise, so we finally find that
\begin{equation} \label{Eq:EstQ1-1}
\mathcal{Q}_1^k(t^+) - \mathcal{Q}_1^k(t^-)
= \mathcal{O}(1) \mathcal{V}(t^-)^{\frac{2p-1}{p}} |\sigma_\alpha|^2 |\sigma_\beta|^2 
                                        (|\sigma_\alpha| + |\sigma_\beta|) .
\end{equation} \par
\ \par \noindent
{\bf Evolution of $\mathcal{Q}_2$.}
The functional $\mathcal{Q}_2$ is treated exactly in the same way as $\mathcal{Q}_1$, 
 with the only difference that a term is removed from $\mathcal{Q}_2$ in the case of
 an interaction of fronts of opposite families.
Hence following the same lines as for $\mathcal{Q}_1$, we find
\begin{equation*}
\mathcal{Q}_2(t^+) - \mathcal{Q}_2(t^-) = - \mathfrak{s}^k_r(\alpha) \, \mathfrak{s}^k_l(\beta)
+ \mathcal{O}(1) \mathcal{V}(t^-)^{\frac{2p-1}{p}} |\sigma_\alpha|^2 |\sigma_\beta|^2 
                                        (|\sigma_\alpha| + |\sigma_\beta|).
\end{equation*}
Now using Proposition~\ref{Pro:PropsVpEndpoints}--\ref{Item:EndPoint2} and \eqref{Eq:Mentre2et12}, we have that
$$
\mathfrak{s}^1_r(\alpha) = \mathcal{V}^1(\geq \alpha) - \mathcal{V}^1(> \alpha) 
        \geq |\widehat{\sigma}_\alpha|^p \geq 2^{-p}|\sigma_\alpha|^p
\ \text{ and } \ 
\mathfrak{s}^1_l(\beta) 
        = \mathcal{V}^1(\leq \beta) - \mathcal{V}^1(< \beta) \geq |\widehat{\sigma}_\beta|^{p} \geq 2^{-p}|\sigma_\beta|^{p}.
$$
We finally get
\begin{equation*}
\mathcal{Q}_2(t^+) - \mathcal{Q}_2(t^-) \leq - 2^{-p} |\sigma_\alpha|^p \, |\sigma_\beta|^{p}
+ \mathcal{O}(1) \mathcal{V}(t^-)^{\frac{2p-1}{p}} |\sigma_\alpha|^2 |\sigma_\beta|^2 
                                        (|\sigma_\alpha| + |\sigma_\beta|) .
\end{equation*}
\ \par \noindent
{\bf Evolution of $\mathcal{Q}_3$.}
The case of $\mathcal{Q}_3$ is simpler, because $\mathcal{Q}_3$ involves only standard
 (local, unampflified) forces. 
Using Proposition~\ref{Prop:InteractionDF} and the fact that at an interaction point for different families, 
 the outgoing fronts have the same nature as the incoming ones, one finds,
 using simply \eqref{Eq:SommePuissanceP} for $p=3$,
\begin{equation*}
\mathcal{Q}_3(t^+) - \mathcal{Q}_3(t^-) =  \mathcal{O}(1) \,|\sigma_\alpha(t^-)|^3 |\sigma_\beta(t^-)|^3 .
\end{equation*}
\ \par \noindent
{\bf Conclusion.}
By $H(\tau)$ and Proposition~\ref{Pro:EstimeesVetQ} we have $\mathcal{V}(t) \lesssim V_p(u_0)^p$.
Hence using $p<2$, and relying on the $- 2^{-p} |\sigma_\alpha(t^-)|^p |\sigma_\beta(t^-)|^p$ term in the decay
 of $\mathcal{Q}_2$, we easily find that $\Upsilon(t^+) - \Upsilon(t^-) \leq 0$ in that case, 
 provided that $V_p(u_0)$ is small enough. This does not require $C_1$ or $C_2$ to be large. \par
\smallskip
%
%
%
%
%
%
\subsubsection{Interactions within a family: monotone case}
\label{sss:monotone}
 
We consider two fronts $\alpha$ and $\beta$ of the same family (say, of family $1$),
 with $\sigma_\alpha \sigma_\beta >0$, interacting at time~$t$.
We rely on Proposition~\ref{Prop:InteractionSF} to measure the strengths of new waves. 
%
%
We denote $\mathcal{E}_{\alpha\beta}$ a term allowing to measure the error in the interaction estimate, namely
\begin{equation*}
\mathcal{E}_{\alpha \beta} =
\left\{ \begin{array}{l}
 |\sigma_\alpha \sigma_\beta^2| + |\sigma_\alpha^2 \sigma_\beta| 
        \text{ if } \alpha \text{ and } \beta \text{ are both classical}, \medskip \\
 |\sigma_\alpha \sigma_\beta^2| + |\sigma_\alpha^2 \sigma_\beta| + |\sigma_\beta|^3 
        \text{ if } \alpha \text{ is a shock and } \beta \text{ is a CW}, \medskip \\
 |\sigma_\alpha \sigma_\beta^2| + |\sigma_\alpha^2 \sigma_\beta| + |\sigma_\alpha|^3 
        \text{ if } \alpha \text{ is a CW and } \beta \text{ is a shock}, \medskip \\
 |\sigma_\alpha \sigma_\beta^2| + |\sigma_\alpha^2 \sigma_\beta| + |\sigma_\alpha|^3 + |\sigma_\beta|^3 
        \text{ if } \alpha \text{ and } \beta 
        \text{ are both CW with } \sigma_\alpha + \sigma_\beta \leq - 2 {\nu}, \medskip \\
 0 \text{ if } \alpha \text{ and } \beta 
        \text{ are both CW with } \sigma_\alpha + \sigma_\beta > - 2{\nu}.
\end{array} \right.
\end{equation*}
Call $\alpha'$ the outgoing front of family $1$; 
 there can be several outgoing fronts in family $2$, which we call $\mu_1,\ldots,\mu_k$ (all of the same sign).
Proposition~\ref{Prop:InteractionSF} establishes that
\begin{equation} \label{Eq:EvolForcesMonot}
\sigma_{\alpha'} = \sigma_\alpha + \sigma_\beta + \mathcal{O}(1) \mathcal{E}_{\alpha\beta}^3
\ \text{ and } \ 
\sigma_{\mu_1} + \ldots + \sigma_{\mu_k} =\mathcal{O}(1) \mathcal{E}_{\alpha\beta}.
\end{equation}
The preamplification factors of $\alpha$ and $\beta$ do not change across the interaction (and coincide), 
 while the preamplification factor for the outgoing $\mu_1,\ldots,\mu_k$ are bounded, so that
\begin{equation} \label{Eq:SortieMC}
\widehat{\sigma}_{\alpha'} = \widehat{\sigma}_\alpha + \widehat{\sigma}_\beta + \mathcal{O}(1) \mathcal{E}_{\alpha\beta}^3
\ \text{ and } \ 
\widehat{\sigma}_{\mu_1} + \ldots + \widehat{\sigma}_{\mu_k} =\mathcal{O}(1) \mathcal{E}_{\alpha\beta}.
\end{equation}
\ \par
\noindent
{\bf Evolution of $\mathcal{V}$.}
As before, $\mathcal{V}$ is modified due to \eqref{Eq:SortieMC}, and also because of the change
 in the preamplification factors of fronts other than $\alpha$ and $\beta$. \par
\ \par
\noindent
{\it Evolution of $\mathcal{V}^1$.} 
Due to the newly created $2$-wave at the interaction point,
the $1$-fronts $\gamma$ to the right of $\beta$ get an additional factor
 (only non-trivial when $k=1$ and $\mu_1$ is a shock) given by 
 $(1+\mathcal{C}^1_{\mu_1 \gamma} (\sigma_{\mu_1})_*^3)$. 
Denoting $\{ \gamma^1_1, \ldots, \gamma^1_M \}$ the $1$-fronts at time $t^-$ as in \eqref{Eq:DefASets}, with $\beta=\gamma_k$,
 this corresponds to multiplying the sequence
$(\widehat{\sigma}_{\gamma^1_1}, \ldots, \widehat{\sigma}_{\gamma^1_k}, 
            \widehat{\sigma}_{\gamma^1_{k+1}} \ldots, \widehat{\sigma}_{\gamma^1_M})$
by the sequence
$(1,\ldots,1, 1 + \mathcal{C}^1_{\mu_1 \gamma^1_{k+1}} (\sigma_{\mu_1})_* ^3, 
                \ldots, 1 + \mathcal{C}^1_{\mu_1 \gamma^1_{M}} (\sigma_{\mu_1})_*^3 )$,
that is to say, to add to it
$(0,\ldots,0, \mathcal{C}^1_{\mu_1 \gamma^1_{k+1}} (\sigma_{\mu_1})_*^3 \, \sigma_{\gamma^1_{\mu_{k+1}}} , 
               \ldots, \mathcal{C}^1_{\mu_1 \gamma^1_{M}} (\sigma_{\mu_1})_*^3 \, \sigma_{\gamma^1_{\mu_{K}}} )$.
Using Corollary~\ref{Cor:MultiplSp} we estimate the $p$-sum of this term by
\begin{equation*}
(\sigma_{\mu_1})_*^3 \, v_p(0, \mathcal{C}^1_{\mu_1 \gamma^1_{k+1}}, \ldots,
            \mathcal{C}^1_{\mu_1 \gamma^1_{M}}) \, 
            s_p(\widehat{\sigma}_{\gamma^1_{\mu_{k+1}}},\ldots,\widehat{\sigma}_{\gamma^1_{\mu_{K}}})
\lesssim \mathcal{E}_{\alpha\beta}^3 \, \mathcal{V}^1(t^-)^{\frac{1}{p}} ,
\end{equation*}
where we used Lemma~\ref{Lem:VpCoefsInteraction} (since $H(\tau)$ is valid) 
 and Proposition~\ref{Pro:ElemePropsVp}--\ref{Item:Ajoutde0}. \par
Taking additionally into account the error due to $\widehat{\sigma}_\alpha,\widehat{\sigma}_\beta$ being transformed 
 into $\widehat{\sigma}_{\alpha'}$, using Proposition~\ref{Pro:ElemePropsVp}--\ref{Item:DLVp}\&\ref{Item:Fusion}, 
 we find that
\begin{equation*}
\mathcal{V}^1(t^+) \leq \mathcal{V}^1(t^-) 
+ \mathcal{O}(1) \mathcal{V}^1(t^-)^{\frac{p-1}{p}} \mathcal{E}_{\alpha\beta}.
\end{equation*}
\ \par
\noindent
{\it Evolution of $\mathcal{V}^2$.} 
We now consider the change in the preamplification factors for $2$-fronts. 
For $\gamma$ of family $2$ such that $\gamma <\alpha$, the factor evolves according to
\begin{equation} \label{Eq:ChgtPreampFac2}
\widehat{\sigma}_{\gamma}(t^+) = \widehat{\sigma}_{\gamma}(t^-) 
\frac{ (1 + \mathcal{C}^2_{\alpha \gamma} (\sigma_{\alpha'})_*^3) }
     { (1 + \mathcal{C}^2_{\alpha \gamma} (\sigma_\alpha)_*^3)  
                (1 + \mathcal{C}^2_{\alpha \gamma} (\sigma_\beta)_*^3)  }.
\end{equation}
The fact that the coefficient $\mathcal{C}^2_{\alpha \gamma}$ is the same for all three expressions 
is due to the fact that all these fronts correspond to the same front-line when interacting with $\gamma$. 
Hence here we set
\begin{equation} \label{Eq:fCoefMC}
f(x) = \frac{ (1 + x (\sigma_{\alpha'})_*^3) }
            { (1 + x (\sigma_\alpha)_*^3) (1 + x (\sigma_\beta)_*^3)  }, \ \ x \in [-C_*,C_*],
\end{equation}
and observe that here we have in all cases
$\|f\|_{W^{1,\infty}} = \mathcal{O}(1) \Big((\sigma_{\alpha'})_*^3 - (\sigma_\alpha)_*^3 - (\sigma_\beta)_*^3 \Big)
= \mathcal{O}(1) \mathcal{E}_{\alpha \beta}$.
Hence we find 
\begin{equation*} 
\widetilde{\mathcal{V}}^2(t) \leq \mathcal{V}^2(t^-) + \mathcal{O}(1) \mathcal{V}^2(t^-)^{\frac{p-1}{p}} \mathcal{E}_{\alpha \beta},
\end{equation*}
where, as in Subsection~\ref{Subsec:IOF}, $\widetilde{\mathcal{V}}^2(t)$ denotes the value obtained after updating
 the preamplification factors in $\mathcal{V}^2(t^-)$. \par
The rest of the change in $\mathcal{V}^2$ is due the new $2$-waves $\mu_1,\ldots,\mu_k$ which have to be included in $\mathcal{V}^2$.
Due to \eqref{Eq:SortieMC}, we have $s_p(\mu_1,\ldots,\mu_k) \lesssim \mathcal{E}_{\alpha\beta}$.
Using Proposition~\ref{Pro:ElemePropsVp}--\ref{Item:DLVp}, we deduce
\begin{equation*} 
\mathcal{V}^2(t^+) - \widetilde{\mathcal{V}}^2(t)  \leq \mathcal{O}(1) 
\Big(\widetilde{\mathcal{V}}^2(t)^{\frac{p-1}{p}} \mathcal{E}_{\alpha \beta} + \mathcal{E}_{\alpha \beta}^p\Big).
\end{equation*}
Finally, using the rough estimate $\mathcal{E}_{\alpha\beta} \lesssim \mathcal{V}_1(t^-)^{1/p}$, 
 we find that
\begin{equation*} 
\mathcal{V}(t^+) \leq \mathcal{V}(t^-) 
+ \mathcal{O}(1) \widetilde{\mathcal{V}}(t^-)^{\frac{p-1}{p}} \mathcal{E}_{\alpha \beta}.
\end{equation*}
\ \par
\ \par
\noindent
{\bf Evolution of $\mathcal{Q}_1$.}
The modification of $\mathcal{Q}_1^1$ has three sources:
\begin{itemize}
\item the various $\mathfrak{s}^1_r(\gamma)$ and $\mathfrak{s}^1_l(\gamma)$ are modified 
      (when $\gamma < \alpha$ and $\gamma > \alpha$, respectively) due to the change 
      in the preamplification factors that have to take into account the appearing of a new wave in family~$2$,
\item the terms $( \mathfrak{s}^1_r(\alpha) + \mathfrak{s}^1_r(\beta) ) \, \mathfrak{s}^1_l(\gamma)$ are
      replaced with $\mathfrak{s}^1_r(\alpha') \, \mathfrak{s}^1_l(\gamma)$ and in the other direction,
      terms $\mathfrak{s}^1_r(\gamma) \, ( \mathfrak{s}^1_l(\alpha) + \mathfrak{s}^1_l(\beta) )$ are
      replaced with $\mathfrak{s}^1_r(\gamma) \mathfrak{s}^1_l(\alpha')$,
\item the term $\mathfrak{s}^1_r(\alpha) \, \mathfrak{s}^1_l(\beta)$ is removed 
      from the sum.
\end{itemize}
The change of the preamplification factors were considered when we studied the evolution of $\mathcal{V}$. 
Reasoning as for \eqref{Eq:EstQ1-1}, using again $\mathcal{E}_{\alpha\beta} \lesssim \mathcal{V}_1(t^-)^{1/p}$,
 we deduce that the first error can be estimated by
 $\mathcal{O}(1) \mathcal{V}(t^-)^{\frac{2p-1}{p}} \mathcal{E}_{\alpha \beta}$. \par
For what concerns the second error, Proposition~\ref{Pro:ElemePropsVp}--\ref{Item:Fusion} 
 proves that there would be no change if we had $\sigma_{\alpha'}=\sigma_\alpha + \sigma_\beta$. 
Hence, the error is only due to the additional $\mathcal{O}(1)\mathcal{E}_{\alpha\beta}^3$ term. 
This involves with Proposition~\ref{Pro:ElemePropsVp}--\ref{Item:DLVp} that this second error can be measured by 
 $\mathcal{O}(1)\mathcal{E}_{\alpha\beta}^3 \mathcal{V}^1(t^-)^{\frac{2p-1}{p}}$. \par
For what concerns the removed term $\mathfrak{s}^1_r(\alpha) \, \mathfrak{s}^1_l(\beta)$, 
due to Proposition~\ref{Pro:PropsVpEndpoints}--\ref{Item:EndPoint2} and $\sigma_\alpha \sigma_\beta >0$,
 we have that
$$
\mathfrak{s}^1_r(\alpha) = \mathcal{V}^1(\geq \alpha) - \mathcal{V}^1(> \alpha) 
        \geq p |\widehat{\sigma}_\beta|^{p-1}|\widehat{\sigma}_\alpha|
\ \text{ and } \ 
\mathfrak{s}^1_l(\beta) 
        = \mathcal{V}^1(\leq \beta) - \mathcal{V}^1(< \beta) \geq |\widehat{\sigma}_\beta|^{p},
$$
so using \eqref{Eq:Mentre2et12}, the corresponding gain is
\begin{equation} \label{Eq:Loss1}
- \mathfrak{s}^1_r(\alpha) \mathfrak{s}^1_l(\beta) \leq -p\, 2^{-2p} |\sigma_\alpha| |\sigma_\beta|^{2p-1} .
\end{equation}
Estimating in the other way around
$$
\mathfrak{s}^1_r(\alpha) 
        = \mathcal{V}^1(\geq \alpha) - \mathcal{V}^1(> \alpha) \geq |\widehat{\sigma}_\alpha|^p
\ \text{ and } \ 
\mathfrak{s}^1_l(\beta) = \mathcal{V}^1(\leq \beta) - \mathcal{V}^1(< \beta) 
        \geq p |\widehat{\sigma}_\alpha|^{p-1}|\widehat{\sigma}_\beta|,
$$
we also have
\begin{equation*}
- \mathfrak{s}^1_r(\alpha) \mathfrak{s}^1_l(\beta) \leq -p\, 2^{-2p} |\sigma_\alpha|^{2p-1} |\sigma_\beta| .
\end{equation*}
The modification of $\mathcal{Q}_1^2$ has two sources:
\begin{itemize}
\item the various $\mathfrak{s}^2_r(\gamma)$ and $\mathfrak{s}^2_l(\gamma)$ are modified 
      (when $\gamma \leq \alpha$ and $\gamma \geq \alpha$, respectively) due to the change 
      in the preamplification factors that have to take into account the merging of $\alpha$ and $\beta$,
\item the new interaction terms $\mathfrak{s}^2_r(\mu_i) \, \mathfrak{s}^2_l(\gamma)$ and
      $\mathfrak{s}^2_r(\gamma) \mathfrak{s}^2_l(\mu_i)$. 
\end{itemize}
Reasoning as before, using the same function $f$ as for \eqref{Eq:fCoefMC}, we see that the error corresponding
 to the first source can be estimated by 
$\mathcal{O}(1) \mathcal{V}^2(t^-) \Big( \mathcal{V}^2(t^-)^{\frac{p-1}{p}} \mathcal{E}_{\alpha\beta} 
        + \mathcal{E}_{\alpha\beta}^p \Big)$. 
For the second source, we have $\sum_{i=1}^{k} |\mu_i| \lesssim \mathcal{E}_{\alpha\beta}$ which with
 Proposition~\ref{Pro:ElemePropsVp}--\ref{Item:SpSpPrime}, \ref{Item:DLVp} yields 
 $\sum_{i=1}^{k} \mathfrak{s}^2_r(\mu_i) \lesssim \mathcal{V}^2(t^-)^{\frac{p-1}{p}} \mathcal{E}_{\alpha\beta} 
 + \mathcal{E}_{\alpha\beta}^p$, so that we can estimate the second error as the first one. \par
\ \par
\noindent
In total, using again $\mathcal{E}_{\alpha\beta} \lesssim \mathcal{V}_1(t^-)^{1/p}$, we deduce that
\begin{equation} \label{Eq:EstQ1-2}
\mathcal{Q}_1(t^+) - \mathcal{Q}_1(t^-)
\leq -\frac{p\, 2^{-2p}}{2} |\sigma_\alpha| |\sigma_\beta|^{2p-1} 
  -\frac{p\, 2^{-2p}}{2} |\sigma_\alpha|^{2p-1} |\sigma_\beta| 
+ \mathcal{O}(1) \mathcal{V}(t^-)^{\frac{2p-1}{p}} \mathcal{E}_{\alpha \beta}.
\end{equation}
{\bf Evolution of $\mathcal{Q}_2$.}
The functional $\mathcal{Q}_2$ is treated exactly in the same way as $\mathcal{Q}_1$, 
 with the only difference that no term is removed from $\mathcal{Q}_2$ here. 
Hence we get here
\begin{equation} \label{Eq:EstQ2-2}
\mathcal{Q}_2(t^+) - \mathcal{Q}_2(t^-)
= \mathcal{O}(1) \mathcal{V}(t^-)^{\frac{2p-1}{p}}  \mathcal{E}_{\alpha \beta}.
\end{equation}
\ \par
\noindent
{\bf Evolution of $\mathcal{Q}_3$.}
We denote
\begin{equation*}
\mathcal{D}_{\alpha \beta} =
\left\{ \begin{array}{l}
|\sigma_\beta|^3 \text{ if } \alpha \text{ is a shock and } \beta \text{ is a CW}, \\
|\sigma_\alpha|^3  \text{ if } \alpha \text{ is a CW and } \beta \text{ is a shock}, \\
|\sigma_\alpha|^3 + |\sigma_\beta|^3 \text{ if } \alpha \text{ and } \beta \text{ are CW with }
                            \sigma_\alpha + \sigma_\beta \leq - 2{\nu}, \\
0 \text{ otherwise.}
\end{array} \right.
\end{equation*}
Then one deduces from \eqref{Eq:EvolForcesMonot} and the fact that $\alpha'$ is a CW only if $\alpha$
 and $\beta$ are both CW with $\sigma_\alpha + \sigma_\beta > -2\nu$, that
\begin{equation} \label{Eq:VarQ3}
\mathcal{Q}_3(t^+) \leq \mathcal{Q}_3(t^-) 
+ 3 (1+\delta_{CW}(\alpha')) |\sigma_\alpha| |\sigma_\beta| \big( |\sigma_\alpha| + |\sigma_\beta| \big) 
- \mathcal{D}_{\alpha \beta}
+ \mathcal{O}(1) \mathcal{E}_{\alpha\beta}^3.
\end{equation} \par
\ \par
\noindent   
\ \par \noindent
{\bf Conclusion.}
For $C_1$ large enough in \eqref{Eq:DefGlimmFunc}, the decay
 $-\frac{p\, 2^{-2p}}{2} |\sigma_\alpha| |\sigma_\beta|^{2p-1} -\frac{p\, 2^{-2p}}{2} |\sigma_\alpha|^{2p-1} |\sigma_\beta|$ 
 compensates for the second term in the right-hand side of \eqref{Eq:VarQ3}.
 (We can see here that $C_1$ is only needed for $p=\frac{3}{2}$.)
Then under $H(\tau)$, if $V_p(u_0)$ is small enough, and taking $C_2$ large enough
 (which is only needed if $p=1$) the terms 
$-\frac{p\, 2^{-2p}}{2} |\sigma_\alpha| |\sigma_\beta|^{2p-1} -\frac{p\, 2^{-2p}}{2} |\sigma_\alpha|^{2p-1} |\sigma_\beta|$ 
and 
$- \mathcal{D}_{\alpha \beta}$
compensate for all terms containing $\mathcal{E}_{\alpha\beta}$. 
Hence we also find in that case that $\Upsilon(t^+) - \Upsilon(t^-) \leq 0$. \par
%
%
%
%
%
%
%
%
%
\subsubsection{Interactions within a family: non-monotone case}
\label{sss:nonmonotone}
As in the previous paragraph, we consider again two fronts $\alpha$ and $\beta$ 
 of the same family (say, of family $1$), interacting at time $t$, 
 with this time $\sigma_\alpha \sigma_\beta <0$.  \par
%
%
\ \par
\noindent
{\bf Estimate of $\mathcal{V}$.}
The estimates on the evolution of the parts $\mathcal{V}$ of the functional are similar 
 in this case as the ones of the monotone case. 
Even, they are simpler since there is no transformation of a compression wave in a shock here, and the error
 term is systematically
 $\mathcal{E}_{\alpha\beta}= |\sigma_\alpha| |\sigma_\beta| (|\sigma_\alpha| + |\sigma_\beta|)$. 
Hence we find
\begin{equation*} 
\mathcal{V}(t^+) \leq \mathcal{V}(t^-) + \mathcal{O}(1) \mathcal{V}(t^-)^{\frac{p-1}{p}}  |\sigma_\alpha| |\sigma_\beta| 
\left(|\sigma_\alpha| + |\sigma_\beta|\right) .
\end{equation*}
The evolution of $\mathcal{Q}_1$, $\mathcal{Q}_2$ and $\mathcal{Q}_3$ is merely slightly different. \par
\ \par
\noindent
{\bf Estimate of $\mathcal{Q}_1$ and $\mathcal{Q}_2$.}
The evolution in $\mathcal{Q}_1$  and $\mathcal{Q}_2$ of the various $\mathfrak{s}^1_r(\gamma)$ 
 and $\mathfrak{s}^1_l(\gamma)$  due to the change in the preamplification factors is similar
 to the monotone case, but the estimates on the removed term $\mathfrak{s}^1_r(\alpha)\mathfrak{s}^1_l(\beta)$ is 
 different due to non-monotonicity, that is, we can merely estimate here 
 (using again Proposition~\ref{Pro:PropsVpEndpoints}--\ref{Item:EndPoint2} and \eqref{Eq:Mentre2et12})
\begin{equation*} 
- \mathfrak{s}^1_r(\alpha) \mathfrak{s}^1_l(\beta) \leq - 2^{-2p} |\sigma_\alpha|^{p} |\sigma_\beta|^{p} .
\end{equation*}
Hence we have here
\begin{gather*} 
\mathcal{Q}_1^k(t^+) - \mathcal{Q}_1^k(t^-)
\leq - 2^{-2p} |\sigma_\alpha|^{p} |\sigma_\beta|^{p} 
+ \mathcal{O}(1) \mathcal{V}(t^-)^{\frac{2p-1}{p}}  
    |\sigma_\alpha| |\sigma_\beta| \left(|\sigma_\alpha| + |\sigma_\beta|\right), \\
\mathcal{Q}_2(t^+) - \mathcal{Q}_2(t^-) 
= \mathcal{O}(1) \mathcal{V}(t^-)^{\frac{2p-1}{p}} |\sigma_\alpha| |\sigma_\beta| 
\left(|\sigma_\alpha| + |\sigma_\beta|\right) .
\end{gather*}
\ \par
\noindent
{\bf Estimate of $\mathcal{Q}_3$.}
Here we benefit from the cancellation effect between $\sigma_\alpha$ and $\sigma_\beta$. 
Note that due to the behavior of our front-tracking algorithm in the case of same-family
interaction, there cannot be a transformation of a CW into a shock here. \par
Using the fact that $\sigma_\alpha \sigma_\beta <0$, we have
\begin{equation*}
|\sigma_\alpha + \sigma_\beta|^3 - |\sigma_\alpha|^3 - |\sigma_\beta|^3
\leq -\frac{3}{4} \max(|\sigma_\alpha|, |\sigma_\beta|)^2 \min(|\sigma_\alpha|, |\sigma_\beta|).
\end{equation*}
(For $a>b>0$, one has $(a-b)^3 - a^3 - b^3 =- 3 a^2 b + 3 a b^2 - b^3$ and one uses
$3ab \leq b^2 + \frac{9}{4} a^2$.) \par
Adding the errors of \eqref{Eq:EvolForcesMonot}, we find
\begin{equation*}
\mathcal{Q}_3(t^+) - \mathcal{Q}_3(t^-) \leq
-\frac{3}{4} \max(|\sigma_\alpha|, |\sigma_\beta|)^2 \min(|\sigma_\alpha|, |\sigma_\beta|)
+ \mathcal{O}(1) |\sigma_\alpha|^3 |\sigma_\beta|^3 \big( |\sigma_\alpha| + |\sigma_\beta| \big)^3.
\end{equation*}
\ \par \noindent
{\bf Conclusion.} 
Relying on the decay 
 $-\frac{3}{4} \max(|\sigma_\alpha|, |\sigma_\beta|)^2 \min(|\sigma_\alpha|, |\sigma_\beta|)$
 of the $\mathcal{Q}_3$ term,
 we find again in that case that $\Upsilon(t^+) - \Upsilon(t^-) \leq 0$ provided that $V_p(u_0)$ 
is small enough and that $C_2$ is large enough in the case $p=1$.
%
%
%
%
%
%
%
%
%
%
\subsubsection{Measures along curves}
\label{sss:curves}
We consider $\Gamma$ a wave measure curve in the strip $[0,\tau] \times \R$.
Again we only have to consider interaction times and to prove that at such $t$:
\begin{gather} 
\label{Eq:DecayUpsilonGamma1}
\mathcal{V}_\Gamma(t^+) +\mathcal{Q}(t^+)\leq \mathcal{V}_\Gamma(t^-) +\mathcal{Q}(t^-) \\
\label{Eq:DecayUpsilonGamma2}
\mathfrak{M}_\Gamma(t^+) +\mathcal{Q}(t^+)\leq \mathfrak{M}_\Gamma(t^-) +\mathcal{Q}(t^-).
\end{gather}
\ \par
\noindent
{\bf Evolution of $\mathcal{V}_\Gamma$.} 
Since in $\mathcal{V}_\Gamma$ the preamplified forces $\widehat{\sigma}_\gamma$ are frozen after $\gamma$ has met
 $\Gamma$, we only have to consider the evolution of preamplified forces:
\begin{itemize}
\item on the half-line to the left of $(t,\Gamma(t))$ for $1$-lower and $2$-upper measure curves,
\item on the half-line to the right of $(t,\Gamma(t))$ for $2$-lower and $1$-upper measure curves.
\end{itemize}
Call $I_t$ the corresponding half-line at time $t$. 
Then as for $\mathcal{V}$, the change of $\mathcal{V}_\Gamma$ at time $t$ is due to:
\begin{itemize}
\item The change in the strength of the fronts that interact, 
    and the possible strength of the new outgoing waves. 
    This is compensated by $\mathcal{Q}(t^+) - \mathcal{Q}(t^-)$, as shown in the previous paragraphs.
\item The modification of the preamplification factors: here this concerns only fronts that cross $I_t$,
      since other preamplification factors are frozen.
      Notice in particular that when an interaction of opposite families is met on $I_t$, the corresponding
      preamplification factor is not represented in $\mathfrak{M}_\Gamma(t)$, because this interaction takes place
      in the limited component of $(\R_+ \times \R) \setminus \Gamma$.
      Hence no factor $\mathcal{M}_\alpha$ disappears from the preamplified strengths in $\mathcal{V}_\Gamma(t)$.
      Consequently, updating $\mathcal{V}_\Gamma(t)$ corresponds to multiplying the preamplified strengths by 
    $\left(1,\ldots,1, \left(\frac{\mathcal{M}_{\alpha}(t^+)}{\mathcal{M}_{\alpha}(t^-)}\right)_{\alpha > \Gamma(t)}\right)$ 
        if $\Gamma$ is a $1$-lower or a $2$-upper measure curve, and by
    $\left(\left(\frac{\mathcal{M}_{\alpha}(t^+)}{\mathcal{M}_{\alpha}(t^-)}\right)_{\alpha < \Gamma(t)}, 1,\ldots,1,\right)$ 
    if $\Gamma$ is a $2$-lower or a $1$-upper measure curve.
    The $p$-variation $v^p_p\left(\frac{\mathcal{M}_{\alpha}(t^+)}{\mathcal{M}_{\alpha}(t^-)}\right)$ has been estimated in the previous paragraphs. ``Flattening'' the sequence $\frac{\mathcal{M}_{\alpha}(t^+)}{\mathcal{M}_{\alpha}(t^-)}$ on the left/right of $\Gamma(t)$, this gives:
\begin{eqnarray}
\nonumber 
v_p\left(1,\ldots,1, \left(\frac{\mathcal{M}_{\alpha}(t^+)}{\mathcal{M}_{\alpha}(t^-)} \right)_{\alpha > \Gamma(t)}\right)
&=& v_p\left(0,\ldots,0, \left(\frac{\mathcal{M}_{\alpha}(t^+)}{\mathcal{M}_{\alpha}(t^-)} \right)_{\alpha > \Gamma(t)}-1\right) \\
\nonumber 
&\leq& v_p\left( \left(\frac{\mathcal{M}_{\alpha}(t^+)}{\mathcal{M}_{\alpha}(t^-)}\right)-1 \right)
+ \left\| \left(\frac{\mathcal{M}_{\alpha}(t^+)}{\mathcal{M}_{\alpha}(t^-)}\right)-1 \right\|_\infty \\
\label{Eq:EstvpMsM}
&=& v_p\left( \left(\frac{\mathcal{M}_{\alpha}(t^+)}{\mathcal{M}_{\alpha}(t^-)}\right) \right)
+ \left\| \left(\frac{\mathcal{M}_{\alpha}(t^+)}{\mathcal{M}_{\alpha}(t^-)}\right)-1 \right\|_\infty .
\end{eqnarray}
The right-hand side has been estimated in the above sections; it can be absorbed by $\mathcal{Q}(t^+)-\mathcal{Q}(t^-)$,
see in particular \eqref{Eq:EstOscilCoefs} and \eqref{Eq:ChgtPreampFac2}-\eqref{Eq:fCoefMC}.
\end{itemize}
\ \par
\noindent
{\bf Evolution of $\mathfrak{M}_\Gamma$.} 
The reasoning for $\mathfrak{M}_\Gamma$ is similar as for the preamplification factors in the case of $\mathcal{V}_\Gamma$.
However the ``flattening'' of the sequence does not correspond to the same set of indices,
 since only future interactions taking place in the unlimited component of $(\R_+ \times \R) \setminus \Gamma$
 before time horizon $\tau$ are considered.

The main point is as follows. 
Consider for instance the $2$-front $\beta$ involved in the interaction at time $t$.
Then if $\alpha_1$ and $\alpha_2$ are two $1$-fronts have a future interaction
 with $\beta$ in the unlimited component of $(\R_+ \times \R) \setminus \Gamma$
 before time horizon $\tau$, then the same is true for any $1$-front between $\alpha_1$ and $\alpha_2$.
So the reasoning is the same for $\mathfrak{M}_\Gamma$ as for $\mathcal{V}_\Gamma$, multiplying
 by a sequence of the form 
$\left(1,\ldots,1, 
\left( \frac{\mathcal{M}_{\alpha}(t^+)}{\mathcal{M}_{\alpha}(t^-)}
         \right)_{\underline{\alpha} \leq \alpha \leq \overline{\alpha}},
1,\ldots,1 \right)$,
where $\underline{\alpha}$ and $\overline{\alpha}$ are the extremal fronts that have a future interaction
 taking place in the unlimited component of $(\R_+ \times \R) \setminus \Gamma$ before time horizon $\tau$. \par
Hence reasoning as before, we find
\begin{equation*}
\mathfrak{M}_\Gamma(t^+) - \mathfrak{M}_\Gamma(t^-) 
\lesssim \mathfrak{M}_\Gamma(t^{-})^{\frac{p-1}{p}} \, 
    v_p\left(0,\ldots,0, \left(\frac{\mathcal{M}_{\alpha}(t^+)}{\mathcal{M}_{\alpha}(t^-)} -1 \right)_{\alpha > \Gamma(t)},0,\ldots,0\right),
\end{equation*}
which we estimate as in \eqref{Eq:EstvpMsM}.
%
%
%
%
%
%
%
%
%
\subsection{Propagating $H(\tau)$: proof of Proposition~\ref{Pro:Induction}}
\label{Subsec:propagation}
Proposition~\ref{Pro:Induction} will be a direct consequence of the following lemma.
\begin{lemma} \label{Lem:Induction}
There exists $c>0$ such that if $V_p(u_0)<c$, then for $\nu \in (0,\nu_0)$ the following holds.
Suppose that for some $\tau \in (0,T_\nu)$, 
the front-tracking approximation $u^\nu$ satisfies $H(\tau)$. 
Then there exists $\delta>0$ such that  $u^\nu$ satisfies $H(\tau + \delta)$.
\end{lemma}
\begin{proof}[Proof of Lemma~\ref{Lem:Induction}]
As a first condition on $c$, we ask that $c < C^{-1/3}$, where $C$ is the constant in \eqref{Eq:VVtilde}.
As a consequence, using the first condition in \eqref{Eq:ApproxInitiale}, we have
\begin{equation} \label{Eq:EstVtildeInit}
\widetilde{V}_p(u^\nu(0,\cdot)) \leq V_p(u_0) + C V_p(u_0)^3 \leq 2 V_p(u_0).
\end{equation}

We fix $\tau \in (0,T_\nu)$ and suppose that $H(\tau)$ is satisfied. The conclusion of Lemma~\ref{Lem:Induction} is trivial 
if $\tau$ is not an interaction time.
Hence, we suppose that $\tau$ is an interaction time and choose $\delta >0$ such that $\tau + \delta$ is less
than the next interaction time. \par
We introduce a measure curve $\Gamma$ in the strip $[0,\tau + \delta] \times \R$
with $t_{\max} \in (\tau,\tau+\delta]$ 
(other wave measure curves being already taken into account in $H(\tau)$).
The goal is to prove that for $t \in (\tau,t_{\max}]$, $\widetilde{V}_p(u^\nu(t,\cdot))$ and $\widetilde{V}_p[u^\nu;\Gamma](t)$ satisfy \eqref{Eq:H1tau}. \par
\ \par
\noindent
{\it Estimate of $\widetilde{V}_p(u^\nu(t,\cdot))$.}
Since $H(\tau)$ is valid, using the time-horizon $\tau$ for the functionals, we deduce from Proposition~\ref{Pro:DecayGlimm}
 that $\Upsilon$ is non-increasing on $[0,\tau]$. 
Hence $\Upsilon(\tau^-) \leq \Upsilon(0)$, and in particular 
\begin{equation*}
\mathcal{V}(\tau^-) \leq \mathcal{V}(0) + \mathcal{O}(1) \mathcal{V}(0)^2.
\end{equation*}
(We could notice here that since the time-horizon is $\tau$, one has $\mathcal{V}(\tau^-)=\widetilde{V}_p^p(u^\nu(\tau^-,\cdot))$
 though this is not strictly necessary.)
Using Proposition~\ref{Pro:EstimeesVetQ}, we infer that
\begin{eqnarray}
\nonumber
\widetilde{V}^p_p(u^\nu(\tau^-,\cdot)) 
&\leq&  \,V_p^p(u_0) (1 + \mathcal{O}(1) V_p(u_0)^3)^2 + \mathcal{O}(1) V^{2p}_p(u_0)
            (1 + \mathcal{O}(1) V_p(u_0)^3)^3 \\
\label{Eq:EstVpEnfin}
&\leq& \, V_p^p(u_0) + \mathcal{O}(1) V^{2p}_p(u_0).  
\end{eqnarray}
Hence provided that $V_p(u_0)$ is small enough, one has
\begin{equation*}
\widetilde{V}_p(u^\nu(\tau^-,\cdot)) \leq 3 V_p(u_0).  
\end{equation*}
Now the production at the interaction point at time $\tau$ cannot be larger than $\mathcal{O}(1)\widetilde{V}_p(u^\nu(\tau^-,\cdot))^3$,
 according to Propositions~\ref{Prop:InteractionDF} and \ref{Prop:InteractionSF}. 
Hence measuring the $p$-sum by the $1$-one, we see that 
\begin{equation*}
\widetilde{V}_p(u^\nu(\tau^+,\cdot)) \leq \widetilde{V}_p(u^\nu(\tau^-,\cdot))  + \mathcal{O}(1){V}^3_p(u^\nu(\tau^-,\cdot)) 
\leq 4 V_p(u_0),
\end{equation*}
provided again that $V_p(u_0)$ is small enough. \par
\ \par
\noindent
{\it Estimate of $\widetilde{V}_p[u^\nu;\Gamma](t)$.}
To the curve $\Gamma$, we can associate the intermediate curves $\Gamma_t$ as in Section~\ref{Subsec:APA}. 
We let $\check{\Gamma}=\Gamma_{\tau^-}$.
Now we use the functionals $\mathcal{V}_{\check{\Gamma}}$ and $\mathfrak{M}_{\check{\Gamma}}$ associated
 with the curve $\check{\Gamma}$ and the time horizon $T=\tau$. 
Using Proposition~\ref{Pro:DecayGlimm} and the fact that $H(\tau)$ is satisfied, we have
\begin{equation*}
\mathcal{V}_{\check{\Gamma}}(\tau^-) \leq \mathcal{V}_{\check{\Gamma}}(0) + \mathcal{O}(1) \mathcal{V}(0)^2
        \leq 2 V_p(u_0) + \mathcal{O}(1)V_p(u_0)^{2p}
\ \text{ and } \ 
\mathfrak{M}_{\check{\Gamma}}(\tau^-) \leq \mathfrak{M}_{\check{\Gamma}}(0) + \mathcal{O}(1) \mathcal{V}(0)^2
        \lesssim V_p(u_0)^{2p},
\end{equation*}
where we used Proposition~\ref{Pro:EstimeesVetQ}. \par
Using this estimate on $\mathfrak{M}_{\check{\Gamma}}(\tau^-)$ and  reasoning as for \eqref{Eq:Mmoins1},
 we obtain for the inverses of the preamplification factors along $\check{\Gamma}$
 (using the notations of \eqref{Eq:MiGamma}, that is calling $\gamma^i_1,\ldots,\gamma^i_{n_i}$ the 
 $i$-fronts crossing $\check{\Gamma}$):
\begin{equation*} 
v_p^p \left( \left({\widetilde{\mathcal{M}}^{\check{\Gamma}}}_{\gamma^i_1}\right)^{-1}, \ldots, 
             \left({\widetilde{\mathcal{M}}^{\check{\Gamma}}}_{\gamma^i_{n_i}}\right)^{-1} \right)
\lesssim V_p(u_0)^{2p},
\end{equation*}
where only the family $i=j$ is considered if $\check{\Gamma}$ a $j$-upper curve. \par
Now we use the fact that for fronts $\alpha$ crossing $\check{\Gamma}$, the modified interaction coefficients 
 $\tilde{\mathcal{C}}^{i,\check{\Gamma}}_{\beta\alpha}$ coincide with their unmodified version $\mathcal{C}^{i}_{\beta\alpha}$,
 since both only consider possible interactions in the unlimited component of $\R_+ \times \R \setminus \check{\Gamma}$ before
 the time horizon $\tau$. (Recall that for an $i$-upper measure curve, only fronts of family $i$ are considered 
 in  $\mathfrak{M}$, see \eqref{Eq:MGamma}.) 
If follows that
\begin{equation*}
\widetilde{V}_p[u^\nu;\check{\Gamma}](\tau^-) 
= s_p \Big(\widetilde{\mathcal{M}}^{\check{\Gamma}}_{\gamma^i_1}(t)^{-1} \, \widehat{\sigma}_{\gamma^i_1}(t),
     \ldots, \widetilde{\mathcal{M}}^{\check{\Gamma}}_{\gamma^i_{n_i}}(t)^{-1} \, \widehat{\sigma}_{\gamma^i_{n_i}}(t) \Big).
\end{equation*}
Consequently, Corollary~\ref{Cor:MultiplSp}, we find 
\begin{equation*} 
\widetilde{V}_p[u^\nu;\check{\Gamma}](\tau^-) \leq 2 V_p(u_0) (1 + \mathcal{O}(1) V_p(u_0)^{2p}) 
\leq 3 V_p(u_0),
\end{equation*}
provided that $V_p(u_0)$ is small enough. 
Again, the production at the interaction point at time $\tau$ cannot be larger than $\mathcal{O}(1){V}_p(u^\nu(\tau^-,\cdot))^3$,
 so here we find 
\begin{equation*}
\widetilde{V}_p[u^\nu;\Gamma](\tau^+) 
= \widetilde{V}_p[u^\nu;\Gamma](\tau^-) + \mathcal{O}(1){V}_p(u_0)^3
= \widetilde{V}_p[u^\nu;\check{\Gamma}](\tau^-) + \mathcal{O}(1){V}_p(u_0)^3.
\end{equation*}
The conclusion follows in the same way as for $\widetilde{V}_p[u^\nu]$.
This ends the proof of Lemma~\ref{Lem:Induction}.
\end{proof}
\begin{proof}[Proof of Proposition~\ref{Pro:Induction}]
Consider a front-tracking approximation $u^\nu$ for $\nu \in (0,\nu_0)$.
Then for positive times $\tau$ before the first interaction, $u^\nu$ satisfies $H(\tau)$ since in that case,
for any $t \in [0,\tau]$ and for any measure curve $\Gamma$ in the strip $[0,\tau] \times \R$, one gets from
\eqref{Eq:EstVtildeInit} that
$\widetilde{V}_p(u^\nu(t,\cdot)) = \widetilde{V}_p[u^\nu;\Gamma](t) \leq 2 V_p(u_0)$. \par
We let
%
$T_\star = \sup\, \{ \tau \in (0,T_\nu) \ / \ H(\tau) \text{ is valid}\}$.
%
If we had $T_*<T_\nu$, then $T_*$ would be an interaction time, and Lemma~\ref{Lem:Induction} would allow
 going beyond $T_*$, bringing a contradiction. 
The conclusion follows.
\end{proof}
%
%
%
%
%
%
%
%
%
%
\subsection{End of the proof}
\subsubsection{Global existence of front-tracking approximations}
The first step to conclude is to prove that $T_\nu=+\infty$, where $T_\nu>0$ was defined in Subsection~\ref{Subsec:FTA}. 
In what follows we prove that the various conditions limiting $T_\nu$ do not occur in finite time. \par
\ \par
\noindent
{\bf Finite number of fronts.}
Let us prove that the number of fronts remains bounded on $[0,T_\nu)$.
This argument is classical.
In the above front-tracking algorithm, when two fronts meet, there can be more than to outgoing fronts 
 only in the case of a same-family interaction, with an outgoing wave in the opposite family is a rarefaction
 of strength at least $\nu>0$. 
Due to the analysis of Paragraphs~\ref{sss:monotone} and \ref{sss:nonmonotone}, this implies a decay 
 of size $\nu$ in $\Upsilon$, whatever the choice of time horizon $T$.
Since $\Upsilon$ can be bounded independently of the time horizon,
 such an event happens only finitely many times, and the total number of fronts remains bounded. \par
\ \par
\noindent
{\bf Admissibility of states.}
Let us prove that we have an a priori bound on $\|u^\nu - \overline{u}\|_\infty$ in $[0,T_\nu)$, 
 so all states remain in $U$ and $T_\nu$ cannot be limited by Riemann problems needing states outside.
It is actually a direct consequence of $H(\tau)$ being valid for all $\tau < T_\nu$.
Since $u^\nu$ is constant on the left and on the right a compact set in space,
 with states $u^\nu(-\infty)$ and $u^\nu(+\infty)$ satisfying 
 $\| u^\nu(\pm \infty) - \overline{u}\|_\infty \leq \| u_0 - \overline{u}\|_\infty$,
 using the bound on $\widetilde{V}_p(u^\nu(t,\cdot))$ and
 choosing $\varepsilon_0>0$ in \eqref{Eq:Smallness} suitably small ensures that all states remain
 in $B(\overline{u};r)$ for $r$ fixed as previously. \par
\ \par
\noindent
{\bf Non accumulation of fronts.}
The analysis here borrows arguments from \cite{BressanColombo95}, but the front-tracking algorithm that we use allows
 simplifying a bit the analysis. \par
Suppose that fronts accumulate at time $T_\nu$. 
We call an extended front line a continuous locally piecewise affine line $t \mapsto \gamma(t)$ defined on some
 interval $[t_0,T_\nu)$, which follows fronts of a given characteristic family.
Due to finite speed of propagation, all extended front lines are uniformly Lipschitz and can be propagated 
 till time $T_\nu$ included.
Due to the finite number of fronts, accumulations can occur only at a finite number of places
 (that are final points of extended front lines).
Pick $x^*$ one of them.
Due to the previous analysis, we can introduce a rectangle $[T_\nu-h,T_\nu] \times [x^*-h,x^*+h]$, 
 with $h>0$ small such that:
\begin{itemize}
\item there is no interaction point with more than two outgoing fronts in $[T_\nu-h,T_\nu] \times [x^*-h,x^*+h]$.
\end{itemize}
Reducing again $h$ if necessary, we can also require that 
\begin{itemize}
\item there is no interaction point
 in $[T_\nu-h,T_\nu] \times [x^*-h,x^*+h]$ for which an incoming compression front is transformed in a shock front.
There are indeed only a finite number of such points where this occurs, due to the analysis 
 of Subsection~\ref{sss:monotone}: see in particular the $\mathcal{D}_{\alpha\beta}$ term in \eqref{Eq:EstQ1-2},
 which gives a minimal decay on $\Upsilon$ independently of the time horizon. 
\end{itemize}
It follows that the number of shock fronts is non-increasing in $[T_\nu-h,T_\nu] \times [x^*-h,x^*+h]$; reducing again
 $h$ if necessary, we can suppose that 
\begin{itemize}
\item the number of shock fronts is constant in $[T_\nu-h,T_\nu] \times [x^*-h,x^*+h]$. 
\end{itemize}
Since an incoming shock front generates an outgoing shock front in the same family (unless it disappears), we can follow
 these shocks as disjoint front lines in $[T_\nu-h,T_\nu)$.
It follows that we can reduce $h>0$ in order that
\begin{itemize}
\item all extended shock front lines in $[T_\nu-h,T_\nu] \times [x^*-h,x^*+h]$ converge to $x^*$ as $t \rightarrow T_\nu$.
\end{itemize}
We denote them $x_1(t) < \ldots < x_N(t)$ where $x_1(t),\ldots,x_k(t)$ correspond to fronts of the second family,
and $x_{k+1}(t),\ldots,x_N(t)$ correspond to fronts of the first family. \par
Now other fronts present in $[T_\nu-h,T_\nu] \times [x^*-h,x^*+h]$ are either rarefaction waves or compression waves. 
We define $\mathcal{N}(t)$ as the number of such fronts (RW or CW) of family $1$ and to the right of $x_k(t)$ plus
 the number of such fronts (RW or CW) of family $2$ and to the left of $x_{k+1}(t)$.
Interactions involving such fronts between themselves do not generate other fronts: they cross if it is an interaction 
 of opposite families, and they merge if it is an interaction of the same family. 
When an $i$-front of RW/CW type meets a $(3-i)$-shock, they cross as well. 
Finally, when an $i$-front of RW/CW type meets an $i$-shock, they merge as a shock and generate a RW/CW front in the other family,
 so $\mathcal{N}(t)$ decreases, because this generated $(3-i)$-front is to the left of $x_k(t)$ (if $i=2$) and to the right
 of $x_{k+1}(t)$ (if $i=1$). 
It follows that $\mathcal{N}(t)$ is non-increasing, and decreases at each interaction of a RW/CW wave with a shock. \par
Hence, there are only finitely many interactions of this type.
Reducing $h$ again, we can hence obtain
\begin{itemize}
\item The only interactions in $[T_\nu-h,T_\nu) \times [x^*-h,x^*+h]$ concern CW/RW between themselves.
\end{itemize}
Such an interaction in the same family makes the number of fronts decrease. 
Hence interactions of CW/RW of the same family are finite, and reducing $h>0$ again we obtain
\begin{itemize}
\item The only interactions in $[T_\nu-h,T_\nu) \times [x^*-h,x^*+h]$ concern CW/RW of opposite families.
\end{itemize}
In the case of interaction of CW/RW opposite families, the fronts simply cross.
Hence we can follow $n_1$ the lines of $1$-CW/RW in $[T_\nu-h,T_\nu) \times [x^*-h,x^*+h]$ and 
 $n_2$ the lines of $2$-CW/RW in $[T_\nu-h,T_\nu) \times [x^*-h,x^*+h]$. 
These generate at most $n_1 n_2$ interactions. \par
Finally, we see that there are only finitely many interaction points in $[T_\nu-h,T_\nu] \times [x^*-h,x^*+h]$, 
 which contradicts the fact that fronts accumulate at $(T_\nu,x^*)$. \par
\ \par
This finally proves that $T_\nu = +\infty$ for all $\nu \in (0,\nu_0)$. \hfill \qedsymbol \par
\ \par
\noindent
Note that $H(\tau)$ being valid for all $\tau \in (0,T_\nu)=(0,+\infty)$ gives us that \eqref{Eq:EstVpEnfin}
 is valid for all times.
Using Lemma~\ref{Lem:VVtilde}, we deduce
\begin{equation} \label{Eq:HolyGraal}
\forall t \geq 0, \ \ V_p(u^\nu(t,\cdot)) \leq V_p(u_0) + \mathcal{O}(1)V_p(u_0)^{2p}.
\end{equation}
(which will eventually give \eqref{Eq:EstVpGlobale}). \par
%
%
%
%
%
%
%
\subsubsection{Time estimates}
The goal of this paragraph is to prove the following proposition.
\begin{proposition} \label{Pro:TimeEstimates}
The sequence $(u^\nu)$ is bounded in $C^{1/p}([0,+\infty);L^p_{loc}(\R))$.
\end{proposition}
The result of Proposition~\ref{Pro:TimeEstimates} is not as direct as the usual ${Lip}(\R_+;L^1_{loc}(\R))$
 boundedness of the $BV$ case. 
Some arguments here are borrowed from \cite{JenssenRidderPhi20}.
Proposition~\ref{Pro:TimeEstimates} relies on the following lemma.
\begin{lemma} \label{Lem:EstVplocale}
For some $C>0$, one has the following.
Let 
\begin{equation*}
\hat{\lambda} > \max\Big(\sup_\Omega |\lambda_1(u)|, \sup_\Omega \lambda_2(u)\Big).
\end{equation*}
Let $\nu \in (0,\nu_0)$.
Given $t>0$ the following holds for almost all $x^* \in \R$: for $0\leq s<t$, one has:
\begin{equation} \label{Eq:EstVplocale}
|u^\nu(t,x^*) - u^\nu(s,x^*)| \leq C \, V_p \Big(u^\nu(s,\cdot); [x^*-\hat{\lambda}(t-s), x^*+\hat{\lambda}(t-s)]\Big).
\end{equation}

\end{lemma}
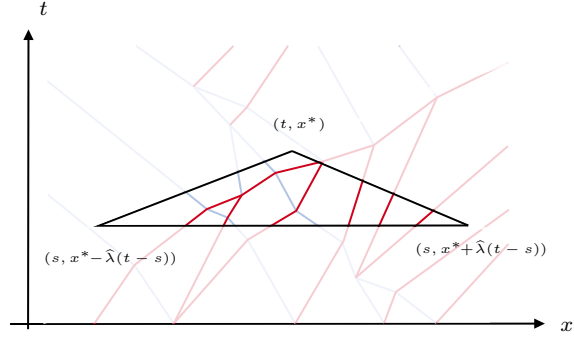
\begin{figure}[ht]
\centering
\tikzset{every picture/.style={line width=0.75pt}} 

\begin{tikzpicture}[x=0.75pt,y=0.75pt,yscale=-1,xscale=1]

\draw    (111,250) -- (375,249.8) ;
\draw [shift={(378,249.8)}, rotate = 179.96] [fill={rgb, 255:red, 0; green, 0; blue, 0 }  ][line width=0.08]  [draw opacity=0] (5.36,-2.57) -- (0,0) -- (5.36,2.57) -- cycle    ;

\draw [line width=0.75]    (120,252.6) -- (120.2,104.8) ;
\draw [shift={(120.2,101.8)}, rotate = 90.08] [fill={rgb, 255:red, 0; green, 0; blue, 0 }  ][line width=0.08]  [draw opacity=0] (5.36,-2.57) -- (0,0) -- (5.36,2.57) -- cycle    ;
\draw [color={rgb, 255:red, 177; green, 196; blue, 228 }  ,draw opacity=1 ]   (202.53,130.53) -- (243.47,174) -- (253.83,193.08) -- (276.75,208.83) -- (283.53,226.73) -- (297.73,249.5) ;

\draw [color={rgb, 255:red, 177; green, 196; blue, 228 }  ,draw opacity=1 ]   (202.53,130.53) -- (229.17,141.06) -- (267.27,168.2) ;

\draw [color={rgb, 255:red, 177; green, 196; blue, 228 }  ,draw opacity=1 ]   (264.79,110.25) -- (292.67,160) ;

\draw [color={rgb, 255:red, 177; green, 196; blue, 228 }  ,draw opacity=1 ]   (129.52,127.71) -- (209.27,192.2) -- (219.62,196.48) -- (229.62,207.62) -- (253.35,249.75) ;

\draw [color={rgb, 255:red, 177; green, 196; blue, 228 }  ,draw opacity=1 ]   (129.81,184.86) -- (173.07,220.4) -- (192.73,249.75) ;

\draw [color={rgb, 255:red, 177; green, 196; blue, 228 }  ,draw opacity=1 ]   (305.04,110) -- (324.17,136) ;

\draw [color={rgb, 255:red, 177; green, 196; blue, 228 }  ,draw opacity=1 ]   (173.79,110) -- (202.53,130.53) ;

\draw [color={rgb, 255:red, 177; green, 196; blue, 228 }  ,draw opacity=1 ]   (283.53,226.73) -- (303.53,233.93) -- (323.73,249.53) ;

\draw [color={rgb, 255:red, 177; green, 196; blue, 228 }  ,draw opacity=1 ]   (220.81,149.86) -- (226.67,185.2) ;

\draw [color={rgb, 255:red, 208; green, 2; blue, 27 }  ,draw opacity=1 ]   (359.79,209.75) -- (323.73,249.53) ;

\draw [color={rgb, 255:red, 208; green, 2; blue, 27 }  ,draw opacity=1 ]   (192.73,249.75) -- (229.62,207.62) -- (253.83,193.08) -- (267.27,168.2) ;

\draw [color={rgb, 255:red, 208; green, 2; blue, 27 }  ,draw opacity=1 ]   (152.9,250) -- (173.07,220.4) -- (209.27,192.2) -- (226.67,185.2) -- (243.47,174) -- (267.27,168.2) ;

\draw [color={rgb, 255:red, 208; green, 2; blue, 27 }  ,draw opacity=1 ]   (253.35,249.75) -- (277.58,208) -- (292.67,160) -- (324.17,136) -- (359.57,118.4) ;

\draw [color={rgb, 255:red, 208; green, 2; blue, 27 }  ,draw opacity=1 ]   (360.04,192.5) -- (303.53,233.93) -- (297.73,249.5) ;

\draw [color={rgb, 255:red, 208; green, 2; blue, 27 }  ,draw opacity=1 ]   (324.17,136) -- (283.53,226.73) ;

\draw [color={rgb, 255:red, 208; green, 2; blue, 27 }  ,draw opacity=1 ]   (222.53,110.13) -- (202.53,130.53) ;

\draw [color={rgb, 255:red, 208; green, 2; blue, 27 }  ,draw opacity=1 ]   (359.54,160.5) -- (283.53,226.73) ;

\draw [color={rgb, 255:red, 208; green, 2; blue, 27 }  ,draw opacity=1 ]   (192.73,249.75) -- (219.62,196.48) -- (226.67,185.2) ;

\draw [color={rgb, 255:red, 208; green, 2; blue, 27 }  ,draw opacity=1 ]   (267.27,168.2) -- (292.67,160) ;

\draw [color={rgb, 255:red, 208; green, 2; blue, 27 }  ,draw opacity=1 ]   (249.79,110) -- (229.17,141.06) -- (220.81,149.86) ;

\draw    (251.8,163) -- (339.71,200.43) ;

\draw    (339.71,200.43) -- (154.86,200.71) -- (251.8,163) ;

\draw [draw opacity=0]   (235.2,195.4) -- (335.2,295.4) ;

\draw [draw opacity=0][fill={rgb, 255:red, 255; green, 255; blue, 255 }  ,fill opacity=0.8 ]   (122,213) -- (364.86,117.29) -- (123.48,89.46) ;

\draw [draw opacity=0][fill={rgb, 255:red, 255; green, 255; blue, 255 }  ,fill opacity=0.8 ]   (129.43,249.57) -- (350.56,249.33) -- (363.14,201.29) -- (123.14,201) -- (127.71,250.14) ;

\draw [draw opacity=0][fill={rgb, 255:red, 255; green, 255; blue, 255 }  ,fill opacity=0.8 ]   (359.14,113.57) -- (250.86,161.86) -- (364,209.86) ;

\draw (124.27,86.3) node [anchor=north west][inner sep=0.75pt]  [font=\footnotesize]  {$t$};
\draw (384.2,247.26) node [anchor=north west][inner sep=0.75pt]  [font=\footnotesize]  {$x$};
\draw (240.5,144.35) node [anchor=north west][inner sep=0.75pt]  [font=\tiny]  {$(t,x^{*})$};
\draw (126.4,210.97) node [anchor=north west][inner sep=0.75pt]  [font=\tiny]  {$(s,x^{*}\!\!-\!\widehat{\lambda} (t-s))$};
\draw (311.64,206.26) node [anchor=north west][inner sep=0.75pt]  [font=\tiny]  {$(s,x^{*}\!\!+\!\widehat{\lambda} (t-s))$};

\end{tikzpicture}
\vspace*{-10mm}
\caption{Domain for local estimates}
\end{figure}
\begin{proof}[Proof of Lemma~\ref{Lem:EstVplocale}] 
Let $t>0$.
As a first condition on $x^*$, we suppose that it is a continuity point of$u^\nu$. 
We consider at time $s$ the initial condition:
\begin{equation*}
u^*_s(x) := \left\{ \begin{array}{l}
u^\nu(s,x) \ \text{ on } \ [x^*-\hat{\lambda}(t-s), x^* + \hat{\lambda}(t-s)], \\
u^\nu(s,(x^*-\hat{\lambda}(t-s))^+) \ \text{ on } \  (-\infty,x^*-\hat{\lambda}(t-s)), \\
u^\nu(s,(x^*+\hat{\lambda}(t-s))^-) \ \text{ on } \ (x^*+\hat{\lambda}(t-s),+\infty).
\end{array} \right. 
\end{equation*}
Corresponding to this initial condition, we construct a piecewise constant function
 on a polygonal partition on $[s,+\infty) \times \R$, that we call $v^\nu$, as follows.
Inside the triangle joining $(s,x^*-\hat{\lambda}(t-s))$, $(s,x^*+\hat{\lambda}(t-s))$ and $(t,x^*)$,
 we let $v^\nu= u^\nu$, and outside we construct it progressively for times $s' \geq s$. 
We will refer to segments joining $(t,x^*)$ to $(s,x^*-\hat{\lambda}(t-s))$ and to $(s,x^*+\hat{\lambda}(t-s))$ 
 as the sides of the triangle (we do not consider the horizontal part of the triangle as a side here).
Without loss of generality, we can suppose that no interaction point of $u^\nu$ lies on these sides. 
Indeed, the set of points $x^*$ such that starting from $(t,x^*)$ we find an interaction point on the sides 
 of the corresponding triangle is negligible.
Now we construct $v^\nu$ according to the following algorithm. For $s'$ close to $s$, we let
\begin{equation*}
v^\nu(s',x) := \left\{ \begin{array}{l}
u^\nu(s,(x^*-\hat{\lambda}(t-s))^+) \ \text{ on } \  (-\infty,x^*-\hat{\lambda}(t-s')), \\
u^\nu(s',x) \ \text{ on } \  [x^*-\hat{\lambda}(t-s'),x^*+\hat{\lambda}(t-s')], \\
u^\nu(s,(x^*+\hat{\lambda}(t-s))^-) \ \text{ on } \ (x^*+\hat{\lambda}(t-s'),+\infty),
\end{array} \right. 
\end{equation*}
as long as $s'$ is sufficiently close to $s$ so that no front in $u^\nu$ meets one of the sides of the triangle. 
Then
\begin{itemize}
\item when a front in $u^\nu$ meets one of the left side of the triangle (in which case it is a $1$-front),
 we extend it as a front separating the same states, {\it at speed } $-\widehat{\lambda}$;
\item when a front in $u^\nu$ meets one of the right side of the triangle (in which case it is a $2$-front),
 we extend it as a front separating the same states, {\it at speed } $\widehat{\lambda}$.
\end{itemize}
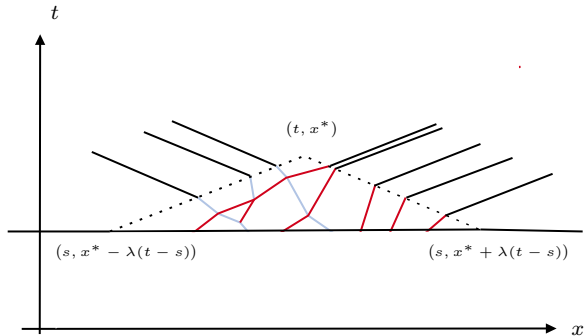
\begin{figure}[ht]
\centering
\tikzset{every picture/.style={line width=0.75pt}} 

\begin{tikzpicture}[x=0.75pt,y=0.75pt,yscale=-1,xscale=1]

\draw    (111,250) -- (375,249.8) ;
\draw [shift={(378,249.8)}, rotate = 179.96] [fill={rgb, 255:red, 0; green, 0; blue, 0 }  ][line width=0.08]  [draw opacity=0] (5.36,-2.57) -- (0,0) -- (5.36,2.57) -- cycle    ;

\draw [line width=0.75]    (120,252.6) -- (120.2,104.8) ;
\draw [shift={(120.2,101.8)}, rotate = 90.08] [fill={rgb, 255:red, 0; green, 0; blue, 0 }  ][line width=0.08]  [draw opacity=0] (5.36,-2.57) -- (0,0) -- (5.36,2.57) -- cycle    ;

\draw [color={rgb, 255:red, 177; green, 196; blue, 228 }  ,draw opacity=1 ]   (238,168) -- (243.5,174) -- (254,193) -- (277,209) ;
\draw [color={rgb, 255:red, 177; green, 196; blue, 228 }  ,draw opacity=1 ]   (199,184) -- (209,192) -- (220,196.5) -- (230,208) ;
\draw [color={rgb, 255:red, 177; green, 196; blue, 228 }  ,draw opacity=1 ]   (225,173) -- (227,185) ;
%
%
%
%
\draw [color={rgb, 255:red, 208; green, 2; blue, 27 }  ,draw opacity=1 ]   (324,136) -- (283.5,227) ;
\draw [color={rgb, 255:red, 208; green, 2; blue, 27 }  ,draw opacity=1 ]   (278,208) -- (293,160) -- (324,136) -- (360,118) ;
\draw [color={rgb, 255:red, 208; green, 2; blue, 27 }  ,draw opacity=1 ]   (230,208) -- (254,193) -- (269,167) ;
\draw [color={rgb, 255:red, 208; green, 2; blue, 27 }  ,draw opacity=1 ]   (173,220) -- (209,192) -- (227,185) -- (243,174) -- (265,168) ;
\draw [color={rgb, 255:red, 208; green, 2; blue, 27 }  ,draw opacity=1 ]   (360,160.5) -- (283.5,227) ;
\draw [color={rgb, 255:red, 208; green, 2; blue, 27 }  ,draw opacity=1 ]   (220,196.5) -- (227,185) ;
%
%
%
\draw [dash pattern={on 1pt off 3pt}]  (155,201) -- (252,163) -- (340,200) ;

\draw (340,200) -- (155,201) ;

\draw [draw opacity=0][fill={rgb, 255:red, 255; green, 255; blue, 255 }  ,fill opacity=1 ]   (129,249) -- (351,249) -- (363,201) -- (123,201.5) -- (128,250) ;
\draw [draw opacity=0][fill={rgb, 255:red, 255; green, 255; blue, 255 }  ,fill opacity=1 ]   (359,114) -- (251,162) -- (364,210) ;

\draw    (146,161) -- (199,184) ;
\draw    (172,151) -- (225,173) ;
\draw    (186,145.5) -- (238,168) ;
\draw    (318,147) -- (265,168) ;
\draw    (376,172) -- (322.5,193) ;
\draw    (356,164) -- (303,184) ;
\draw    (340,157) -- (287,178) ;
\draw    (321,149) -- (267.5,169.5) ;
\draw    (104,201) -- (155,201) ;
\draw    (340,200) -- (391,201) ;

\draw (124,86) node [anchor=north west][inner sep=0.75pt]  [font=\footnotesize]  {$t$};
\draw (384,247) node [anchor=north west][inner sep=0.75pt]  [font=\footnotesize]  {$x$};
\draw (241,144) node [anchor=north west][inner sep=0.75pt]  [font=\tiny]  {$\left( t,x^{*}\right)$};
\draw (126,206) node [anchor=north west][inner sep=0.75pt]  [font=\tiny]  {$\left( s,x^{*} -\lambda ( t-s)\right)$};
\draw (312,206) node [anchor=north west][inner sep=0.75pt]  [font=\tiny]  {$\left( s,x^{*} +\lambda ( t-s)\right)$};

\end{tikzpicture}
\caption{The function $v_\nu$}
\end{figure}
This allows to globally construct $v^\nu$. \par
\ \par
Now as mentioned before, the precise speed of propagation is not used in Subsections~\ref{Subsec:Decay},
 \ref{Subsec:DeduceVpEst} and \ref{Subsec:propagation} on the decay of the Glimm functional and the deduced
 estimates in $\mathcal{W}_p$.
We merely use the strict separation of the propagation speeds for fronts of families $1$ and $2$, the fact that
 fronts of the same family merge when they meet, and the fact that at an interaction point the outgoing waves are 
 described in Subsection~\ref{Subsec:FTA}.
(The fact that the speed of propagation is approximately correct 
 is only used to show the consistency of the approximation). 
Hence we can apply the construction and the results of these subsections to $v^\nu$, with time horizon $t$. \par
As a consequence, we find that the $p$-variation of $v^\nu$ along the sides of the triangle is of size $\mathcal{O}(1)V_p(u^*_s(x))$.
This gives us
\begin{eqnarray*}
|u^\nu(t,x^*) - u^\nu(s,x^*)| &\leq& |u^\nu(s,x^* - \widehat{\lambda}(t-s)) - u^\nu(s,x^*)| +  \mathcal{O}(1)V_p(u^*_s(x)) \\
&\leq& \mathcal{O}(1)V_p(u^*_s(x)).
\end{eqnarray*}
Since $V_p(u^*_s(x))$ is precisely $V_p \Big(u^\nu(s,\cdot); [x^*-\hat{\lambda}(t-s), x^*+\hat{\lambda}(t-s)]\Big)$,
 this gives \eqref{Eq:EstVplocale}. This ends the proof of Lemma~\ref{Lem:EstVplocale}.
\end{proof}
We can now prove Proposition~\ref{Pro:TimeEstimates}.
\begin{proof}[Proof of Proposition~\ref{Pro:TimeEstimates}]
Pick $0\leq s<t$. 
Define $L_{s,t} := \widehat{\lambda}(t-s)$.
Then for $k \in \Z$, we raise \eqref{Eq:EstVplocale} to the power $p$ and integrate over
 $[k L_{s,t}, (k+1) L_{s,t}]$.
We easily infer, using that $V_p(f, [a,b])$ is increasing with the interval $[a,b]$, that
\begin{equation} \label{Eq:EstVplocale2}
\int_{k L_{s,t}}^{(k+1) L_{s,t} }|u^\nu(t,x) - u^\nu(s,x)|^p \, dx
\leq C \, L_{s,t} \, V_p^p \Big(u^\nu(s,\cdot); [(k-1)L_{s,t}, (k+2)L_{s,t}]\Big).
\end{equation}
Now due to the construction, on any compact time interval, $u^\nu(t,x) - u^\nu(s,x)$ is compactly supported in $x$,
 because $u^\nu(t,x)$ is equal to its leftmost (respectively rightmost) state for $x$ sufficiently negative
 (resp. large). 
For the same reason, $V_p^p \Big(u^\nu(s,\cdot); [(k-1)L_{s,t}, (k+2)L_{s,t}]\Big)=0$ for $|k|$ large enough.
It follows that we can safely sum \eqref{Eq:EstVplocale2} over $k$ and find
\begin{eqnarray}
\int_{\R} |u^\nu(t,x) - u^\nu(s,x)|^p \, dx 
\nonumber 
&\leq& C \, L_{s,t} \, \sum_{k \in \Z} V_p^p \Big(u^\nu(s,\cdot); [(k-1)L_{s,t}, (k+2)L_{s,t}]\Big) \\
\label{Eq:EstC1surP}
&\leq& 3 C L_{s,t} \, \, V_p^p \Big(u^\nu(s,\cdot); \R\Big),
\end{eqnarray}
where we summed separately over each $\Z / 3\Z$ class and used Proposition~\ref{Pro:ElemePropsVp}-\ref{Item:VpFusion}.
Since $L_{s,t} = \widehat{\lambda}(t-s)$, the conclusion follows.
\end{proof}
%
%
%
%
%
%
%
%
%
\subsubsection{Passage to the limit and conclusion}
\ \par
\noindent
{\bf Compactness.} At this level, we know that each $u^\nu$ is globally defined on $\R^+ \times \R$, 
and we have uniform bounds for $u^\nu$ in $L^\infty(\R_+;\mathcal{W}_p)$ and in $C^{1/p}([0,+\infty);L^p_{loc}(\R))$.
This guarantees us the compactness of the sequence in $L^1_{loc}(\R^+ \times \R)$. 
Here we just sketch the proof (see for instance \cite[Theorem 3.3]{JenssenRidderPhi20} in the more general case of the $\Phi$-variation).
The idea is that the same proof as for $p=1$ (see for instance \cite[Theorem 2.4]{BressanBook00}) can be done here, 
 because Helly's selection principle is also valid in $\mathcal{W}_p$ spaces (see Proposition~\ref{Pro:Classical}): 
 from a bounded sequence in $\mathcal{W}_p(\R)$, one can extract a sequence converging pointwise, and consequently by dominated
 convergence in $L^1_{loc}$.
Then a diagonal argument ensures that we can extract from $u^\nu$ as sequence that converges in $L^1_{loc}$ for all rational times.
Then thanks to the $C^{1/p}([0,+\infty);L^p_{loc}(\R))$ estimate, we can extend this limit defined on $\Q_+ \times \R$
 to $\R_+  \times \R$ by uniform continuity.
Then one can easily obtain the convergence of the subsequence in $L^1_{loc}(\R_+ \times \R)$ and almost everywhere. \par
Hence, up to a subsequence, we can suppose that
\begin{equation} \label{Eq:LimiteUnu}
u^\nu \longrightarrow u \ \text{ in } \ L^1_{loc}(\R_+ \times \R),
\end{equation}
for some $u \in L^\infty(\R_+;\mathcal{W}_p) \cap C^{1/p}([0,+\infty);L^p_{loc}(\R))$,
which satisfies \eqref{Eq:EstVpGlobale} and \eqref{Eq:TimeEstGlobale}. \ \par
%
%
%
%
%
%
Now it remains to prove that $u$ is indeed an entropy solution. 
This is done in several steps. \par
\ \par
\noindent
{\bf Estimate on the size of rarefaction fronts.} 
Let us show that rarefaction fronts in $u^\nu$ are all of order $\nu$.
Due to \eqref{Eq:Monotonicity} and \eqref{Eq:VitesseChocModifiee},
 two successive rarefaction fronts of the same family cannot meet. 
 Hence rarefaction fronts can only meet negative fronts of the same family (which make their strengths decrease)
 and fronts of the other family. \par
Call $\alpha$ a rarefaction at time $t$.
To simplify the presentation, we suppose that $\alpha$ is of family $1$, the other case being symmetrical.
We can trace back its history by following the unique incoming rarefaction of the 
 same family at each interaction, to its birthplace, say at $(t_0,x_0)$. Call $a$ the corresponding front line.
Due to the previous considerations, its strength is at most its initial strength,
 amplified by all crossings with $2$-fronts during its lifetime.
Recalling that $C_*$ is a bound for functions $\mathcal{C}^1$ and $\mathcal{C}^2$,
 we find
\begin{equation*}
|\sigma_\alpha(t)| \leq |\sigma_\alpha(t_0)| \prod_{\gamma \in \mathcal{B}_\alpha} (1 + C_*|\sigma_\gamma|^3 ),
\end{equation*}
where $\mathcal{B}_\alpha$ is the set of all $2$-fronts $\gamma$ crossing $\alpha$ between times $t_0$ and $t$.
Now, denoting $\Gamma^{\downarrow}_a$ a $1$-lower measure curve following the front line $a$, we find that
\begin{equation*}
\ln \left( \prod_{\gamma \in \mathcal{B}_\alpha} (1 + C_* |\sigma_\gamma|^3) \right) 
\leq \sum_{\gamma \in \mathcal{B}_\alpha} C_* |\sigma_\gamma|^3 
\lesssim s^3_3(\sigma_\gamma, \gamma \in \mathcal{B}_\alpha) \leq \mathcal{V}_{\Gamma^{\downarrow}_a}.
\end{equation*}
This is bounded according to Proposition~\ref{Pro:DecayGlimm}.
Since initially, rarefaction fronts are of size $\leq \varepsilon$, we find that
 for any rarefaction front $\alpha$ at any time,
\begin{equation} \label{Eq:SmallRar}
\sigma_\alpha(t) = \mathcal{O}(\nu).
\end{equation} \ \par
\ \par
\noindent
{\bf Estimate on the size of compression fronts.}
\label{Place:EstCW}
The case of compression fronts is somewhat similar, but the starting reasoning differs,
 since CF can meet other CF and shocks. 
However, when they do, the only possibility that the outgoing front of the same family is a CF,
 is that entering fronts were CF or RF, and that the sum of their (negative) strengths was more than $-2\nu$. 
Since the initial strength of a CF must be less than $\nu$ (since our convention in the front-tracking algorithm
 is that if the interaction is strong enough to produce an outgoing wave of size $\nu$, we use a classical wave),
 it follows that, beyond $2\nu$, the (absolute) strength of a CF can only increase by means of interactions
 with the other family.
Then we can reason as in the previous case, and obtain that \eqref{Eq:SmallRar} is valid for compression fronts as well. \par
\ \par
\noindent
{\bf Errors on the weak formulation.} 
Let us now prove that we obtain an entropy solution in the limit, by estimating the errors on 
 the weak formulation for the sequence $u^\nu$. \par
Let us consider a convex (in the large sense) entropy couple $(\eta,q)$ as in \eqref{Def:CoupleEntropie}.
Note that among these entropy couples, one can choose 
 $(\eta(u),q(u)) = (\pm \ell(\overline{u}) \cdot u, \pm \ell(\overline{u}) \cdot f(u) )$,
 so that the fact that a limit is a weak distributional solution will follow from the same analysis. \par
Let $\varphi \in \mathcal{D}( (0,+\infty) \times \R)$ such that $\varphi \geq 0$.
Given a front-tracking approximation $u^\nu$, an integration by parts shows that
\begin{multline*}
\int_{\R_+ \times \R} \big( \eta(u^\nu(t,x)) \varphi_t(t,x) + q(u^\nu(t,x)) \varphi_x(t,x) \big) 
= \int_{\R_+} \sum_{\alpha \in \mathcal{A}(t)} \varphi(t,x_\alpha(t)) 
\Big( [q(u^\nu)]_\alpha - \dot{x}_\alpha [\eta(u^\nu)]_\alpha \Big),
\end{multline*}
where we recall that $\mathcal{A}(t)$ designates the set of fronts in $u^\nu$ at time $t$,
 and where $[g(u)]_\alpha := g(t,u(x_\alpha^+)) - g(t,u(x_\alpha^-))$ is the jump of the quantity 
 $g$ across the jump $\alpha$ at time $t$. Define 
\begin{multline} \label{Eq:ErreurEntropie}
\mathcal{H}^\nu(t):= \sum_{\alpha \in \mathcal{A}(t)} \varphi(t,x_\alpha(t)) 
\Big( [q(u^\nu)]_\alpha - \dot{x}_\alpha [\eta(u^\nu)]_\alpha \Big) \\
\ \text{ and } \ 
\widetilde{\mathcal{H}}^\nu(t):= \sum_{\alpha \in \mathcal{A}(t)} \varphi(t,x_\alpha(t)) 
\Big( [q(u^\nu)]_\alpha - \overline{\lambda}_{i_\alpha}(u_\nu^-,u_\nu^+) [\eta(u^\nu)]_\alpha \Big),
\end{multline}
where $i_{\alpha}$ denotes the characteristic family of the front $\alpha$. \par
Recall that all fronts travel at modified shock speed such as defined in \eqref{Eq:VitesseChocModifiee}, 
 with a possible additional speed modification defined in \eqref{Eq:ModifSpeed}. 
We begin by estimating the error if the fronts traveled at actual shock speed, that is 
 we estimate $\widetilde{\mathcal{H}}^\nu(t)$.
We consider $\alpha \in \mathcal{A}(t)$ and discuss according to its nature.
\begin{itemize}
\item If $\alpha$ is a shock, then since it satisfies Liu's $E$-condition, it is entropic in the sense that
\begin{equation*}
[q(u^\nu)]_\alpha - \dot{x}_\alpha [\eta(u^\nu)]_\alpha \geq 0,
\end{equation*}
see e.g. \cite[(8.5.8)]{DafermosBook10}.
\item If $\alpha$ is a rarefaction or a compression front, then it is elementary to show
 (see e.g. \cite[Theorem8.5.1]{DafermosBook10}) that
\begin{equation} \label{Eq:EstErreurTildeHCWRW}
[q(u^\nu)]_\alpha - \dot{x}_\alpha [\eta(u^\nu)]_\alpha \geq -\mathcal{O}(1) |\sigma_\alpha|^3 
\geq -\mathcal{O}(1) \nu |\sigma_\alpha|^2 ,
\end{equation}
due to the above analysis on the sizes of rarefaction and compression fronts.
\end{itemize}
We deduce that
\begin{equation} \label{Eq:EstHTilde}
\widetilde{\mathcal{H}}^\nu(t) \gtrsim - \nu \, V^2_2(u^\nu(t,\cdot)) \gtrsim - \nu \, V^2_p(u^\nu(t,\cdot)).
\end{equation}
It remains to estimate $\mathcal{H}^\nu(t) - \widetilde{\mathcal{H}}^\nu(t)$ which embodies
 the errors due to the modification of the fronts speeds with respect to standard shock speed.
Let us discuss the two modifications of the speed separately.
\begin{itemize}
\item Concerning the possible additional speed modification \eqref{Eq:ModifSpeed}, this gives an additional error on
        $[q(u^\nu)]_\alpha - \dot{x}_\alpha [\eta(u^\nu)]_\alpha$ of size $\mathcal{O}(1) \nu |\sigma_\alpha|^2$.
      Hence this can be treated as for \eqref{Eq:EstErreurTildeHCWRW}, and gives the same type of error as for 
      $\widetilde{\mathcal{H}}^\nu$.
\item The error due to the speed modification \eqref{Eq:VitesseChocModifiee} is more subtle,
       because the error in absolute value for each front is of size $\mathcal{O}(1) \nu |\sigma_\alpha|$,
       while the strengths $\sigma_\alpha$ are not uniformly absolutely summable.
      Taking into account the previous steps and \eqref{Eq:VitesseChocModifiee},
       the quantity that remains to be estimated is
\begin{multline} \label{Eq:ErreurPrincipaleH}
\sum_{\alpha \in \mathcal{A}(t)} \varphi(t,x_\alpha(t)) 
\Big( \overline{\lambda}_{i_\alpha}(u^\nu(x_\alpha^-), u^\nu(x_\alpha^+)) 
    - \overline{\lambda}_{i_\alpha}^\nu(u^\nu(x_\alpha^-), u^\nu(x_\alpha^+)) \Big) 
[\eta(u^\nu)]_\alpha \\
=
\sum_{\alpha \in \mathcal{A}(t)} \nu \, \varphi(t,x_\alpha(t)) 
\Big(  w_i(u^\nu(x_\alpha^+)) -w_i(\overline{u}) \Big) 
[\eta(u^\nu)]_\alpha.
\end{multline}
Now we describe $\mathcal{A}(t)$ according to the position of the front, from left to right.
Due to Proposition~\ref{Pro:ElemePropsVp}-\ref{Item:LipVp}, we have 
\begin{equation*}
s_p \Big([\eta(u^\nu)]_\alpha, \alpha \in \mathcal{A}(t)\Big) 
= v_p\Big(\eta(u^\nu(x_\alpha^+)), \alpha \in \mathcal{A}^*(t)\Big) \lesssim \widetilde{V}_p(u^\nu),
\end{equation*}
where $\mathcal{A}^*(t)$ is obtained by adding a trivial front to the left of all others. \par
On another side, due to the regularity of $\varphi$ and $w_i$, it is elementary to see that
\begin{multline*}
v_1\Big(\varphi(t,x_\alpha(t)), \ \alpha \in \mathcal{A}(t) \Big) \leq TV(\varphi(t,\cdot)) \lesssim 1 \\
\text{ and } \ 
v_p\Big(w_i \big( u^\nu(x_\alpha^+) \big) -w_i(\overline{u}), \ \alpha \in \mathcal{A}(t) \Big) 
= v_p\Big(  w_i \big( u^\nu(x_\alpha^+) \big), \ \alpha \in \mathcal{A}(t) \Big) \leq \widetilde{V}_p(u^\nu(t,\cdot)).
\end{multline*}
Hence using twice Corollary~\ref{Cor:MultiplSp}, we find that
\begin{equation*}
s_p 
\Big( \nu \, \varphi(t,x_\alpha(t)) \, 
\big(  w_i \big( u^\nu(x_\alpha^+) \big) -w_i(\overline{u}) \big) 
\, [\eta(u^\nu)]_\alpha \Big)
\lesssim \nu \, \widetilde{V}_p(u^\nu(t,\cdot))^2.
\end{equation*}
In particular, the right-hand side of \eqref{Eq:ErreurPrincipaleH} can be estimated by 
$\nu \, \widetilde{V}_p(u^\nu(t,\cdot))^2$.
\end{itemize}
Gathering the estimates above, we find that for all $t \in \R^+$,
\begin{equation*}
\mathcal{H}^\nu(t) \geq - \mathcal{O}(1) \nu \widetilde{V}_p(u^\nu(t,\cdot))^2.
\end{equation*}
Hence using Lebesgue's dominated convergence theorem and passing to the limit as $\nu \rightarrow 0^+$,
 we find that the limit $u$ in \eqref{Eq:LimiteUnu} satisfies
\begin{equation*}
\int_{\R_+ \times \R} \big( \eta(u(t,x)) \varphi_t(t,x) + q(u(t,x)) \varphi_x(t,x) \big) 
\geq 0.
\end{equation*}
Finally, that $u$ satisfies the initial condition is a consequence of \eqref{Eq:ApproxInitiale}
 and \eqref{Eq:EstC1surP}. \par
This ends the proof of Theorem~\ref{Thm:Main}.
%
%
%
%
%
%
%
%
%
%
%
%
\section{Justification of Remarks~\ref{rem:CasNgeq3} and \ref{rem:OndesEngendreesBV}}
\label{Sec:Remarks}
\subsection{About Remark~\ref{rem:CasNgeq3}: Young's counterexample for $n \times n$ systems with $n \geq 3$}
\label{Subsec:Counterexample}
Remark~\ref{rem:CasNgeq3} is a direct consequence of Young's example \cite[pp. 549-555]{RYoung99}.
We reproduce the example, and then explain how it can be used to show instability in $V_p$ seminorm.
Let
\begin{equation*}
f_2(u,v,w) = \begin{pmatrix} w + 2uv \\ 0 \\ u(1-4v^2) -2vw \end{pmatrix}.
\end{equation*}
With this flux function $f_2$, System~\eqref{Eq:SCL} is strictly hyperbolic with $\lambda_1=-1$, $\lambda_2=0$
 and $\lambda_3=1$, the corresponding characteristic fields being linearly degenerate. \par
For $\Delta x >0$, one describes a $4 \Delta x$-periodic
 initial data $U_0$, obtained as follows.
On the interval $[0,4 \Delta x)$, $U_0$ is constant on each interval $[k \Delta x, (k+1)\Delta x)$,
 the corresponding constants from left to right being called $1,2,5,6$. 
The discontinuities $(1,2)$ and $(5,6)$ will be chosen as $2$-waves,
 while $(2,5)$ and $(6,1)$ determine a full Riemann problem, with intermediate states $3$, $4$ and $7$, $0$,
 respectively. See Figure~\ref{Fig:Young}. \par
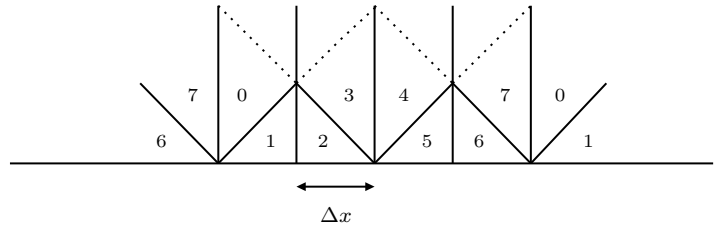
\begin{figure}[!ht]
\centering
\tikzset{every picture/.style={line width=0.75pt}} 

\begin{tikzpicture}[x=0.75pt,y=0.75pt,yscale=-1.3,xscale=1.3]

\draw    (100,201) -- (370,201) ;
\draw    (150,170) -- (180,201) ;
\draw  [dash pattern={on 0.84pt off 2.51pt}]  (180,140) -- (210,170) ;
\draw  [dash pattern={on 0.84pt off 2.51pt}]  (210,170) -- (227,153) -- (241,140) ;
\draw    (180,201) -- (210,170) ;
\draw    (180,140) -- (180,201) ;
\draw    (210,170) -- (240,201) ;
\draw    (240,201) -- (270,170) ;
\draw    (240,140) -- (240,201) ;
\draw    (210,140) -- (210,201) ;
\draw    (270,170) -- (300,201) ;
\draw    (300,201) -- (329,170) ;
\draw    (300,140) -- (300,201) ;
\draw  [dash pattern={on 0.84pt off 2.51pt}]  (240,140) -- (270,170) ;
\draw  [dash pattern={on 0.84pt off 2.51pt}]  (270,170) -- (287,153) -- (300,140) ;

\draw    (213,210) -- (237,210) ;
\draw [shift={(240,210)}, rotate = 180] [fill={rgb, 255:red, 0; green, 0; blue, 0 }  ][line width=0.08]  [draw opacity=0] (3.57,-1.72) -- (0,0) -- (3.57,1.72) -- cycle    ;
\draw [shift={(210,210)}, rotate = 0] [fill={rgb, 255:red, 0; green, 0; blue, 0 }  ][line width=0.08]  [draw opacity=0] (3.57,-1.72) -- (0,0) -- (3.57,1.72) -- cycle    ;
\draw    (270,140) -- (270,201) ;

\draw (218,217) node [anchor=north west][inner sep=0.75pt]  [font=\footnotesize]  {$\Delta x$};
\draw (155,189) node [anchor=north west][inner sep=0.75pt]  [font=\scriptsize]  {$6$};
\draw (197,189) node [anchor=north west][inner sep=0.75pt]  [font=\scriptsize]  {$1$};
\draw (167,171) node [anchor=north west][inner sep=0.75pt]  [font=\scriptsize]  {$7$};
\draw (186,171) node [anchor=north west][inner sep=0.75pt]  [font=\scriptsize]  {$0$};
\draw (217,189) node [anchor=north west][inner sep=0.75pt]  [font=\scriptsize]  {$2$};
\draw (227,171) node [anchor=north west][inner sep=0.75pt]  [font=\scriptsize]  {$3$};
\draw (248,171) node [anchor=north west][inner sep=0.75pt]  [font=\scriptsize]  {$4$};
\draw (257,189) node [anchor=north west][inner sep=0.75pt]  [font=\scriptsize]  {$5$};
\draw (277,189) node [anchor=north west][inner sep=0.75pt]  [font=\scriptsize]  {$6$};
\draw (287,171) node [anchor=north west][inner sep=0.75pt]  [font=\scriptsize]  {$7$};
\draw (308,171) node [anchor=north west][inner sep=0.75pt]  [font=\scriptsize]  {$0$};
\draw (319,189) node [anchor=north west][inner sep=0.75pt]  [font=\scriptsize]  {$1$};

\end{tikzpicture}
\caption{Young's example}
\label{Fig:Young}
\end{figure}
Then for a given state $1$ and $\alpha>0$, $\beta>0$, one chooses the strengths
of the discontinuities $(1,2)$, $(2,3)$, $(3,4)$, $(4,5)$, and $(5,6)$ 
as $\beta$, $\alpha$, $-\beta$, $-\alpha$, and $\beta$, respectively. 
Calculations in \cite[p. 552]{RYoung99} show then that the strengths of discontinuities $(6,7)$, $(7,0)$, and $(0,1)$
 are then $-\alpha$, $-\beta$, and $\alpha$, respectively. \par
One associates the corresponding entropy solution $U$ (obtained by progressively solving Riemann problems),
 and \cite[(5.10)]{RYoung99} proves that for any $k \in \N$, $U(2k \Delta x, \cdot)$ has the exact same structure,
 with $\alpha$ and $\beta$ replaced respectively by $\alpha_k$ and $\beta_k$, with
\begin{equation*}
\beta_k = \beta
\ \text{ and } \ 
\alpha_k = \left(\frac{1+\beta}{1-\beta}\right)^k \alpha.
\end{equation*}
Now we use this example as follows.
For $n \in \N \setminus \{0\}$ we let 
\begin{equation*}
\Delta x = \frac{1}{n} \ \text{ and } \  \alpha=\beta= \frac{\varepsilon}{n^{1/p}}
\end{equation*}
in the previous construction, which gives $U_0$.
Now we define
\begin{equation*}
\mathscr{U}_0(\cdot) := U_0(\cdot) {\bf 1}_{[-2,3]}(\cdot).
\end{equation*}
Call $\mathscr{U}$ the corresponding entropy solution (which again is obtained for each time by solving a finite
 number of Riemann problems).
Due to the finite speed of propagation (speeds being $-1$, $0$, $1$), one has
\begin{equation*}
\mathscr{U}_{|[0,1]} = U_{|[0,1]} \text{ for times } t \text{ in } [0,2].
\end{equation*}
Now it is elementary to check (measuring the $p$-variation for jumps of each 
 characteristic family separately for instance) that
\begin{equation*}
V_p(\mathscr{U}_0)= \mathcal{O}(1) \varepsilon
\ \text{ while } \ 
V_p(\mathscr{U}(2,\cdot)) \geq V_p(\mathscr{U}(2,\cdot)_{|[0,1]}) 
\geq \varepsilon \left(\frac{1 + \frac{\varepsilon}{n^{1/p}}}{1 - \frac{\varepsilon}{n^{1/p}}}\right)^n
\geq \varepsilon \exp \left(2 \varepsilon n^{1-\frac{1}{p}}\right),
\end{equation*}
where the inequality $\frac{1+x}{1-x} \geq \exp(2x)$ for $x \leq 1$ was used. \par
Hence we can find initial states with arbitrarily small $p$-variation, which generate a solution with arbitrarily
large $p$-variation at time $2$ (a time that is arbitrary by scale-invariance of the equation).
%
%
%
%
%
%
%
%
\subsection{About Remark~\ref{rem:OndesEngendreesBV}: total production of new waves}
\label{Subsec:NewWaves}
Here we consider ``new'' waves, by which we mean waves created in the other family at the interaction point
 of two waves of the same family.
We can see that the total production of new waves in absolute value (at the level of wave-front tracking approximations) 
 is bounded. 
By contrast the increase of the strengths of waves across interaction points of opposite families can 
 be of unbounded total variation, as explained in the introduction. \par
The fact that the total production of these new waves is bounded (taking the $\ell^1$ norm of the size of the jumps)
 can be seen as an easy consequence of the analysis of Subsection~\ref{Subsec:Decay}. 
Indeed when considering interactions of waves of the same family in Subsection~\ref{Subsec:Decay},
 we proved that $\mathcal{Q}$ decayed of an amount of 
 $\mathcal{O}(1)|\sigma_\alpha||\sigma_\beta|(|\sigma_\alpha| + |\sigma_\beta|)$,
 which is of the same order as the size of newly produced wave at this interaction point, 
 independently of the time horizon $T$. \par
Hence we could prove the decay of a functional of the type
\begin{equation*}
\mathcal{Q}(t) + \sum_{\substack{\alpha \text{ new wave} \\ \text{created before } t}} |\sigma_\alpha^0|,
\end{equation*}
where $\sigma_\alpha^0$ denotes the strength of the wave $\alpha$ at the moment it was created.
This proves that 
\begin{equation*}
\sum_{\alpha \text{ new wave}} |\sigma_\alpha^0| \lesssim V_p(u_0)^{2p}.
\end{equation*}
Going a bit further, and extending these new waves across interactions of opposite families in a natural way,
 and considering them as separate parts of the outgoing wave of the same family at interactions of the same family 
 (in the spirit of Liu's wave-tracing technique \cite{LiuWT77,LiuYang02}), we could even replace 
 the $|\sigma_\alpha^0|$ above with some $|\sigma_\alpha(t)|$, we omit the details. \par
%
%
%
%
%
%
%
\section{Appendix}
\label{Sec:Appendix}
%
%
%
%
%
%
%
%
\subsection{Proofs of elementary properties of $s_p$: Propositions~\ref{Pro:ElemePropsVp} and \ref{Pro:PropsVpEndpoints}}
\label{Subsec:ProofsVp}
\begin{proof}[Proof of Proposition~\ref{Pro:ElemePropsVp}] \ \par
\noindent
{\bf \ref{Item:SpSpPrime}-\ref{Item:TriangIneg}}. These properties are direct consequences of their counterpart
in the sequence space $\ell^p$. \par
\ \par
\noindent
{\bf \ref{Item:DLVp}}. This is a consequence of the above triangle inequality $s_p(a+b) \leq s_p(a) + s_p(b)$
and the elementary inequality
\begin{equation} \label{Eq:SommePuissanceP}
\text{for } x,y \geq 0, \ \  |x+y|^p \leq |x|^p + \mathcal{O}(1)|x|^{p-1}|y| + 2 |y|^p,
\end{equation}
which is trivial for $x \leq (\sqrt[3]{2} -1) y$, and in the opposite case a consequence of
 $\frac{d}{dt}|x+t|^p \lesssim |x|^{p-1}$ for $t \in [0,y]$. \par
\ \par
\noindent
{\bf \ref{Item:Fusion}}. The first inequality expresses that the supremum over all partitions is larger
than the one over partitions that force $a_j$ and $a_{j+1}$ to belong to the same component.
Concerning the same-sign case, it is easy to see that in an optimal partition (calling $a_{n+1}=t$):
\[
s_p^p(a_1,\ldots,a_{n},t) =\sum_{j=0}^{k-1} \left|\sum_{\ell=i_{j}+1}^{i_{j+1}} a_{\ell} \right|^{p},
\]
two successive sums $\sum_{\ell=i_{j}+1}^{i_{j+1}} a_{\ell}$ and $\sum_{\ell=i_{j+1}+1}^{i_{j+2}} a_{\ell}$
 have opposite sign, and that the extreme values of a component $a_{i_{j}+1}$ and $a_{i_{j+1}}$ have the same sign
 as the sum $\sum_{\ell=i_{j}+1}^{i_{j+1}} a_{\ell}$. 
It follows that two successive values $a_i$ and $a_{i+1}$ having the same strict sign
 belong to the same component in an optimal partition, which justifies the claim. \par
\ \par
\noindent
{\bf \ref{Item:VpFusion}}. Property~\ref{Item:VpFusion} comes from the fact that taking the supremum in \eqref{Def:sp}
for $a \diamond b$ is larger without constraint than with the constraint of a partition obtained 
by joining a partition of $a$ and a partition of $b$. \par
\ \par
\noindent
{\bf \ref{Item:Interleave}}. 
Consider an optimal partition for $v_p(a)$. 
For $i \in \{1,\ldots,k\}$, let $N_i$ the position of $a_i$ in the interleaved sequence $c$.
We start from an optimal partition $0=i_{0}<\dots<i_{j}=k$ for $v_p(a)$:
\begin{equation*}
v_{p}(a)= \left(\sum_{\ell=0}^{j-1} \left|a_{i_{\ell+1}} - a_{i_{\ell}+1} \right|^{p} \right)^{1/p}.
\end{equation*}
We let $n_\ell := N_{i_\ell}$ for $\ell=0,\ldots,j$. Then
\begin{equation*}
v_{p}(a)= \left(\sum_{\ell=0}^{j-1} \left|c_{N_{\ell+1}} - c_{N_{\ell}+1} \right|^{p} \right)^{1/p}.
\end{equation*}
Adding $0$ at the beginning and $k+k'$ at the end if necessary, $N_0<\ldots<N_j$ form a partition for $c$,
so $v_p(c) \leq v_p(a)$ follows. \par
Concerning $|s_p(a) - s_p(b)| \leq s_p(c) \leq s_p(a) + s_p(b)$, we define $(\tilde{a}_j)_{j=1,\ldots,k+k'}$
 by $\tilde{a}_j=a_i$ if $j=N_i$ and $\tilde{a}_j=0$ otherwise. 
Similarly, we define $\tilde{b}$ as the sequence retaining only the
 term $b_i$ in $c$. Obviously, $s_p(\tilde{a})=s_p(a)$ and $s_p(\tilde{b})=s_p(b)$.
Then $c=\tilde{a}+\tilde{b}$, and the conclusion follows from point~\ref{Item:TriangIneg}. \par
\ \par
\noindent
{\bf \ref{Item:Repetition}}. Property~\ref{Item:Repetition} comes from the fact that for any partition for $b$,
 the sum in \eqref{Def:vp} corresponding to this partition for $b$, is represented in the sums defining $v_p(a)$. 
Of course one has equality if each $a_i$ is represented at least once in $b$. \par
\ \par
\noindent
{\bf \ref{Item:Ajoutde0}}. Using Property~\ref{Item:Repetition}, we have
$v_p(0,a_1,\ldots,a_k) \leq v_p(a_1,a_1,\ldots,a_k) + v_p(-a_1,0,\ldots,0) = v_p(a_1,\ldots,a_k) +|a_1|$. \par
\ \par
\noindent
{\bf \ref{Item:LipVp}}.
The proof of \ref{Item:LipVp} is immediate. \par
\ \par
\noindent
{\bf \ref{Item:VpProd}}. 
We first treat the case $n=2$.
Consider an optimal partition $0=i_{0}<\dots<i_{j}=k$ for $v_p(a_i b_i)$:
\begin{align*}
v_{p}(a_i b_i)
&= \left(\sum_{m=0}^{j-1} \left| a_{i_{m+1}} b_{i_{m+1}}  - a_{i_{m}+1} b_{i_{m}+1} \right|^{p} \right)^{1/p} \\
& = \| a_{i_{m+1}} b_{i_{m+1}}  - a_{i_{m}+1} b_{i_{m}+1} \|_{\ell^p(0,\ldots,j-1)} \\
& \leq \| a_{i_{m+1}} (b_{i_{m+1}}  -  b_{i_{m}+1}) \|_{\ell^p(0,\ldots,j-1)} 
+ \|  b_{i_{m+1}}  (a_{i_{m+1}} - a_{i_{m}+1})\|_{\ell^p(0,\ldots,j-1)} \\
& \leq  \| a\|_\infty \, \| b_{i_{m+1}}  -  b_{i_{m}+1}  \|_{\ell^p(0,\ldots,j-1)} 
+  \| b\|_\infty \, \| a_{i_{m+1}} - a_{i_{m}+1} \|_{\ell^p(0,\ldots,j-1)} \\
&\leq  \|a\|_\infty  \,v_p(b) + \|b\|_\infty \, v_p(a).
\end{align*}
Then the general case is an immediate induction.
\end{proof}
%
%
%
%
%
%
%
\begin{proof}[Proof of Proposition~\ref{Pro:PropsVpEndpoints}] \ \par
\noindent
{\bf \ref{Item:EndPoint1}.}
The map $t \mapsto s_p^p(x_1,\ldots,x_{n},t)$ is clearly convex as a supremum of convex functions (see \eqref{Def:sp}). 
Now for each $x_1,\ldots,x_{n}$ and $t \geq 0$, we write $x_{n+1}=t$ 
and consider an optimal partition $0=i_{0}<\dots<i_{k}= n+1$ so that
\begin{equation} \label{Eq:UnePartitionOptimale}
s_p^p(x_1,\ldots,x_{n},t) 
    = \sum_{j=0}^{k-1} \left|\sum_{\ell=i_{j}+1}^{i_{j+1}} x_{\ell} \right|^{p} .
\end{equation}
Then either $i_{k-1}+1=n+1$ and 
\[
s_p^p(x_1,\ldots,x_{n},t) 
    =\sum_{j=0}^{k-2} \left|\sum_{\ell=i_{j}+1}^{i_{j+1}} x_{\ell} \right|^{p} 
    + |t|^p,
\]
or $i_{k-1}<n$ and necessarily $t$ has the same sign as $\sum_{\ell=i_{k-1}+1}^{n} x_{\ell}$,
or the partition would not be optimal. In both cases, we obtain an increasing expression of $t$. \par
It follows that for $t$ in $\R_+$, $s_p^p(x_1,\ldots,x_{n},t)$ can be written as a supremum of expressions
 as in the right-hand side of \eqref{Eq:UnePartitionOptimale} over only partitions satisfying $i_{k-1}+1=n+1$
 or for which the last sum between $i_{k-1}+1$ and $n$ is positive. 
Hence $s_p^p(x_1,\ldots,x_{n},t)$ can be expressed as a (finite) supremum
of increasing functions, so $t \mapsto s_p^p(x_1,\ldots,x_{n},t)$ is increasing.
The reasoning on $\R^-$ is similar. \par
\ \par
\noindent
{\bf \ref{Item:EndPoint2}.} Consider an optimal partition $0=i_{1}<\dots<i_{k}=n$ so that 
\begin{equation} \label{Eq:OP1}
s_p^{p}(x_1,\ldots,x_n) 
    = \sum_{j=0}^{k-1} \left| \sum_{\ell=i_{j}+1}^{i_{j+1}} x_{\ell} \right|^{p} .
\end{equation}
Then clearly 
\[
s_p^{p}(x_1,\ldots,x_n,x_{n+1}) 
    \geq \sum_{j=0}^{k-1} \left|\sum_{\ell=i_{j}+1}^{i_{j+1}} x_{\ell} \right|^{p} + |x_{n+1}|^p,
\]
which gives the second claim. \par
When $x_n$ and $x_{n+1}$ have the same sign, consider again an optimal partition $0=i_{1}<\dots<i_{k}=n$
so that \eqref{Eq:OP1} applies.
Then in the last component of the partition $(x_{i_{k-1}+1},\ldots,x_n)$, $x_n$ has the same sign
as the sum of the other terms $S:=\sum_{\ell=i_{k-1}+1}^{i_{k}-1} x_\ell$,
because otherwise the partition would not be optimal.
Now we have
\begin{equation*}
s_p^{p}(x_1,\ldots,x_n,x_{n+1}) 
    \geq \sum_{j=0}^{k-2} \left|\sum_{\ell=i_{j}+1}^{i_{j+1}} x_{\ell} \right|^{p} 
    + |S+x_n+x_{n+1}|^p,
\end{equation*}
so that
\begin{eqnarray*}
s_p^{p}(x_1,\ldots,x_n,x_{n+1}) - s_p^{p}(x_1,\ldots,x_n) 
    &\geq& |S+x_n+x_{n+1}|^p - |S+x_n|^p \\
    &\geq& |x_n+x_{n+1}|^p - |x_n|^p \\
    &\geq&     p|x_n|^{p-1} |x_{n+1}| ,
\end{eqnarray*}
where we used the convexity of $t \mapsto |t|^p$ and the fact that $S$, $x_n$ and $x_{n+1}$ have the same sign. \par
\ \par
\noindent
{\bf \ref{Item:EndPoint3}.} This case is clearly symmetrical.
\end{proof}
%
%
%
%
%
%
%
%
\subsection{Proof of the multiplicative property of $s_p$: Corollary~\ref{Cor:MultiplSp}}
\label{Subsec:ProofsDLY}
Consider $a$ and $b$ as in the statement. We consider an optimal partition $0=k_1 < \cdots < k_\ell = n$ for $ab$, so that
\begin{equation*}
s_p(ab) = \sum_{i=1}^{\ell-1} \left| \sum_{j=k_i+1}^{k_{i+1}} a_j b_j\right|^p.
\end{equation*}
Now calling $\delta b := (b_2-b_1,\ldots,b_n-b_{n-1})$, we have for each index $i$
\begin{equation*}
\sum_{j=k_i+1}^{k_{i+1}} a_j b_j
    = b_{k_i+1} \sum_{j=k_i+1}^{k_{i+1}} a_j 
    + \sum_{k_i+1 \leq m < j \leq k_{i+1}} a_j \, \delta b_m.
\end{equation*}
Now we use the convexity inequality $|x+y|^p \leq 2^{p-1} (|x|^p + |y|^p)$ and the discrete Love-Young inequality
(using $p<2$) to deduce, denoting $s_p(x;m_1;m_2) :=s_p(x_{m_1},\ldots,x_{m_2})$,
\begin{equation*}
\left| \sum_{j=k_i+1}^{k_{i+1}} a_j b_j\right|^p \leq
    2^{p-1} s_p(a;k_i+1,k_{i+1})^p \Big( \| b\|_\infty^p  + [1+\zeta(2/p)] \, s_p(\delta b;k_i+1,k_{i+1})^p \Big) .
\end{equation*}
Using $s_p(a;k_i+1,k_{i+1})^p \leq s_p(a)^p$, Proposition~\ref{Pro:ElemePropsVp}-\ref{Item:VpFusion} repeatedly,
and $v_p(b)=s_p(\delta b)$, we reach the conclusion.
%
%
%
%
%
%
%
%
%
\subsection{Proof of the description of wave curves for monotone fields: Proposition~\ref{Pro:LaxCurves}}
\label{Subsec:WCMontoneFields}
\begin{proof}
 Without loss of generality, we can assume $i=n$ (since we did not order the eigenvalues here);
 this will simplify a bit the notations in the sequel. In consistency with \eqref{Eq:Parametrisation},
 we normalize here $r_n$ by $\ell_n^0 \cdot r_n=1$, where for all $j$ we abridge $\ell_j^0:=\ell_j(u_0)$. \par
 
 For $s \geq 0$, due to assumption \eqref{Eq:Monotonicity} it is clear that 
  $L:s \mapsto \lambda_n(\mathcal{R}_n(s,u_0))$ is a continuous non-decreasing function.
 It follows that $L^{-1}$ is defined as a possibly discontinuous increasing function.
 Hence it is elementary to check that for $u_1 = \mathcal{R}_n(s;u_0)$, the function
 \begin{equation*}
u(t,x) = \left\{ \begin{array}{l}
    u_0 \ \text{ for } \ \frac{x}{t} \leq \lambda_n(u_0), \\
    \mathcal{R}_n(t; u_0) \ \text{ for } \ \frac{x}{t} = \lambda_n( \mathcal{R}_n(t;u_0)), \ t \in [0,s], \\
    u_1 \ \text{ for } \ \frac{x}{t} \geq \lambda_n(u_1),
\end{array}\right.
\end{equation*}
is well-defined apart from (at most countable) discontinuities of $L^{-1}$ and gives an admissible solution
to the Riemann problem $(u_0,u_1)$ and is composed of rarefactions and possibly contact discontinuities
at speeds where $L^{-1}$ is discontinuous. \par
The main point is to prove that for $s<0$, shocks are admissible, despite the possible degeneracy of 
$\nabla \lambda_n \cdot r_n$. 
This relies on the fact that the vanishing of $\nabla \lambda_n \cdot r_n$ guides the tangency of $\mathcal{R}_n$
and $\mathcal{S}_n$ at base point (see the remark below). \par
%
%
The main point is to observe that the Hugoniot locus can be obtained by means of Newton's iterations applied
 to the map constructed as follows. 
For $s_0>0$, given a map $(\alpha_j)_{j =1 \ldots n-1} : [-s_0,s_0] \rightarrow \R^{n-1}$
 such that $\alpha(s) / s$ is continuous (so that in particular $\alpha_j(0)=0$), we first associate the curve
$\gamma[(\alpha_j)_{j =1 \ldots n-1}]$ given by the map
\begin{equation*}
    \gamma[(\alpha_j)](s) = u_0 + s r_n(u_0) + \sum_{j =1}^{n-1} \alpha_j(s) r_j(u_0).
\end{equation*}
Next, considering a function $V : [-s_0,s_0] \rightarrow \R$ with values in a neighborhood of $\lambda_n(u_0)$,
we associate $F[(\alpha_j) ; V]: [-s_0,s_0] \rightarrow \R^{n}$ by
\begin{equation} \label{Eq:ApplNewton}
F[(\alpha_j) ; V](s):=
\begin{pmatrix}
    \ell_1^0 \cdot \big[ f(\gamma(s)) - f(u_0) - V (\gamma(s) -u_0)\big] \\
    \vdots \\
    \ell_{n-1}^0 \cdot \big[ f(\gamma(s)) - f(u_0) - V (\gamma(s) -u_0)\big] \smallskip \\
    \displaystyle V -    \overline{\lambda}_n(u_0,\gamma(s))
\end{pmatrix}
\text{ where } \gamma:=\gamma[(\alpha_j) ],
\end{equation}
where we denoted
\begin{equation} \label{Eq:DefLambdaBar}
\overline{\lambda}_n(u_0,\gamma(s)) := \frac{\ell_n^0 \cdot (f(\gamma(s)) - f(u_0))}{s} .
\end{equation}

Then in a neighborhood of $\overline{u}$ ensuring that $(\ell_1^0,\ldots,\ell_{n-1}^0,\ell_n(\overline{u}))$
is a basis of $\R^n$, finding the Hugoniot locus amounts to finding a zero of $F$.
This can be obtained by means of the Newton-Kantorovich theorem.
We let $E$ the space of continuous functions $(\alpha_1,\ldots,\alpha_{n-1}, V) :[-s_0,0] \rightarrow \R^{n}$ 
for which the $n-1$ first components satisfy that $ \alpha_j(s) / s$ is continuous; 
these components are endowed with the norm
$\| \alpha_j(s) / s \|_{\infty}$, the last one with $\| V \|_\infty$; this gives a norm on $E$ 
by taking the max of these.  \par
We compute the differential of $F$:
 \begin{multline*}
 DF[(\alpha_j) ; V]\big((\beta_{j'}) ; v\big)
 =\\
 \displaystyle
 \begin{pmatrix}
    \ell_{1}^0 \cdot (df_\gamma(r_{1}^0) - V r_1^0)   &   \cdots &  \ell_{1}^0 \cdot (df_\gamma(r_{n-1}^0) - V r_{n-1}^0)
        &  \ell_{1}^0 \cdot (\gamma - u_0) \\
    \vdots & \ & \vdots & \vdots \\
    \ell_{n-1}^0 \cdot (df_\gamma(r_{1}^0) - V r_1^0)   &  \cdots 
        &  \ell_{n-1}^0 \cdot (df_\gamma(r_{n-1}^0) - V r_{n-1}^0)    &  \ell_{n-1}^0 \cdot (\gamma - u_0) \\
    \II \displaystyle - \frac{\ell_n^0 \cdot df_\gamma( r_{1}^0 )}{s}   &  \cdots 
        & \displaystyle - \frac{\ell_n^0 \cdot df_\gamma( r_{n-1}^0 )}{s}   & 1
 \end{pmatrix}
 \begin{pmatrix}
     \beta_{1} \\ \vdots \\ \beta_{n-1} \\ v
 \end{pmatrix}.
 \end{multline*}
 We see that the coefficients are bounded as long as $[(\alpha_j) ; V] \in B_E([0;\lambda_n(u_0)], r)$, $r>0$;
 for the coefficients in the last line we use $\ell_n^0 \cdot df_{u_0}( r_{j}^0 )=0$ for $j \leq n-1$. 
 At base point $[0;\lambda_n(u_0)]$ this matrix 
 \begin{equation*}
 \begin{pmatrix}
    \lambda_1^0 - \lambda_n^0  & 0  &    \cdots &  0 & 0\\
    0 & \lambda_2^0 - \lambda_n^0 & 0 & \vdots & \vdots \\
    0 & \ & \ddots & \vdots & \vdots \\
    0 & \cdots & 0 &  \lambda_{n-1}^0 - \lambda_n^0   &  0 \\
    \II \displaystyle - \frac{\ell_n^0 \cdot df_{u_0+sr_n^0} ( r_{1}^0 )}{s}   &  \cdots &  \cdots 
        & \displaystyle  - \frac{\ell_n^0 \cdot df_{u_0+sr_n^0}( r_{n-1}^0 )}{s}   & 1
 \end{pmatrix}
 \end{equation*}
 is clearly invertible.
 Hence, by choosing $r>0$ suitably small, we see that $DF[(\alpha_j) ; V]$ is invertible
 for $[(\alpha_j) ; V] \in B_E([0;\lambda_n(u_0)], r)$;
 it is easy to check that the inverse is continuous for the $E$-norm. 
 In the same way, the second derivative of $F$ is also bounded in such a ball (recalling that $\gamma$ is affine in
  $\alpha_i$). 
 Hence, in order to be able to apply the Newton-Kantorovich theorem, 
 it is sufficient to check that given $R>0$, there exists $r>0$ such that 
 $\| F[(\alpha_j) ; V] \|_E \leq R$ on $B_E([0;\lambda_n(u_0)], r)$. This is elementary. \par
 We will use the following consequence of the Newton-Kantorovich theorem: given a curve $\gamma \in E$ and a velocity function $V \in C^0([-s_0,s_0])$, one has for some $K>0$
\begin{equation*}
\| (\gamma; V) - (\mathcal{S}_n; \overline{\lambda}_n(u_0,\mathcal{S}_n)) \|_{E \times L^\infty} 
\leq K \| F(\gamma,V) \|_{E \times L^\infty}.
\end{equation*}
 We apply this principle with $\gamma=\mathcal{R}_n$ and $V = \overline{\lambda}_n(u_0,\mathcal{R}_n(\cdot))$
 (so that the velocity part of $F(\gamma;V)$ vanishes) and deduce
\begin{equation} \label{Eq:EstRnSn}
\| \mathcal{R}_n - \mathcal{S}_n \|_E +
\| \overline{\lambda}_n(u_0,\mathcal{R}_n(\cdot)) - \overline{\lambda}_n(u_0,\mathcal{S}_n(\cdot)) \|_{\infty}
\leq K \left\| F[\mathcal{R}_n;\overline{\lambda}_n(u_0,\mathcal{R}_n(\cdot))]_{i = 1\ldots n-1} \right\|_E .
\end{equation}
 Now for $i = 1\ldots n-1$ and $s \in [-s_0,s_0]$ we have
 \begin{equation*}
 \ell_i^0 \cdot \big[ f(\mathcal{R}_n(s)) - f(u_0) 
    - \overline{\lambda}_n(u_0,\mathcal{R}_n(s)) (\mathcal{R}_n(s) -u_0)\big] 
=
\ell_i^0 \cdot \int_0^s [\lambda_n(\mathcal{R}_n(\tau)) - \overline{\lambda}_n(u_0,\mathcal{R}_n(s))] 
        r_n(\mathcal{R}_n(\tau))\, d\tau.
 \end{equation*}
 Using moreover that $\ell_i^0 \cdot r_n(\mathcal{R}_n(\tau)) = \mathcal{O}(\tau)$, so we get
 \begin{gather} 
 \label{Eq:DistRS1}
 \| \mathcal{R}_n - \mathcal{S}_n \|_{L^\infty([ s,0]) } 
    \leq C \int_0^s \tau 
        \big| \lambda_n(\mathcal{R}_n(\tau)) - \overline{\lambda}_n(u_0,\mathcal{R}_n(s)) \big| \, d\tau. \\
 \label{Eq:DistRS2}
 \| \overline{\lambda}_n(u_0,\mathcal{R}_n(\cdot)) - \overline{\lambda}_n(u_0,\mathcal{S}_n(\cdot)) \|_{\infty}
    \leq \frac{C}{|s|} \int_0^s \tau 
        \big| \lambda_n(\mathcal{R}_n(\tau)) - \overline{\lambda}_n(u_0,\mathcal{R}_n(s)) \big| \, d\tau.
 \end{gather}
 We will use the following lemma.
 \begin{lemma} \label{Lem:InegFnConvexes}
 Let $h : [-1,1] \rightarrow \R$ a $W^{2,1}$ convex function. Then for all $s \in  [-1,1]$,
 \begin{equation*}
 \left| \, \int_0^s t \left|h'(t) - \frac{h(s) - h(0)}{s} \right| \, dt \, \right|
    \leq \frac{s^2}{2} \left| h'(s)- \frac{h(s) - h(0)}{s} \right|.
 \end{equation*}
 \end{lemma}
 \begin{proof}[Proof of Lemma~\ref{Lem:InegFnConvexes}]
 Without loss of generality, we assume $s>0$ (replacing $s$ with $-s$ and $h$ with $h(-\cdot)$ if necessary).
 By the mean value theorem, we introduce $\tau \in (0,s)$ such that 
 \begin{equation*}
     h'(\tau) = \frac{h(s) - h(0)}{s},
 \end{equation*}
 so that 
 \begin{equation*}
 \int_0^s t \left|h'(t) - \frac{h(s) - h(0)}{s} \right| \, dt 
 =
 \int_0^\tau t \left( \frac{h(s) - h(0)}{s} - h'(t) \right) \, dt 
 +
 \int_\tau^s t \left( h'(t) - \frac{h(s) - h(0)}{s} \right) \, dt =: I_1 + I_2.
 \end{equation*}
 Concerning the second integral, we have $h'(t) \leq h'(s)$ for $\tau \leq t \leq s$, so that
 \begin{equation*}
I_2 \leq \left(h'(s) - \frac{h(s) - h(0)}{s} \right) \int_\tau^s t \, dt 
    =  \frac{(s-\tau)^2}{2} \left(h'(s)- \frac{h(s) - h(0)}{s} \right).
 \end{equation*}
 Concerning the first one, we integrate by parts to obtain
 \begin{equation*}
 I_1
 = \left[ \frac{t^2}{2} \left( \frac{h(s) - h(0)}{s} - h'(t) \right) \right]_0^\tau 
    + \int_0^\tau \frac{t^2}{2} h''(t) \, dt 
 = \int_0^\tau \frac{t^2}{2} h''(t) \, dt .
 \end{equation*}
Since $h$ is convex, it follows that
 \begin{equation*}
 I_1 \leq \frac{\tau}{2} \int_0^s t h''(t) \, dt 
     = \frac{\tau}{2} \left( \left[ th'(t)\right]_0^{s} - \int_0^s h'(t) \, dt \right)
     =  \frac{s\tau}{2} \left(h'(s) - \frac{h(s) - h(0)}{s} \right),
 \end{equation*}
 which, with $s^2 - s \tau + \tau^2 \leq s^2$, concludes the proof of Lemma~\ref{Lem:InegFnConvexes}.
 \end{proof}
 Going back to the proof of Proposition~\ref{Pro:LaxCurves}, we apply Lemma~\ref{Lem:InegFnConvexes} to $h$ 
 defined as a primitive of $s \mapsto \lambda_n(\mathcal{R}_n(s,u_0))$, 
 that is convex due to the monotonicity assumption. 
 We observe that since $\ell_n^0  \cdot r_n (\cdot) =1$, contracting $\mathcal{R}_n(s,u_0)$ in $\mathcal{R}_n(s)$, we have  
 \begin{equation} \label{Eq:ShockSpeedRar}
s \overline{\lambda}_n(u_0,\mathcal{R}_n(s)) 
   = \ell_n^0  \cdot \int_0^s \frac{d}{d\sigma}\Big(f(\mathcal{R}_n(\sigma))\Big) \, d \sigma
   = \ell_n^0  \cdot \int_0^s \lambda_n(\mathcal{R}_n(\sigma)) r_n (\mathcal{R}_n(\sigma)) \, d \sigma
   = \int_0^s \lambda_n(\mathcal{R}_n(\sigma)) \, d \sigma.
\end{equation}
Consequently, \eqref{Eq:DistRS1}-\eqref{Eq:DistRS2} give
\begin{gather}  
\label{Eq:DistRS1b}
    \| \mathcal{R}_n - \mathcal{S}_n \|_{L^\infty([ s,0]) } 
        \leq C |s|^2 \left[ \overline{\lambda}_n(u_0,\mathcal{R}_n(s)) - \lambda_n(\mathcal{R}_n(s)) \right]. \\
\label{Eq:DistRS2b}
    \| \overline{\lambda}_n(u_0,\mathcal{R}_n(\cdot)) - \overline{\lambda}_n(u_0,\mathcal{S}_n(\cdot)) \|_{\infty} 
        \leq C |s| \left[ \overline{\lambda}_n(u_0,\mathcal{R}_n(s)) - \lambda_n(\mathcal{R}_n(s)) \right] .
\end{gather}
 It follows that for $s \in [-s_0,0]$,
 \begin{equation*}
 \overline{\lambda}_n(u_0,\mathcal{S}_n(\cdot)) - \lambda_n(\mathcal{S}_n(s)) 
 \geq
 \overline{\lambda}_n(u_0,\mathcal{R}_n(\cdot)) - \lambda_n(\mathcal{R}_n(s)) 
 + \mathcal{O}(s^2) \left[ \overline{\lambda}_n(u_0,\mathcal{R}_n(\cdot)) - \lambda_n(\mathcal{R}_n(s)) \right],
 \end{equation*}
 where we used monotonicity of the field to deduce that for $s<0$
\begin{equation*}
\frac{\int_0^s \lambda_n(\mathcal{R}_n(\sigma)) \, d \sigma}{s} \geq \lambda_n(\mathcal{R}(s,u_0)).
\end{equation*}
In particular, reducing $s_0$ if necessary, we have
\begin{equation*}
\overline{\lambda}_n(u_0,\mathcal{S}_n(\cdot)) - \lambda_n(\mathcal{S}_n(s)) 
\geq
\frac{1}{2} \left( \overline{\lambda}_n(u_0,\mathcal{R}_n(\cdot)) - \lambda_n(\mathcal{R}_n(s)) \right) \geq 0 ,
\end{equation*}
 the last inequality coming from the monotonicity of the characteristic field. 
 Differentiating the Rankine-Hugoniot relation and applying $\ell_n(\mathcal{S}_n(s,u_0))$, we have
 \begin{multline} \label{Eq:DerivShockSpeed}
\ell_n(\mathcal{S}_n(s,u_0)) \cdot \dot{\mathcal{S}}(s,u_0) \ 
    \Big(\lambda_n(\mathcal{S}_n(s,u_0)) - \overline{\lambda}_n(u_0,\mathcal{S}_n(s,u_0)) \Big) \\
= \partial_s \overline{\lambda}_n(u_0,\mathcal{S}_n(s,u_0)  )
    \ \ell_n(\mathcal{S}_n(s,u_0)) \cdot \Big( \mathcal{S}_n(s,u_0)- u_0 \Big).
\end{multline}
It follows (on some neighborhood of $\overline{u}$) that $s \mapsto \overline{\lambda}_n(u_0,\mathcal{S}_n(s;u_0))$ is non-decreasing on $[-s_0,0]$,
from which we deduce that for all $s \in [-s_0,0]$, the shock $(\mathcal{S}_n(s;u_0),u_0)$
satisfies Liu's $E$-condition.
\end{proof}
\begin{remark}
Notice that by the mean value theorem we have 
\begin{equation*}
|\lambda_n(\mathcal{R}_n(\tau)) - \overline{\lambda}_n(u_0,\mathcal{R}_n(s)) | \leq \| (\nabla \lambda_n \cdot r_n) \circ \mathcal{R}_n(\cdot) \|_{L^1([s,0])},
\end{equation*}
so \eqref{Eq:DistRS1}-\eqref{Eq:DistRS2} gives us
\begin{gather*}
\| \mathcal{R}_n - \mathcal{S}_n \|_{L^\infty([ - s_0,0]) } 
        \leq C s_0^2 \| (\nabla \lambda_n \cdot r_n) \circ \mathcal{R}_n(\cdot)\|_{L^1([ - s_0,0])}, \\
\| \overline{\lambda}_n(u_0,\mathcal{R}_n(\cdot)) - \overline{\lambda}_n(u_0,\mathcal{S}_n(\cdot)) \|_{\infty} 
        \leq C s_0 \| (\nabla \lambda_n \cdot r_n) \circ \mathcal{R}_n(\cdot) \|_{L^1([ - s_0,0])},
\end{gather*}
which is classical in both cases of genuine nonlinearity and linear degeneracy; in the case of an algebraic degeneracy of $\nabla\lambda_n \cdot r_n$ at $u_0$, this was also pointed out by Zumbrun \cite{Zumbrun93}. \par
\end{remark}
%
%
%
%
%
%
%
%
%
\subsection{Proof of sharper interaction estimates: Propositions~\ref{Prop:InteractionDF} and \ref{Prop:InteractionSF}}
\label{Subsec:ProofsShaperEstimates}
\begin{proof}[Proof of Proposition~\ref{Prop:InteractionDF} (Interactions of opposite families)]
The fact that when the $i$-th incoming wave is a rarefaction or a compression wave, 
then $\sigma'_{3-i}=\sigma_{3-i}$ is classical and a straightforward consequence of \eqref{Eq:RarRC}. 
Due to the symmetry between the two characteristic fields, 
it is therefore sufficient to prove \eqref{Eq:InteractionDF} and \eqref{Eq:InteractionDF-CW} on $\sigma_1'$ 
when $(u_l,u_m)$ is a shock wave (and in particular $\sigma_2<0$). \par
\ \par
\noindent
We discuss according to the nature of the wave $(u_m,u_r)$:
\begin{itemize}
\item When $(u_m,u_r)$ is a rarefaction ($\sigma_1>0$) or a compression wave ($\sigma_1<0$), then 
\begin{equation*}
u_r=\mathcal{R}_1(\sigma_1;\mathcal{S}_2(\sigma_2;u_l)) 
        =  \mathcal{S}_2(\sigma'_2;\mathcal{R}_1(\sigma'_1;u_l)).
\end{equation*}
(This writing covers the first statement when $\sigma_1>0$ and the second one when $\sigma_1<0$). \par
Projecting on the $w_2$ axis gives $\sigma_2=\sigma_2'$ while
projecting on the $w_1$ axis, denoting $\widetilde{u}_m=\mathcal{R}_1(\sigma'_1;u_l)$
 and recalling \eqref{Eq:ChocRC}, gives
\begin{eqnarray}
\nonumber
\sigma'_1 &=& \sigma_1 + \sigma_2^3 \mathcal{D}_2(\sigma_2;u_l)  
    - \sigma_2^3 \mathcal{D}_2(\sigma_2;\widetilde{u}_m)  \\
\label{Eq:DiffSigma1}
 &=& \sigma_1 + \sigma_2^3 \mathcal{D}_2(-\sigma_2;u_m) 
    - \sigma_2^3 \mathcal{D}_2(-\sigma_2;u_r)  .
 \end{eqnarray}
Hence \eqref{Eq:InteractionDF} holds (in both cases of a rarefaction and a compression) 
if we set $\mathcal{C}^1=\mathcal{C}^1_{\mathcal{R}}$ where
\begin{equation} \label{Eq:DefC1a}
\mathcal{C}^1_{\mathcal{R}}(u_m;\sigma_1,\sigma_2) 
:= -\frac{ \mathcal{D}_2(-\sigma_2;u_r) - \mathcal{D}_2(-\sigma_2;u_m) }{\sigma_1}
= -\frac{ \mathcal{D}_2(-\sigma_2;\mathcal{R}_1(\sigma_1;u_m)) - \mathcal{D}_2(-\sigma_2;u_m) }{\sigma_1},
\end{equation}
which due to the regularity of $\mathcal{D}_2$ is Lipschitz. \par
%
%
%
%
%
\item The case where  $(u_m,u_r)$ is a shock (and in particular $\sigma_1<0$) is a bit less explicit.
It is a classical and direct consequence  of the implicit function theorem that 
$(u_m;\sigma_1,\sigma_2) \mapsto (\sigma_1',\sigma_2')$ (in terms of Rankine-Hugoniot curves solely),
 is a smooth function.
Since it equals the identity for $\sigma_1=0$ and $\sigma_2=0$, it is of the form 
 $(\sigma_1',\sigma_2') = (\sigma_1,\sigma_2) + \sigma_1\sigma_2(\Phi_1,\Phi_2)$ where $\Phi_1, \Phi_2$ are smooth
 functions of $(\sigma_1,\sigma_2)$.
It remains to prove that $\Phi_i$ can be smoothly factorized by $\sigma_i^2$.
To see that, similarly to \eqref{Eq:DefC1a}, we start from 
 $u_r=\mathcal{S}_1(\sigma_1;\mathcal{S}_2(\sigma_2;u_l)) =  \mathcal{S}_2(\sigma'_2;\mathcal{S}_1(\sigma'_1;u_l))$,
 and denoting here $\widetilde{u}_m=\mathcal{S}_1(\sigma'_1;u_l)$ and projecting on the $w_1$ and $w_2$ axes,
 we find as for \eqref{Eq:DiffSigma1} that
\begin{equation} \label{Eq:S1S2}
\left\{ \begin{array}{l}
\sigma'_1 = \sigma_1 + \sigma_2^3 \mathcal{D}_2(\sigma_2;u_l) - {\sigma'_2}^3 \mathcal{D}_2(\sigma'_2;\widetilde{u}_m) 
= \sigma_1 + \sigma_2^3 \mathcal{D}_2(-\sigma_2;u_m) - {\sigma'_2}^3 \mathcal{D}_2(-\sigma'_2;{u}_r) 
\smallskip \\
\sigma'_2 = \sigma_2 + \sigma_1^3 \mathcal{D}_1(\sigma_1;u_m) - {\sigma'_1}^3 \mathcal{D}_1(\sigma'_1;u_l) .
\end{array} \right. 
\end{equation}
We divide $\sigma_1'-\sigma_1$ by $\sigma_1 \sigma_2^3$ and replace $\sigma_2'$ by $\sigma_2+ \sigma_1\sigma_2\Phi_2$
when considering the last term:
\begin{equation*}
\frac{\sigma_1'-\sigma_1}{\sigma_1 \sigma_2^3} 
= \frac{\mathcal{D}_2(-\sigma_2;u_m) - (1+\sigma_1 \Phi_2)^3 \, \mathcal{D}_2(-\sigma_2(1+\sigma_1 \Phi_2);{u}_r)}
                    {\sigma_1}.
\end{equation*}
We obtain that \eqref{Eq:InteractionDF} holds when setting $\mathcal{C}^1=\mathcal{C}^1_{\mathcal{S}}$ where:
\begin{multline} \label{Eq:DefC1b}
\mathcal{C}_{\mathcal{S}}^1(u_m;\sigma_1,\sigma_2) 
= -\frac{ \mathcal{D}_2(-\sigma_2;u_r) - \mathcal{D}_2(-\sigma_2;u_m) }{\sigma_1} \\
+ \frac{ \mathcal{D}_2(-\sigma_2;u_r) - (1+\sigma_1 \Phi_2)^3 \, \mathcal{D}_2(-\sigma_2(1+\sigma_1 \Phi_2);u_r) }
                                            {\sigma_1}.
\end{multline}
It is a consequence of the regularity of $\Phi_2$ and $\mathcal{D}_2$ that $\mathcal{C}^1_\mathcal{S}$ is Lipschitz.
\end{itemize}
Now we define
\begin{equation*}
\mathcal{C}^1(u_m;\sigma_1,\sigma_2) = \mathcal{C}_{\mathcal{R}}^1(u_m;\sigma_1,\sigma_2) 
        \text{ if } \sigma_1 >0
\ \text{ and } \ 
\mathcal{C}^1(u_m;\sigma_1,\sigma_2) = \mathcal{C}_{\mathcal{S}}^1(u_m;\sigma_1,\sigma_2) 
        \text{ if } \sigma_1 <0.
\end{equation*}

We have to check that these two expressions $\mathcal{C}^1_\mathcal{R}$ and $\mathcal{C}^1_\mathcal{S}$
 are compatible at $\sigma_1=0$, and to explain how this generates the additional error term in \eqref{Eq:InteractionDF-CW}. \par
From $\sigma_1' - \sigma_1 = \mathcal{O}(\sigma_1 \sigma_2)$ and the second equation in \eqref{Eq:S1S2},
we infer that $\sigma_2' - \sigma_2 = \mathcal{O}(\sigma_1^3 \sigma_2)$, that is $\Phi_2 = \mathcal{O}(\sigma_1^2)$.
Together with the smoothness of $\mathcal{D}_2$, this proves that the second term in \eqref{Eq:DefC1b}
 is of order $\mathcal{O}(\sigma_1^2)$. \par
In particular, for $\sigma_1<0$, we have 
$\mathcal{C}^1(u_m;\sigma_1,\sigma_2) - \mathcal{C}_{\mathcal{R}}^1(u_m;\sigma_1,\sigma_2) = \mathcal{O}(\sigma_1^2)$,
which justifies the additional error in \eqref{Eq:InteractionDF-CW}.
Moreover, this proves that $\mathcal{C}^1$ is continuous at $\sigma_1=0$.
Since it is regular on each side $\sigma_1<0$ and $\sigma_1>0$, this proves that $\mathcal{C}^1$ is 
globally Lipschitz as claimed.
\end{proof}
\ \par
\begin{proof}[Proof of Proposition~\ref{Prop:InteractionSF} (Interactions within a family)]
We only treat interactions of the first family by symmetry. \par
\noindent
{\bf 1.} We consider the case of classical waves.
In that case, the second estimate on $\widetilde{\sigma}_2$ is classical and not peculiar to $2 \times 2$ systems,
see e.g. \cite[Lemma 7.2]{BressanBook00}.
Concerning the first inequality, we project 
$u_r= \mathcal{T}_2(\widetilde{\sigma}_2;\mathcal{T}_1(\widetilde{\sigma}_1;u_l)) 
    = \mathcal{T}_1(\sigma'_1;\mathcal{T}_1(\sigma_1;u_l))$ on the $w_1$-axis to obtain:
\begin{equation*}
w_1[u_r] - w_1[\mathcal{T}_1(\widetilde{\sigma}_1;u_l)] + \widetilde{\sigma}_1 = \sigma'_1 + \sigma_1.
\end{equation*}
Since the wave between $\mathcal{T}_1(\widetilde{\sigma}_1;u_l)$ and $u_r$ is of the second family 
and due to the second-order tangency of rarefaction curves and Lax wave curves, we deduce
\begin{equation*}
w_1[u_r] - w_1[\mathcal{T}_1(\widetilde{\sigma}_1;u_l)] = \mathcal{O}(\widetilde{\sigma}_2^3),
\end{equation*}
which proves the claim. \par
\ \par
\noindent
{\bf 2.}
In the case of a compression wave meeting a shock, we treat only interactions in family one 
 with the compression wave on the right, the other cases being similar.
We reason as previously. 
We use the notation $\check{u}_m:=\mathcal{S}_1(\sigma_1;u_l)$.
First, by the classical Glimm estimates, we have 
 $\widetilde{\sigma}_1 = \sigma_1 + \sigma'_1 + \mathcal{O}(\sigma_1 \sigma'_1)$.
In particular, we see that $\widetilde{\sigma}_1 -\sigma_1 = \mathcal{O}(\sigma'_1)$
and that $\sigma_1 = \mathcal{O}(\widetilde{\sigma}_1)$, since $\sigma_1$ and $\sigma'_1$ have the same sign. \par
Now we project
$u_r= \mathcal{T}_2(\widetilde{\sigma}_2;\mathcal{S}_1(\widetilde{\sigma}_1;u_l)) 
    = \mathcal{R}_1(\sigma'_1;\mathcal{S}_1(\sigma_1;u_l))$ 
on the $w_1$/$w_2$-axes and use \eqref{Eq:ChocRC} to obtain:
\begin{equation} \label{Eq:S1S2SFCMS}
(\widetilde{\sigma}_2)_*^3 \, \mathcal{D}_2(\check{u}_m, {\sigma}_2) + \widetilde{\sigma}_1
= \sigma'_1 + \sigma_1
\ \text{ and } \ 
\widetilde{\sigma}_2 + \widetilde{\sigma}_1^3 \, \mathcal{D}_1(\widetilde{\sigma}_1;u_l) 
= \sigma_1^3 \, \mathcal{D}_1(\sigma_1;u_l) . 
\end{equation}
We infer from the second equation and the regularity of $\mathcal{D}_1$ that
 $\widetilde{\sigma}_2 = \mathcal{O}(1) |\widetilde{\sigma}_1|^2 |\widetilde{\sigma}_1 - \sigma_1| $,
 so that 
 $\widetilde{\sigma}_2 = \mathcal{O}(1) |\widetilde{\sigma}_1|^2 |\sigma'_1| $, which gives the second part of
 \eqref{Eq:InteractionsSF-CWvS}.
Injecting in the first equation of \eqref{Eq:S1S2SFCMS}, we get the remaining part. \par
\ \par
\noindent
{\bf 3.} Finally, the case of two compression/rarefaction waves is a direct consequence of \eqref{Eq:RarRC}. 
\end{proof}
%
%
%
%
%
%
\subsection{Estimate on the modified $p$-variation: proof of Lemma~\ref{Lem:VVtilde}}
\label{Subsec:ProofVVtilde}
Call $\gamma^1_1,\ldots,\gamma^1_{n_1}$ the list of all $1$-fronts at time $t$, 
 and $\gamma^2_1,\ldots,\gamma^2_{n_2}$ the list of all $2$-fronts at time $t$, from left to right.
Now to get the $p$-variations $V_p(w_1(u^\nu))$ and $V_p(w_2(u^\nu))$ from $\widetilde{V}_p(u^\nu)$,
 we need to include the jumps of $w_1$ at fronts $\gamma^2_k$ (that we will denote $[w_1]_{\gamma^2_k}$)
 in $s_p(\sigma_{\gamma^1_1},\ldots,\sigma_{\gamma^1_{n_1}})$,  
 and the jumps of $w_2$ at fronts $\gamma^1_i$ (that we will denote $[w_2]_{\gamma^1_i}$)
 in $s_p(\sigma_{\gamma^2_1},\ldots,\sigma_{\gamma^2_{n_2}})$,  respectively.
Due to \eqref{Eq:RarRC}-\eqref{Eq:ChocRC},
 for any $k= 1,\ldots,{n_2}$ (resp. $i=1,\ldots,{n_1}$), one has $|[w_1]_{\gamma^2_k}| \lesssim |\sigma_{\gamma^2_k}|^3$
 (resp. $|[w_2]_{\gamma^1_i}| \lesssim |\sigma_{\gamma^1_i}|^3$).
It follows that
\begin{equation} \label{Eq:SwpV}
s_p([w_2]_{\gamma^1_1},\ldots,[w_2]_{\gamma^1_{n_1}}) 
\lesssim s_1(|\sigma_{\gamma^1_1}|^3,\ldots,|\sigma_{\gamma^1_{n_1}}|^3) 
\leq \sum_{\gamma \in \mathcal{A}(t)} |\sigma_\gamma|^3
\leq V_3(u^\nu(t,\cdot))^3 \leq V_p(u^\nu(t,\cdot))^3 .
\end{equation}
Using Proposition~\ref{Pro:ElemePropsVp}--\ref{Item:Interleave}, we have
\begin{equation*}
s_p([w_2]_{\gamma}; \gamma \in \mathcal{A}(t)) 
\leq s_p(\sigma_{\gamma^2_1},\ldots,\sigma_{\gamma^2_{n_2}})
+ s_p([w_2]_{\gamma^1_1},\ldots,[w_2]_{\gamma^1_{n_1}}) 
\leq s_p(\sigma_{\gamma^2_1},\ldots,\sigma_{\gamma^2_{n_2}})
+ \mathcal{O}(1) V_p(u^\nu(t,\cdot))^3,
\end{equation*}
and all the same, 
\begin{equation*}
s_p([w_1]_{\gamma}; \gamma \in \mathcal{A}(t)) 
\leq s_p(\sigma_{\gamma^1_1},\ldots,\sigma_{\gamma^1_{n_1}})
+ \mathcal{O}(1) V_p(u^\nu(t,\cdot))^3.
\end{equation*}
Using the triangle inequality for the $p$-norm in $\R^2$, we find
\begin{equation*}
V_p(u^\nu(t,\cdot)) \leq \widetilde{V}_p(u^\nu(t,\cdot)) + \mathcal{O}(1)V_p(u^\nu(t,\cdot))^3.
\end{equation*}
Proposition~\ref{Pro:ElemePropsVp}--\ref{Item:Interleave} also gives us
\begin{equation*}
s_p(\sigma_{\gamma^2_1},\ldots,\sigma_{\gamma^2_{n_2}})
\leq s_p([w_2]_{\gamma}; \gamma \in \mathcal{A}(t))  
+ \mathcal{O}(1) V_p(u^\nu(t,\cdot))^3,
\end{equation*}
so by the same reasoning we find 
\begin{equation*}
\widetilde{V}_p(u^\nu(t,\cdot)) \leq V_p(u^\nu(t,\cdot)) + \mathcal{O}(1)V_p(u^\nu(t,\cdot))^3.
\end{equation*}
Now instead of \eqref{Eq:SwpV} we can use
\begin{equation*} 
s_p([w_2]_{\gamma^1_1},\ldots,[w_2]_{\gamma^1_{n_1}}) 
\lesssim s_1(|\sigma_{\gamma^1_1}|^3,\ldots,|\sigma_{\gamma^1_{n_1}}|^3) 
\lesssim s_3(\sigma_{\gamma^1_1},\ldots,\sigma_{\gamma^1_{n_1}})^3 
\leq \widetilde{V}_p(u^\nu(t,\cdot))^3.
\end{equation*}
So following the same lines we find
\begin{equation*}
V_p(u^\nu(t,\cdot)) = \widetilde{V}_p(u^\nu(t,\cdot)) + \mathcal{O}(1) \widetilde{V}_p(u^\nu(t,\cdot))^3.
\end{equation*}
This concludes the proof of Lemma~\ref{Lem:VVtilde}.
%
%
%
%
%
\ \par
\noindent
\bibliography{biblio-sclvp}

\begin{thebibliography}{10}

\bibitem{Bianchini-Riemann03}
S.~Bianchini.
\newblock On the {R}iemann problem for non-conservative hyperbolic systems.
\newblock {\em Arch. Ration. Mech. Anal.}, 166(1):1--26, 2003.

\bibitem{BianchiniColomboMonti10}
S.~Bianchini, R.~M. Colombo, and F.~Monti.
\newblock {$2\times 2$} systems of conservation laws with {$\mathbf L^\infty$}
  data.
\newblock {\em J. Differential Equations}, 249(12):3466--3488, 2010.

\bibitem{BourdariasChoudhuryGuelmameJunca22}
C.~Bourdarias, A.~P. Choudhury, B.~Guelmame, and S.~Junca.
\newblock Entropy solutions in {$BV^s$} for a class of triangular systems
  involving a transport equation.
\newblock {\em SIAM J. Math. Anal.}, 54(1):791--817, 2022.

\bibitem{BourdariasGisclonJunca14}
C.~Bourdarias, M.~Gisclon, and S.~Junca.
\newblock Fractional {$BV$} spaces and applications to scalar conservation
  laws.
\newblock {\em J. Hyperbolic Differ. Equ.}, 11(4):655--677, 2014.

\bibitem{MR3538367}
C.~Bourdarias, M.~Gisclon, S.~Junca, and Y.-J. Peng.
\newblock Eulerian and {L}agrangian formulations in {$BV^S$} for gas-solid
  chromatography.
\newblock {\em Commun. Math. Sci.}, 14(6):1665--1685, 2016.

\bibitem{BressanBook00}
A.~Bressan.
\newblock {\em Hyperbolic systems of conservation laws. The one-dimensional
  Cauchy problem.}, volume~20 of {\em Oxford Lecture Series in Mathematics and
  its Applications}.
\newblock Oxford University Press, Oxford, 2000.

\bibitem{BressanColombo95}
A.~Bressan and R.~M. Colombo.
\newblock The semigroup generated by {$2\times 2$} conservation laws.
\newblock {\em Arch. Rational Mech. Anal.}, 133(1):1--75, 1995.

\bibitem{Cheverry98}
C.~Cheverry.
\newblock Syst{\`e}mes de lois de conservation et stabilit\'{e} {BV}.
\newblock {\em M\'{e}m. Soc. Math. Fr. (N.S.)}, (75), 1998.

\bibitem{DafermosBook10}
C.~M. Dafermos.
\newblock {\em Hyperbolic conservation laws in continuum physics}, volume 325
  of {\em Grundlehren der mathematischen Wissenschaften}.
\newblock Springer-Verlag, Berlin, second edition, 2005.

\bibitem{DudleyNorvaivsa11}
R.~M. Dudley and R.~Norvai\v{s}a.
\newblock {\em Concrete functional calculus}.
\newblock Springer Monographs in Mathematics. Springer, New York, 2011.

\bibitem{MR4130242}
S.~S. Ghoshal, B.~Guelmame, A.~Jana, and S.~Junca.
\newblock Optimal regularity for all time for entropy solutions of conservation
  laws in {$BV^s$}.
\newblock {\em NoDEA Nonlinear Differential Equations Appl.}, 27(5):46, 2020.

\bibitem{Glimm65}
J.~Glimm.
\newblock Solutions in the large for nonlinear hyperbolic systems of equations.
\newblock {\em Comm. Pure Appl. Math.}, 18:697--715, 1965.

\bibitem{GlimmLax70}
J.~Glimm and P.~D. Lax.
\newblock {\em Decay of solutions of systems of nonlinear hyperbolic
  conservation laws}.
\newblock Memoirs of the American Mathematical Society, No. 101. American
  Mathematical Society, Providence, RI, 1970.

\bibitem{MR4069619}
B.~Guelmame, S.~Junca, and D.~Clamond.
\newblock Regularizing effect for conservation laws with a {L}ipschitz convex
  flux.
\newblock {\em Commun. Math. Sci.}, 17(8):2223--2238, 2019.

\bibitem{JenssenRidderPhi20}
H.~K. Jenssen and J.~Ridder.
\newblock On {$\Phi$}-variation for 1-d scalar conservation laws.
\newblock {\em J. Hyperbolic Differ. Equ.}, 17(4):843--861, 2020.

\bibitem{LiuWT77}
T.~P. Liu.
\newblock The deterministic version of the {G}limm scheme.
\newblock {\em Comm. Math. Phys.}, 57(2):135--148, 1977.

\bibitem{LiuYang02}
T.-P. Liu and T.~Yang.
\newblock Weak solutions of general systems of hyperbolic conservation laws.
\newblock {\em Comm. Math. Phys.}, 230(2):289--327, 2002.

\bibitem{Marconi2018}
E.~Marconi.
\newblock {Regularity estimates for scalar conservation laws in one space
  dimension}.
\newblock {\em J. Hyper. Differential Equations}, 15(04):623--691, Dec. 2018.

\bibitem{MusielakOrlicz59}
J.~Musielak and W.~Orlicz.
\newblock On generalized variations. {I}.
\newblock {\em Studia Math.}, 18:11--41, 1959.

\bibitem{Schochet1991Feb}
S.~Schochet.
\newblock {Sufficient conditions for local existence via Glimm's scheme for
  large BV data}.
\newblock {\em J. Differential Equations}, 89(2):317--354, 1991.

\bibitem{LCYoung36}
L.~C. Young.
\newblock An inequality of the {H}\"{o}lder type, connected with {S}tieltjes
  integration.
\newblock {\em Acta Math.}, 67(1):251--282, 1936.

\bibitem{RYoung99}
R.~Young.
\newblock Exact solutions to degenerate conservation laws.
\newblock {\em SIAM J. Math. Anal.}, 30(3):537--558, 1999.

\bibitem{Zumbrun93}
K.~Zumbrun.
\newblock Decay rates for nonconvex systems of conservation laws.
\newblock {\em Comm. Pure Appl. Math.}, 46(3):353--386, 1993.

\end{thebibliography}
\bibliographystyle{abbrv} 
\addcontentsline{toc}{section}{References}
\end{document}